\documentclass{gtpart}    
\usepackage{pinlabel}
\usepackage[all]{xy}
 \usepackage{enumitem}
\usepackage{amsfonts}
\usepackage{amsmath, amscd, amssymb, yhmath, color}
\usepackage{latexsym}
\usepackage{mathrsfs}
\usepackage{xy}
\input xy
\usepackage{url}
\xyoption{all}  

\def \Mor{\mathop{\mathrm{Mor}}}
\def \Ob{\mathop\mathrm{Ob}}
\def \R{\mathbb{R}}
\def \Z{\mathbb{Z}}

\def \Q{\mathbb{Q}}

\def \M{\mathcal{M}}

\def \J{\mathcal{J}}

\def \MO{\mathbf{MO}}

\def \BO{\mathop{\mathrm{BO}}}

\def \id{\mathrm{id}}
\def \Diff{\mathop{\mathrm{Diff}}}
\def \BDiff{\mathop{\mathrm{BDiff}}}

\def \B{\mathbf{B}}

\def \r{\mathcal{R}}

\def \co{\colon\thinspace}
\def \ccc{}

\title{The bordism version of the h-principle}
\author{Rustam Sadykov}
\givenname{Rustam}
\surname{Sadykov}
\address{Department of Mathematics\\ Kansas State University}
\email{rstsdk@gmail.com}


\newtheorem{theorem}{Theorem}[section]
\newtheorem{lemma}[theorem]{Lemma}
\newtheorem{warn}[theorem]{Warning}

\newtheorem{corollary}[theorem]{Corollary}
\newtheorem{proposition}[theorem]{Proposition}
\newtheorem*{B-principle}{B-principle}
\newtheorem*{H-principle}{H-principle}

\theoremstyle{remark}
\newtheorem{remark}[theorem]{Remark}

\theoremstyle{definition}
\newtheorem{definition}[theorem]{Definition}

\newtheorem{example}[theorem]{Example}

\let\OLDthebibliography\thebibliography
\renewcommand\thebibliography[1]{
  \OLDthebibliography{#1}
  \setlength{\parskip}{0pt}
  \setlength{\itemsep}{1pt}
}

\begin{document}
\begin{abstract} 
In view of the Segal construction each category with a coherent operation gives rise to a cohomology theory. Similarly each open stable differential relation $\mathcal{R}$ imposed on smooth maps of manifolds determines cohomology theories $k^*_{\mathcal{R}}$ and $h^*_{\mathcal{R}}$; the cohomology theory $k^*_\r$ describes invariants of solutions of $\r$, while $h^*_\r$ describes invariants of so-called stable formal solutions of $\r$.  We  prove 
the bordism version of the h-principle: The cohomology theories $k^*_{\mathcal{R}}$ and $h^*_{\mathcal{R}}$ are equivalent for a fairly arbitrary open stable differential relation $\r$. Furthermore, we 
determine the homotopy type of $h^*_\r$. Thus, we show that for a fairly arbitrary open stable differential relation $\r$, the machinery of stable homotopy theory can be applied to perform explicit computations and determine invariants of solutions. 

In the case of the differential relation whose solutions are all maps, our construction amounts to the Pontrjagin-Thom construction. In the case of the covering differential relation our result is equivalent to the Barratt-Priddy-Quillen theorem asserting that the direct limit of classifying spaces $B\Sigma_n$ of permutation groups $\Sigma_n$ of finite sets of $n$ elements is homology equivalent to each path component of the infinite loop space $\Omega^{\infty}S^{\infty}$.
In the case of the submersion differential relation imposed on maps of dimension $d=2$ the cohomology theories $k^*_{\mathcal{R}}$ and $h^*_{\mathcal{R}}$ are not equivalent. Nevertheless, our methods still apply and can be used to recover~\cite{Sa5} the Madsen-Weiss theorem (the Mumford Conjecture). 
\end{abstract}
 
\maketitle


\date{\today}



\section{Introduction}
The \emph{h-principle} is a general observation that differential geometry problems can often be reduced to problems in (\emph{unstable}) homotopy theory. For example, let us consider three fairly different problems in differential geometry on the existence of structures on a given open smooth manifold $M$.

\begin{description}
\item[\bf Problem 1] Does $M$ admit a symplectic/contact structure? 
\item[\bf Problem 2] Does $M$ admit a foliation of a given codimension? 
\item[\bf Problem 3] Does $M$ admit a Riemannian metric of positive/negative scalar curvature? 
\end{description}

According to a very general theorem of Gromov~\cite{Gro} each of the mentioned differential geometry problems reduces to a  homotopy theory problem (which can be approached by standard homotopy theory methods). 

Similarly, the \emph{b-principle} is a general observation that differential geometry problems can often be reduced to problems in \emph{stable} homotopy theory. The known instances of the b-principle type phenomena include

\begin{description}
\item[\bf Example 1] the Barratt-Priddy-Quillen theorem,
\item[\bf Example 2] the Audin-Eliashberg theorem on Legendrian immersions,
\item[\bf Example 3] the Mumford conjecture on moduli spaces of Riemann surfaces. 
\end{description}

We propose a general setting for the b-principle, and show that the b-principle phenomenon occurs for a fairly arbitrary (open stable) differential relation $\mathcal{R}$ imposed on smooth maps, and therefore the machinery of stable homotopy theory can be applied to 
find invariants of solutions of the differential relation $\mathcal R$. Systematic explicit computations will be carried out in subsequent papers, e.g., see \cite{Sa2}.

{\ccc
The precise definition of an open stable differential relation can be found in \S\ref{a:1}. Loosely speaking, a coordinate invariant differential relation is \emph{open stable}  if its solution space is  open in the space of smooth maps and it satisfies the pullback property: given a pair of transverse maps $f$ and $g$ into the same target manifold, if $f$ is a solution, then $g^*f$ is also a solution.

}


\begin{figure}[ht]
\centering
\includegraphics[height=1.9in]{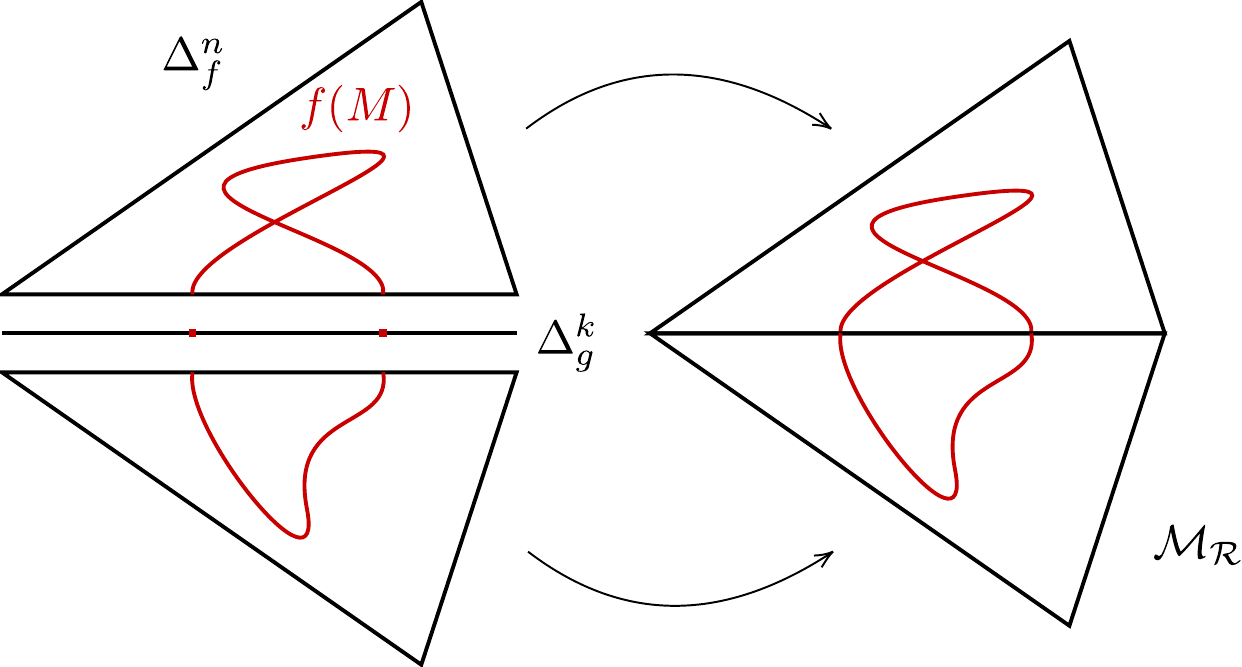}
\caption{Constructing $\M_\r$ for the differential relation of immersions of codimension $1$.}
\label{fig:13a}
\end{figure}

The proposed general b-principle is formulated in terms of the \emph{moduli space} $\M_\r$ of  solutions (cf. Mumford \cite{Mu}), which is an appropriate quotient space of all proper solutions of 
$\mathcal{R}$. 
We will work with a simplicial version of $\M_\r$. Its (informal) construction is fairly simple. Take a copy $\Delta^n_f$ of the standard simplex $\Delta^n$ of dimension $n$, one for each proper solution $f\co M\to \Delta^n$ transverse to each face $\Delta^k\subset \Delta^n$. Such a solution $f$ restricts to a solution over each face $\Delta^k\subset \Delta^n$. If a solution $f$ over $\Delta^n$ restricts to a solution $g$ over a face $\Delta^k$, then identify $\Delta^k_g$ with the corresponding face of $\Delta^n_f$. The obtained space $\mathcal{M}_\r=(\sqcup \Delta^n_f)/_\sim$ is the moduli space of solutions, see Figure~\ref{fig:13a}.

The moduli space $\M_\r$ of solutions satisfies a universal homotopy property (\S\ref{ss:5.1}), which makes it indispensable for computing invariants of solutions of $\r$.  Namely, given a proper solution $g\co M\to N$, there exists a smooth triangulation of $N$ such that $g$ is transverse to each simplex $\Delta$ of the triangulation. Choose an order on the set of vertices of the triangulation. Then each simplex $\Delta\subset N$ together with $g|g^{-1}\Delta$ has a unique counterpart in $\M_\r$ and therefore there is a classifying map $N\to \M_\r$. Its homotopy class depends only on $g$. 
The universal homotopy property immediately implies that invariants of $\M_\r$ correspond to  invariants of solutions of $\mathcal{R}$, for details see \S\ref{s:inv}. 

In some cases invariants of solutions of $\r$ are fairly complicated and a direct computation of these invariants seems improbable. On the other hand, it turns out that invariants of $\M_\r$ can be computed  (indirectly) by means of stable homotopy theory. Indeed, we will see that the moduli space of solutions $\M_\r$  of an open stable differential relation $\r$ possesses a rich structure of an \emph{H-space with a coherent operation} (i.e., a $\Gamma$-space; see~\S\ref{ss:2.2}). In particular, the group 
completion of $\M_\r$ is an infinite loop space classifying a cohomology theory $k^*_\r$. In order to identify $k^*_\r$, we also consider \emph{stable formal solutions} of $\r$---for a definition, see \S\ref{s:3}---and their moduli space $h\M_\r$; the moduli space $h\M_\r$ of stable formal solutions is produced, mutatis mutandis, by means of the mentioned construction of $\M_\r$. We show that $h\M_\r$  is itself an infinite loop space classifying a cohomology theory $h^*_\r$. 

\begin{example}[Submersion differential relation]\label{ex:1}
The moduli space for submersions of dimension $d$ is the disjoint union $\sqcup \BDiff M$ of classifying spaces of diffeomorphism groups of closed (possibly non-connected) manifolds of dimension $d$; there is one space $\BDiff M$ in the disjoint union for each diffeomorphism type.  The space $\sqcup \BDiff M$ is an H-space with a
coherent operation inherited from the obvious homomorphisms 
\[
\Diff M\times \Diff N\longrightarrow \Diff (M\sqcup N),
\]
\[
  \alpha\times \beta\mapsto \alpha\sqcup \beta 
\]
of diffeomorphism groups of manifolds $M$, $N$ and $M\sqcup N$. Hence, by the Segal construction~\cite{Se}, the space $B_1=\sqcup \BDiff M$ defines a spectrum $\{B_i\}$ such that each of the loop spaces $\Omega B_2$, $\Omega^2B_3$, ... is a group completion of $B_1$, see \S\ref{s:1.1}. 
The cohomology theory $k^*$ of solutions of the submersion differential relation $\mathcal R$ is associated with the infinite loop space of the spectrum $\{ B_i\}$. 

The cohomology theory $h^*$ of stable formal solutions of $\mathcal{R}$ is associated with the connective spectrum of the tangential structure $\BO_d\subset \BO$ of dimension $d$ (see \S\ref{a:3}). In particular, for every manifold $N$, the group $h^0(N)$ is the cobordism group of maps $f\co M\to N$ of dimension $d$ of smooth manifolds equipped with a representation of the stable vector bundle $TM\ominus f^*TN$ by a vector bundle over $M$ of dimension $d$, where $TM$ and $TN$ are the tangent bundles of $M$ and $N$ respectively, see \S\ref{a:3}. 
\end{example}

{\ccc
\begin{example}  Let $\r$ be an open stable differential relation whose solutions are maps of dimension $d$ with only singularities \cite{AVG} from a finite list $\tau$ of simple singularities of smooth maps of dimension $d$. Then the moduli space of solutions is a union 
of the classifying spaces $\BDiff F_\alpha$ of diffeomorphism groups of fibers $F_\alpha$ of maps with only $\tau$-singularities. 

On the other hand, by Theorem~\ref{th:14.5},  the space $h\M_\r$ is homotopy equivalent to the infinite loop space of a spectrum associated with a stable vector bundle $f\co B_\r\to \BO$ of dimension $-d$, see \S\ref{a:3}. It follows (see \cite{Kaz3}, \cite{Va}, and also \cite{Sz}, \cite{Sa6} and \cite{Sa2}) that the space $B_\r$ is a  union  of the classifying spaces $\BDiff \alpha$ of symmetries of singularity types $\alpha\in \tau$. 
\end{example}
}

{\ccc To summarize},
every open stable differential relation $\mathcal{R}$ (see \S \ref{a:1}) imposed on maps of a fixed dimension $d$ determines cohomology theories ${k}^*_{\mathcal{R}}$ of solutions of $\mathcal{R}$ and $h_{\mathcal{R}}^*$ of stable formal solutions of $\mathcal{R}$. Furthermore, there is a natural transformation $\mathfrak{A}_{\mathcal{R}}\co k^*_{\mathcal R}\to h^*_{\mathcal R}$
of cohomology theories, see \S\ref{ss:2.3}.

\begin{example}[Submersion differential relation]\label{ex:1.2aa}  Let $\r$ denote the differential relation whose solutions are submersions of dimension $d$. Then, for a manifold $N$, the homomorphism $k^0_\r(N)\to h_\r^0(N)$ is essentially  defined by associating with a class of a proper submersion
$f\co M\to N$ the class of $f$ equipped with the representation of the stable vector bundle $TM\ominus f^*TN$ by the kernel bundle of the differential $df$,  see Remark~\ref{r:14.4} and \cite[Proposition 4.1]{Se}.  
We will see that elements of $k^1_\r(N)$ are represented by pairs $(f, \alpha)$ of a submersion $f\co M\to N$ and a function $\alpha\co M\to \R$ such that the map 
\[
(f, \alpha)\co M\longrightarrow N\times \R
\] is proper and $\alpha(f^{-1}(x))\ne \R$ for each $x\in N$, see Figure~\ref{fig:1a}.  Elements of $h^1_\r(N)$ are defined similarly, except that now $f$ is an arbitrary function (not necessarily a submersion) equipped with a representation of the stable vector bundle $TM\ominus f^*TN$ by a vector bundle over $M$ of dimension $d$. Again, the natural isomorphisms $k^1_\r(N)\to h^1_\r(N)$ are essentially constructed by  representing $TM\ominus f^*TN$ by $\mathop\mathrm{ker } df$. 
For $i>1$, elements of $k^i_\r(N)$ and $h^i_\r(N)$ and natural homomorphisms $k^i_\r(N)\to h^i_\r(N)$  are defined similarly (cf. Hatcher~\cite{Ha} and Galatius--Randal-Williams~\cite{GR}).   
\end{example}

\begin{remark} \label{ex:1.3aa} Though historically in many applications only the functors $k^0_\r$ and $h^0_\r$ played an important role, it turned out that $k^i_\r$ and $h^i_\r$ for $i\ne 0$ are important in their own right. For example, counterparts of higher order functors $k^i_\r$ and $h^i_\r$  appear in the classification of TQFTs by Lurie (see \cite[section 2.2]{Lu}, especially Definition~2.2.9).  It is also worth mentioning that the role of functors $k^1_\r$ and $h^1_\r$ in the proof of Theorem~\ref{th:1} (as well as, for example, in the proof of the main theorem in \cite{GMTW}) is crucial. 
\end{remark}

{\ccc
\begin{definition}\label{d:b-principle}
An open stable differential relation $\r$ is said to satisfy the \emph{b-principle} if the natural transformation $\mathfrak{A}_\r$ is an equivalence\footnote{This is a refined version of the b-principle introduced in the paper~\cite{Sa}. 
The notion of the b-principle in the current paper is essentially stronger than that in \cite{Sa}.}. 
\end{definition}


When $\r$ satisfies the b-principle, we are in position to apply the Group Completion Theorem, e.g., see \cite{MS}.
We identify the set $\pi_0=\pi_0\M_\r$ of path components of $\M_\r$ with the corresponding subset of the Pontrjagin ring $H_*(\M_\r; \Z)$.  Then $\pi_0$ belongs to the center of $H_*(\M_\r)$. 

\begin{proposition}\label{th:1.10} Suppose that an open stable differential relation $\r$ satisfies the b-principle. Then 
$H_*(\M_\r)[\pi_0^{-1}] \simeq H_*(h\M_\r). $
\end{proposition}

\begin{example} If $d<0$, then the space $\M_\r$ is path connected. In general,  $\pi_0$ is the monoid of equivalence classes of manifolds of dimension $d$ with operation given by taking the disjoint union of manifolds. Two manifolds of dimension $d$ represent the same element in $\pi_0$, if they appear as fibers of a proper function satisfying $\r$. In particular, if all Morse functions on manifolds of dimension $d+1$ are solutions of $\r$, then $\pi_0$ is a group. To the author's knowledge, the monoids $\pi_0\M_\r$ are not known in the case of differential relations $\r$ of Morse functions with critical points of prescribed indices.   
\end{example}

}

The b-principle is a version of the \emph{h-principle}, which is a general homotopy theoretic approach to solving partial differential relations (see the foundational book by Gromov~\cite{Gro}, and a more recent book by Eliashberg and Mishachev~\cite{EM}).

{\ccc
\begin{remark}[B-principle vs h-principle] Every open coordinate-invariant differential relation $\mathcal{R}$ imposed on maps $M\to N$ of given smooth manifolds gives rise to spaces 
$\mathbf{Sol}$ of solutions, and 
$\mathbf{hSol}$ of formal solutions. 
The h-principle for $\mathcal{R}$ asserts that a canonical map  
$\mathbf{Sol} \longrightarrow \mathbf{hSol}$
is a weak homotopy equivalence of spaces.  

Similarly,  every open stable differential relation $\mathcal{R}$ imposed on maps of dimension $d$ gives rise to a cohomology theories $k_\r$ of solutions and $h_\r$ of formal solutions. 

In general, invariants of solutions to differential relations computed by means of
the h-principle reduction are stronger than their b-principle counterparts. On the other hand, the h-principle
invariants are often considerably harder to compute, than the invariants obtained by the b-principle reduction.
\end{remark}
}

 The following example shows that the b-principle for $\r$ may hold true even if the h-principle fails to be true.  

\begin{example} In the case $d=0$, the topological monoid $B_1$ of Example~\ref{ex:1} is homotopy equivalent to the disjoint union  $\sqcup B\Sigma_n$ of classifying spaces of permutation groups $\Sigma_n$ of finite sets of $n\ge 0$ elements, while $h^*$ is associated with the infinite loop space $\Omega^{\infty}S^{\infty}$ \cite{Fu}. In this case the b-principle assertion is equivalent to the Barratt-Priddy-Quillen theorem.   

\begin{theorem}[Barratt-Priddy-Quillen, \cite{BP}]\label{th:bpq} The group completion of the monoid $\sqcup B\Sigma_n$ is weakly homotopy equivalent to $\Omega^{\infty}S^{\infty}$. 
\end{theorem}

\end{example}

In general, it is difficult to solve differential relations. On the other hand, our main result---Theorem~\ref{th:1} below---guarantees that often the cohomology theory of solutions is equivalent to the cohomology theory of stable formal solutions. The latter can be studied by means of standard methods of stable homotopy theory~\cite{Au1}, \cite{Au2}, \cite{Au3}, \cite{Ga}, \cite{RW}, \cite{Sa2}, \cite{Sz}, \cite{Va}. In fact we determine (\S\ref{s:11}) the homotopy type of the infinite loop space classifying $h^*_\r$ for any open stable differential relation $\r$. 
To demonstrate the utility of the b-principle approach,  we carry out in \S\ref{s:8} a short computation (which originally is due to Wells \cite{We}) in the case of immersions. Further (systematic) computations will appear elsewhere, e.g., see \cite{Sa2}. 

\begin{theorem}\label{th:1} The b-principle {\ccc satisfies} every open stable differential relation $\mathcal{R}$ imposed on maps of dimension $d<1$, and for every open stable differential relation $\mathcal{R}$ imposed on maps of dimension $d\ge 1$ provided that each Morse function on a manifold of dimension $d+1\ge 2$ is a solution of the relation $\mathcal{R}$. 
\end{theorem}


\begin{remark} The assumption on Morse functions in the statement of Theorem~\ref{th:1} is in a sense mild. Indeed, every function on a manifold can be approximated by $C^{\infty}$-close Morse functions. On the other hand, the subspace of solutions of an arbitrary open stable differential relation $\mathcal R$ imposed on maps of dimension $d$  is open in the space of $C^{\infty}$-smooth functions, i.e., every function on a manifold of dimension $d+1$ sufficiently $C^{\infty}$-close to a solution of $\mathcal R$ must also be a solution of $\mathcal R$ (see Definition~\ref{d:9a}). 
\end{remark}

\begin{remark} The natural transformation $\mathfrak{A}_{\mathcal{R}}$ for the submersion differential relation $\mathcal R$ imposed on maps of dimension $d\ge 1$ is not an equivalence (see Remark~\ref{r:4.3}). In this case $\mathfrak{A}_{\mathcal R}$ is closely related to the universal Madsen-Tillmann map \cite{GMTW}. The case $d=2$ is of particular interest \cite{EGM}, \cite{GMTW}, \cite{GR}, \cite{MW}, \cite{Ha} as it is related to the Madsen-Weiss theorem~\cite{MW} (the Mumford Conjecture~\cite{Mu1}), which asserts that the cohomology groups of the stable moduli space of Riemann surfaces are isomorphic to those of the Madsen-Tillmann spectrum; in Example~\ref{e:3} we show that in the non-negative degrees the cohomology theory classified by the Madsen-Tillmann spectrum coincides with the cohomology theory $h^*_{\mathcal R}$ of the submersion differential relation imposed on oriented maps of dimension $d$. In a subsequent paper \cite{Sa5} we explain the relation of the b-principle to the Mumford conjecture, and recover the latter.  
\end{remark}

Theorem~\ref{th:1} is a generalization of the mentioned Barratt-Priddy-Quillen theorem (especially see the interpretation by Fuks~\cite{Fu}), the Wells theorem for immersions~\cite{We}, Eliashberg theorems for $k$-mersions, Audin-Eliashberg theorems for Legendrian and Lagrangian immersions \cite{El}, \cite{Va}, \cite{Au3} and general theorems of Ando~\cite{An}, Szucs~\cite{Sz} and the author~\cite{Sa}. 
{\ccc
In fact, the results in the current paper are most closely related to those in \cite{Sa}. In \cite{Sa} we consider open stable differential relations $\mathcal{R}$ that are invariant with respect to the contact group $\mathcal{K}$ and satisfy the h-principle over \emph{closed} manifolds. The main theorem of \cite{Sa} asserts that there exists an infinite loop space $\Omega^{\infty}\mathbf{B}$ such that the $0$-th cobordism group of $\mathcal{R}$-maps to a given manifold $W$ is isomorphic to $[W, \Omega^{\infty}\mathbf{B}]$. In comparison to \cite{Sa}, in the present paper we introduce and work with a very general  notion of an  open stable differential relation $\r$ (that, for example, may not satisfy the mentioned $h$-principle), and give a more general and precise relation of solutions of $\r$ to formal solutions of $\r$ by working on the level of moduli spaces. Furthermore, we give a conceptual connection between group completions, h/b-principles and cobordism categories \cite{GMTW}.



For stable formal solutions of \cite{Sa}, there exists no moduli space. In contrast, in the current paper, we introduce (normal) stable formal solutions and construct its moduli space $h\M_\r$.  Under the assumptions in \cite{Sa}, in the present paper we show that $\M_\r$ is an infinite loop space isomorphic to $h\M_\r$, while the result of \cite{Sa} only implies that $[W, \Omega^{\infty}\mathbf{B}]\simeq [W, \M_\r]$ for any manifold $W$. The latter isomorphisms do not imply that the infinite loops spaces $\M_\r$ and $\Omega^{\infty}\mathbf{B}$ are isomorphic; in fact, they are not  (\ref{w:7.9}).  Our argument is different from that in \cite{Sa} and it has a wider range of applications. For example, in \cite{Sa5} we show that the method developed in the present paper may still apply even without the assumption on Morse functions. 
In review \cite{Sa4} we extend the setting of the current paper to the case of such classes of maps as, for example, embeddings  and strict Morse functions, i.e., Morse functions whose critical points have distinct critical values. We emphasize that \cite{Sa5} and \cite{Sa4} essentially rely on the results of the present paper.


}

\subsection{Organization of the paper}\label{s:1.1}
For reader's convenience we include an Appendix with a review of necessary notions concerning differential relations, partial sheaves,  and $(B,f)$ structures. 

Let $\mathcal{R}$ be an open stable differential relation imposed on maps of a fixed dimension $d$ (see \S\ref{a:1}); 
the \emph{dimension} of a map $M\to N$ of manifolds is defined to be the difference $\dim M-\dim N$ of dimensions of $M$ and $N$. Two simplest open stable differential relations, which the reader may want to keep in mind, are the immersion and submersion differential relations. In these cases \emph{$\mathcal R$-maps}, i.e., solutions of $\mathcal{R}$,  are respectively immersions and submersions $M\to N$ of arbitrary smooth manifolds  with $\dim M-\dim N=d$. 
Another extreme example of an open stable differential relation is the case where $\r$ is the differential relation whose solutions are all smooth maps of dimension $d$; in this case our construction 
is a fancy version of the Pontrjagin-Thom construction.

To begin with we observe that the moduli space $\M_\r$  of solutions is an H-space with a coherent operation (see \S\ref{ss:2.2}); e.g., on vertices the H-space operation 
$\M_\r\times \M_\r\to \M_\r$ is defined by $\Delta^0_f\times \Delta^0_g\mapsto \Delta^0_{f\sqcup g}$. On the 
other hand, Segal showed~\cite{Se} that the construction of the classifying space 
of an abelian topological monoid extends to that of 
$H$-spaces with coherent operations called (for a review, see \S\ref{asec:6}). Furthermore, the 
classifying space of an $H$-space with a coherent operation itself is an $H$-space with a coherent operation. 
Taking iterative classifying spaces 
\[
    \mathcal{M}_\r, B\mathcal{M}_\r, B^2\mathcal{M}_\r, \dots, 
\]
one obtains a spectrum $\mathbf{M}_\r$ classifying a cohomology theory $k^*_\r$ of solutions (\S\ref{ss:2.3}).

To determine the homotopy type of $\Omega^{\infty}\mathbf{M}_\r$ we introduce stable formal solutions. The definition of stable formal solutions may differ from what one would expect, e.g., in the case of the differential relation of submersions of dimension $2$ our definition is somewhat different from that in the proofs of the Mumford conjecture given in \cite{MW} and \cite{EGM}.  Let me pause here and provide a motivation  for our definition. 

Recall that in the h-principle theory a \emph{formal solution} of a differential relation $\r$ is a continuous fiber preserving map $F\co TM\to TN$ whose restriction  to each fiber $T_xM$ is a solution $T_xM\to T_yN$ of $\r$. For example, a formal immersion essentially is a fiberwise linear monomorphism $F\co TM\to TN$.   Formal solutions are not suitable to work with in stable homotopy theory since there is no meaningful pullback operation for formal solutions. For example, let $F$ be  the formal immersion $T\mathbb{C}\to T\mathbb{C}$ defined by $(x, v)\mapsto (x, iv)$, where $x$ is a point in $\mathbb{C}$ and $v$ a vector at $x$. Then, for the standard inclusion
$f\co \R\to \mathbb{C}$, the pullback $f^*F$ is not well-defined.

The counterpart of the notion of formal solutions in the b-principle is the notion of \emph{(normal) stable formal solutions}. Suppose that $M$ is a compact manifold embedded into $N\times \R^{\infty}$. Let $\mathbf{R}^{\infty}$ denote the total space of the trivial vector bundle over $M$ with fiber $\R^{\infty}$, and $\nu$ denote the normal bundle of $M$ in $N\times \R^{\infty}$. Then, intuitively, a (normal) stable formal solution is a continuous fiber bundle map $\mathbf{R}^{\infty}\to \nu$ whose restriction to each fiber is a product 
\[
    \R^{\infty}\simeq\R^k\times \R^{\infty-k}\longrightarrow \nu_0\times \R^{\infty-k}\simeq\nu_x
\]
of a solution of $\r$ and the identity map of $\R^{\infty-k}$, 
where  $\nu_x$ is the fiber of $\nu$ over $x\in M$, $k$ is an integer, and $\nu_0$ is a vector space for which $\nu_x\simeq \nu_0\times \R^{\infty-k}$.  
A (normal) stable formal solution $F$ leads to a (tangential) stable formal solution 
\[
TM\times \R^{\infty}= TM\oplus \mathbf{R}^{\infty}\xrightarrow{\id\oplus F}  TM\oplus \nu\simeq TN\times \R^{\infty},
\] 
which is a stabilized version of a formal solution. The reason we use normal stable formal solutions instead of seemingly simpler  tangential ones is that in contrast to the later ones the former ones behave well (on the nose not just up to homotopy) with respect to taking pullbacks.  In particular, for normal stable formal solutions there is a moduli space $h\M_\r$ and a cohomology theory $h^*_\r$ of stable formal solutions (\S\ref{s:3}, \ref{ss:2.3}).

{\ccc
Every genuine solution of $\r$ gives rise to a normal stable formal solution of $\r$, and therefore it is plausible that
there is a natural transformation $\mathfrak{A}_\r\co k^*_\r\to h^*_\r$. The b-principle asserts that $\mathfrak{A}_\r$ is an equivalence of cohomology theories (\S\ref{ss:2.3}). Let $\M_\r\to h\M_\r$ be the composition of the inclusion of $\M_\r$ into the infinite loop space of $k^*_\r$ and a map classifying the natural transformation $\mathfrak{A}_\r$.  We say that a map $A\to C$ of two $H$-spaces with coherent operations is a \emph{Segal group completion} if $C$ is an infinite loop space and the induced map $B^nA\to B^nC$ is a weak homotopy equivalence for all $n\ge 1$. In \S\ref{s:11} we determine the homotopy type of $h\M_\r$; it turns out that $h\M_\r$ is an infinite loop space. Therefore the b-principle is a group completion type theorem (see  \cite[Proposition 4.1]{Se}).  }

\begin{definition}[A reformulation of Definition~\ref{d:b-principle}] \label{d:9}
If, for a given open stable differential relation $\mathcal R$, the canonical map $\M_\r\to h\M_\r$ is a Segal group
completion, then we say that {\ccc $\mathcal R$ satisfies the b-principle.} 
\end{definition}

To establish the bordism principle, we 
consider the topological space $\mathcal{B}\M_\r$ which is a model 
for the classifying space $B\M_\r$. The definition of the space 
$\mathcal{B}\M_\r$ is similar to that of $\M_\r$; it is the union  of simplices $\Delta^n_{(f, \alpha)}$ parametrized by proper maps 
\[
(f, \alpha)\co V\to \Delta^n\times \R
\]
 such that $f$ is a solution of $\r$
transverse to the faces of $\Delta^n$ and $\alpha(f^{-1}(x))\ne \R$ for all $x\in \Delta^n$, see Figure~\ref{fig:1a} where $N=\Delta^n$. The proof 
that $\mathcal{B}\M_\r$ is a model for the classifying space of $\M_\r$ follows closely the argument of 
Galatius-Madsen-Tillman-Weiss \cite{GMTW}, and relies on the construction of classifying $\Gamma$-spaces by Segal~\cite{Se}. Here we introduce and use \emph{category valued partial $\Gamma$-sheaves} in order to show that $\mathcal{B}\M_\r$
is equivalent to $B\M_\r$ not only as a space, but also as an $H$-space with a coherent operation.

Similarly we construct a model $\mathcal{B}h\M_\r$ for the classifying space of $h\M_\r$. Again $\mathcal{B}h\M_\r$
is a simplicial complex of simplicies $\Delta^n_{(f, \alpha)}$ as in  
$\mathcal{B}\M_\r$ except that here $f\co V\to \Delta^n$ are stable formal solutions of $\r$; see Figure~\ref{fig:7a}, where $N=\Delta^n$ and  $f=(u, F)$ is a pair of a smooth map $u\co V\to \Delta^n$ and a structure $F$ of a stable formal solution covering $u$. 
By means of a singularity theoretic argument we show
that in the definition of $\mathcal{B}h\M_\r$ 
the \emph{rigid condition} $\alpha(u^{-1}(x))\ne \R$ can be replaced with the \emph{flexible condition} that 
each regular fiber of $(u, \alpha)$---which is a smooth closed manifold---is cobordant to zero.
The replacement does not change the homotopy type of $\mathcal{B}h\M_\r$; see Figure~\ref{fig:8a}, where again $N=\Delta^n$ and $f$ is a pair $(u,F)$. 

Assuming the flexible condition, 
we can show that $\mathcal{B}h\M_\r$ admits an alternative definition: it is homotopy equivalent to its subcomplex of  simplicies $\Delta^n_{(f, \alpha)}$, 
one for each proper map $(f, \alpha)\co V\to \Delta^n\times \R$ such that $f$ is a (genuine) solution of $\r$
transverse to the faces of $\Delta^n$. Indeed, given a proper map $(f, \alpha)\co V\to \Delta^n\times \R$, where $f$ is a stable formal solution, we 
can stretch path components of $(f, \alpha)(V)\subset \Delta^n\times \R$ to infinity so that each component of $V$ 
is open. Then we can further modify $(f, \alpha)$ so that $f$ is an (unstable) formal solution. The modification is justified by the Destabilization Lemma~\ref{l:11.1}, which is a general version (for arbitrary open stable differential relations) of the destabilization constructions in \cite[Section 11]{Sa} and \cite[Section 2.5]{EGM}. Finally, the Gromov h-principle \cite{Gro} for \emph{open} manifolds allows us 
to deform the obtained formal solution $V\to \Delta$ of $\r$ to 
a genuine solution of $\r$.

Thus, the space $\mathcal{B}\M_\r$ can be identified with a simplicial subcomplex of $\mathcal{B}h\M_\r$ that consists
of simplicies $\Delta^n_{(f, \alpha)}$ with $\alpha(u^{-1}(x))\ne \R$, where $f=(u, F)$. 
In the case of differential relations imposed on maps of dimension $d<1$, the classifying spaces $\mathcal{B}\M_\r$ and $\mathcal{B}h\M_\r$ are clearly the same (\S\ref{sec:8}), which implies  Theorem~\ref{th:1} in the case $d<1$. The case $d\ge 1$ of Theorem~\ref{th:1} is  established by means of a singularity theoretic construction that produces ``holes" in fibers of $(f, \alpha)$. 

Finally, we determine (\S\ref{s:11}) the homotopy type of the space $h\M_\r$ and list several applications (\S\ref{s:8}). 


\subsection{Acknowledgments} I am thankful to the anonymous referee for many important suggestions
that greatly improved the presentation. I am also thankful to the participants of the 
Topology seminar in CINVESTAV---where I gave a series of talks presenting the proof of the result---and especially to Ernesto Lupercio for many remarks and suggestions. 
Many ideas and constructions in the current paper are motivated/borrowed from proofs of the Mumford Conjecture, notably from very influential papers \cite{MW} by Madsen-Weiss and \cite{GMTW} by Galatius, Madsen, Tillmann and Weiss. I would like to thank Victor Turchin for carefully reading the latest version of the preprint and for many helpful comments. I appreciate the hospitality of the Toronto University where the first draft of the paper was written, as well as the hospitality of Institut des Hautes \'{E}tudes Scientifiques and the Max-Planck-Institut f\"{u}r Mathematik where the paper acquired the present form.

{\ccc
\section{Differential relations}\label{a:1}
}

Let $M$ and $N$ be two manifolds. A (smooth) \emph{map germ} $f\co  (M, x)\to (N, y)$ is represented by a smooth map $f$ from a neighborhood of a point $x$ in $M$ to a neighborhood of a point $y$ in $N$ such that $f(x)=y$. Two maps $f$ and $g$ represent the same map germ if their restrictions to some neighborhood of $x$ in $M$ coincide. An \emph{unfolding} of a map germ $f\co  X\to Y$ at $x$ is a cartesian diagram of smooth map germs 
\[
\begin{CD}
(P, p) @>F >> (Q, F(p)) \\
@A iAA @AjAA \\
(X, x) @> f>> (Y, f(x)) 
\end{CD}
\]
where $P$ and $Q$ are smooth manifolds, $p$ is a point in $P$, and $i$ and $j$ are immersion germs such that $j$ is transverse to $F$. An unfolding is \emph{trivial} it there are map germs $r\co  P\to X$ at $p$ and $s\co  Q\to Y$ at $F(p)$ such that $r\circ i=\id_X$, $s\circ j=\id_Y$ and $f\circ r=s\circ F$. A map germ $f$ is \emph{stable} if each of its unfoldings is trivial. An unfolding is \emph{stable} if the map $F$ in the diagram is stable as a map germ. Any two stable unfoldings of the same dimension of a map germ are isomorphic, see \cite[page 86]{GWPL}. 
It is also known that the subset of map germs with no stable unfoldings is of infinite codimension in the space of all map germs (e.g., see \cite{GWPL}). Thus the stability condition in Definition~\ref{d:9a} below is a mild condition.

Let $d$ be an arbitrary integer; it will be fixed throughout the section. Let $J^k(n)$ denote the set of equivalence classes $[f]$ of map germs 
\begin{equation}\label{eq:germ}
f\co  (\R^{n+d},0) \longrightarrow (\R^n,0)
\end{equation}
where two map germs $f, g$ are equivalent if and only if all derivatives at $0$ of $f$ and $g$ of order $\le k$ are the same. The space $J^k(n)$ has an obvious structure of a smooth manifold diffeomorphic to $\R^{N}$ for some positive integer $N$. It admits a smooth action of the group $\mathcal{G}(n)$ of pairs  $(\alpha, \beta)$ of diffeomorphism germs $\alpha$ of $(\R^{n+d}, 0)$ and $\beta$ of $(\R^{n}, 0)$. An element $(\alpha, \beta)$ takes $[f]$ onto $[\beta\circ f\circ \alpha^{-1}]$. There are inclusions $J^k(n)\subset J^k(n+1)$ that take $f$ onto $f\times \id_{\R}$ where $\id_{\R}$ is the identity map of $\R$. The colimit of these inclusions is denoted by $\mathbf{J}$. Similarly, there are inclusions $\mathcal{G}(n)\to \mathcal{G}(n+1)$, whose colimit is denoted by $\mathbf{G}$. The actions of the groups $\mathcal{G}(n)$ on $J^k(n)$ induce an action of the group $\mathbf{G}$ on $\mathbf{J}$.

\begin{definition}\label{d:9a} A \emph{stable differential relation} $\mathcal{R}$ of order $k$ is defined to be an arbitrary $\mathbf{G}$-invariant subset of $\mathbf{J}$ of equivalence classes of stable map germs. A \emph{solution} of $\mathcal{R}$ is a smooth map of dimension $d$ each map germ of which has a stable unfolding that in some (and hence any) local coordinates represents a point in $\mathcal{R}$. A differential relation $\mathcal{R}$ is \emph{open}, if it is open as a subset of $\mathbf{J}$.    
A solution of a differential relation $\mathcal{R}$ is also called an \emph{$\mathcal{R}$-map}.
\end{definition}

For example, every Thom-Boardman differential relation~\cite{Sa} is an example of a stable differential relation. Its solutions are maps with prescribed Thom-Boardman singularities~\cite{Bo}. Immersion and submersion differential relations are the simplest Thom-Boardman differential relations.

Stable differential relations $\mathcal{R}$ satisfy two properties that are important for our purposes. First, given a solution $f\co M\to N$, the map $f\times \id_\R\co M\times \R\to N\times \R$ is also a solution. And, second, given a solution $f$ over $N$ and a smooth map $u$ to $N$ transverse to $f$, the pullback $u^*f$ is a also a solution.

\section{The moduli space $\mathcal{M}_\r$ of solutions of $\mathcal{R}$}\label{ss:2.1}

For a non-negative integer 
$n$, the \emph{extended $n$-simplex} $\Delta^n_e$ is defined to be the opening of the standard $n$-simplex $\Delta^n$ in 
\[
   \R^n\simeq \{\, (x_0, \dots, x_n)\in \R^{n+1}\, |\, x_0+\cdots+x_n=1\,\}; 
\] 
recall \cite{Gro} that the \emph{opening} of a subspace $A$ of a topological space $V$ is an 
arbitrary small but non-specified open neighborhood of $A$ in $V$. 
As in the case of standard simplices each extended simplex comes with \emph{face} and \emph{degeneracy} maps
\[
\delta_i\co \Delta^{n-1}_e\to \Delta^n_e,  \qquad  \delta_i\co (x_0, \dots, x_{n-1})\mapsto (x_0, \dots, x_{i-1}, 0, x_{i}, \dots, x_{n-1}),
\]
\[
\sigma_j\co \Delta^n_{e}\to \Delta^{n-1}_e, \qquad  \sigma_j\co (x_0, \dots, x_{n})\mapsto (x_0, \dots, x_{j}+x_{j+1}, \dots, x_{n}),
\]
where the indices $i$ and $j$ range from $0$ to $n$ and $n-1$ respectively. 

We adopt a few conventions. First, in order to avoid set-theoretic issues, we often equip a proper map $f\colon V\to N$ of smooth manifolds with 
a lift to an embedding
\begin{equation}\label{eq:2.1.1}
 (f, \beta)\co V\hookrightarrow N\times \R^{\infty}, 
\end{equation}
such that the closure of the image of $\beta$ is compact;
every proper map $f$ admits such a lift, and 
the space of lifts of $f$ is weakly contractible in $C^{\infty}$ Whitney topology. 
Following \cite[Section 2]{MW}
 we call a map (\ref{eq:2.1.1}) a \emph{graphic
map} to $N$. Note that a graphic map (\ref{eq:2.1.1}) is determined by the embedded submanifold $V$. Similarly, a graphic map $(f, \beta)$ to $\Delta^n_e$ is determined by an equivalence class---called a \emph{germ submanifold}---of embedded submanifolds 
\[
   (f, \beta)\co V\hookrightarrow \R^n\times  \R^{\infty},
\]
two embedded submanifolds represent the same germ submanifold if they coincide over  $\Delta^n_e\times \R^{\infty}$. To simplify notation, we often tacitly  identify a graphic map $(f, \beta)$ to $\Delta^n_e$ with any of its representatives. 

Second, all maps to triangulated manifolds in this paper are assumed to be transverse to each simplex
of the triangulation. Similarly, for each graphic map $(f,\beta)$ to $\Delta^n_e$, 
the underlying map $f$ is required to be transverse to every face map of $\Delta^n_e$. 

\begin{lemma}\label{l:2.2} Given a graphic $\mathcal{R}$-map $(f, \beta)$ to $\Delta^n_e$, each face and degeneracy map with target $\Delta^n_e$ pulls back a graphic $\mathcal{R}$-map to $\Delta^{n-1}_e$ and $\Delta^{n+1}_e$ respectively. 
\end{lemma}
\begin{proof} Let $\delta$ be one of the face maps of $\Delta^n_e$. Since $f$ is transverse to $\delta$, there is a smooth manifold $W$ defined by the pullback diagram  
\[
\begin{CD}
W  @>\subset>> V  \\
@V \delta^*f VV @V f VV \\
\Delta^{n-1}_e @>\delta>> \Delta^{n}_e. 
\end{CD}
\]
A counterclockwise rotation of this pullback diagram shows that at each point $w\in W$ the map germ $f$ at $w$ unfolds the map germ $\delta^*f$ at $w$ (see \S\ref{a:1}). Since the differential relation $\mathcal{R}$ is stable, this implies that $\delta^*f$ is an $\mathcal{R}$-map. Also $\delta^*f$ is proper. Consequently, the pullback $(\delta^*f, \beta|W)$ is a 
graphic $\mathcal{R}$-map to $\Delta^{n-1}_e$. The statement of Lemma~\ref{l:2.2} for the degeneracy maps  follows from the fact that each degeneracy map $\sigma$ is a submersion and hence $\sigma^*f$ is a proper $\mathcal{R}$-map.
\end{proof}

A \emph{graphic $\r$-map} $f$ to a simplicial set $C_{\bullet}$ is a family 
of graphic $\r$-maps $f_x$ to extended simplicies $\Delta^n_e$, one for each $n$-simplex $x$ in $C_{\bullet}$, 
such that 
\[
(\delta_i\times \id_{\R^{\infty}})^*(f_x, \beta_x)=(f_{d_ix}, \beta_{d_ix})
\]
and 
\[
(\sigma_j\times \id_{\R^{\infty}})^*(f_x, \beta_x)=(f_{s_jx}, \beta_{s_jx})
\]
for all $i,j$, where $d_i$ and $s_j$ are 
face and degeneracy operators in the simplicial set $C_{\bullet}$. To simplify notation, occasionally we will omit $\beta$ from the notation. 
Lemma~\ref{l:2.2} shows that the correspondence $F\co \mathbf{Sset}^{op}\to \mathbf{Set}$ that associates to each simplicial set $C_{\bullet}$ the 
set of graphic $\r$-maps to $C_{\bullet}$ is a functor.

\begin{definition}\label{d:2.2}
Given a contravariant set valued functor $F$ on a category $\mathcal{C}$, an object 
$\M\in \mathcal{C}$ is said to be the \emph{moduli space} for $F$ if there exists
a natural equivalence $\gamma_{\M}\co F\to \mathbf{Hom}(-, \M)$.
It follows that if the moduli space for $F$ exists, then it 
is unique up to isomorphism.  
\end{definition}

\begin{theorem}\label{th:-1} Let $F\co \mathbf{Sset}^{op}\to \mathbf{Set}$
be the functor that associates to each simplicial set $C_{\bullet}$ the 
set of graphic $\r$-maps to $C_{\bullet}$. Then 
there exists a moduli space $X_\bullet$ for  $F$.  
\end{theorem}

\begin{proof}[Proof of Theorem~\ref{th:-1}]
Let $X_n$ denote the set of smooth (germ) submanifolds $V$ in $\Delta^n_e\times \R^{\infty}$ such that 
the inclusion map is a graphic $\mathcal{R}$-map. Essentially, Lemma~\ref{l:2.2} turns the correspondence $[n]\mapsto X_n$ into a simplicial set $X_{\bullet}$. For $i=0,...,n$, its $i$-th boundary operator $d_i\co X_n\to X_{n-1}$ takes $V$ to its pullback by means of $\delta_i$ defined in Lemma~\ref{l:2.2}. Similarly, for $j=0,...,n-1$, the $j$-th degeneracy operator $s_j\co X_{n-1}\to X_{n}$ pulls $V$ back by means of $\sigma_j$. 

Let $f$ be a graphic $\r$-map to a simplicial set $\mathcal{C}_\bullet$.
Then,  for each simplex $x$ in $\mathcal{C}_\bullet$, 
the pair
$(x, f_x)$ corresponds to a unique simplex in $X_\bullet$; hence $f$ determines
a map $\mathcal{C}_\bullet\to X_\bullet$. The so-defined
transformation
      $F \to \mathbf{Hom}(-, X_{\bullet})$ is an equivalence.
\end{proof}

\begin{definition}\label{d:1.2.1}
The geometric realization $|X_{\bullet}|$ of the simplicial set $X_{\bullet}$ is called 
the \emph{moduli space $\M_{\mathcal R}$} of solutions of the differential relation $\mathcal R$. 
\end{definition}


\begin{remark} The space $\M_\r$ informally defined in the introduction is a ``fat" version of $|X_\bullet|$ containing extraneous degenerate simplicies. The space $\M_\r$ of introduction is homotopy equivalent to the moduli space $\M_\r$ of Definition~\ref{d:1.2.1}, e.g., see \cite{Se}. 
\end{remark}

The moduli space $\M_\r$ of solutions is the principal object of study in this paper. In the rest of the section we give two other definitions of $\M_\r$; we will use interchangeably all three definitions. 

We observe that the space
$\M_\r=|X_\bullet|$ possesses a
\emph{homotopy classifying property}; namely, 
$X_\bullet$ is the moduli space in $\mathop{\mathrm{Ho}}(\mathbf{Sset})$ 
for the functor of concordance classes of $\r$-maps. Recall, that 
the homotopy category $\mathop{\mathrm{Ho}}(\mathbf{Sset})$ is formed by 
formally inverting the weak equivalences in $\mathbf{Sset}$. The category
$\mathop{\mathrm{Ho}}(\mathbf{Sset})$
is equivalent to the category of Kan complexes and simplicial homotopy classes
of maps, e.g., see ~\cite{GJ}.

\begin{definition} Two graphic $\r$-maps $f_0, f_1$ to a simplicial set
$C_\bullet$ are \emph{concordant} if there exists a graphic
$\r$-map $f$ to the simplicial set $C_\bullet\times \Delta^1_\bullet$ such that 
$$\id\times \delta_i\co C_\bullet\times \Delta^0_\bullet\to C_\bullet\times \Delta^1_\bullet$$
pulls $f$ back to $f_i$ for $i=0,1$, where $\Delta^n_\bullet$ is the standard $n$-simplex $\mbox{hom}(-, [n])$.  
\end{definition}

Let $\mathop{\mathrm{Ho}}F$ be the set valued functor on $\mathop{\mathrm{Ho}}(\mathbf{Sset})$
that associates to a simplicial set $K$ the set of concordance classes
of $\r$-maps to $K$. Then in view of the natural equivalence
\begin{equation}\label{eq:2.1.2.1}
    \gamma \co \mathop{\mathrm{Ho}} F \longrightarrow  [-, X_{\bullet}],
\end{equation}    
the complex $X_{\bullet}\in \mathop{\mathrm{Ho}}(\mathbf{Sset})$ is the moduli 
space for $\mathop{\mathrm{Ho}}F$. As above, the condition (\ref{eq:2.1.2.1})
determines the simplicial complex 
$\M_\r=|X_{\bullet}|$ up to homotopy equivalence.

Alternatively, let $\mathcal{F}(N)$ denote the set of submanifolds $V\subset N\times \R^{\infty}$ such that the inclusion $(f, \beta)$ is a graphic 
$\r$-map to a smooth manifold $N$. Given a smooth map $\varphi\co M\to N$, let $\mathcal{F}(\varphi)$
denote the map 
\[
(f, \beta) \mapsto (\varphi\times  \id_{\R^{\infty}})^*(f, \beta)
\]
from the subset of $\mathcal{F}(N)$ of graphic $\r$-maps $(f, \beta)$ with $f$
transverse to $\varphi$ to the set $\mathcal{F}(M)$.
Then $\mathcal{F}_\r=\mathcal{F}$ is a set valued partial sheaf (\S\ref{a:1.5}) on the 
category of smooth manifolds and smooth maps. 
It follows that $\M_\r$ is the geometric realization of $\mathcal{F}_\r$. 

Thus we have given three equivalent definitions of $\M_\r$; namely, $\M_\r$ is 
the geometric realization of (1) the moduli space for $F$, 
(2) the moduli space for $ \mathop{\mathrm{Ho}} F$, and
(3) the partial sheaf $\mathcal F_\r$.

\begin{example}  If every map of dimension $d$ is a solution of a 
differential relation $\mathcal R$, then $\M_{\mathcal R}$ is homotopy
equivalent to the topological space $\Omega^{\infty+d}\MO$. The moduli space of submersions 
of dimension $d$ is a model for $\sqcup \BDiff M$, where $M$ ranges over 
diffeomorphism types of closed (possibly not path connected) manifolds of dimension $d$. For immersions of 
codimension $d$, the moduli space is homotopy equivalent to 
the infinite loop space of the spectrum for the normal $(B, f)$ structure
$\BO_d\subset \BO$, cf. Example~\ref{ex:1}.  The second example follows
from the homotopy classifying 
property for $\M_{\mathcal R}$, while the two other cases follow from the 
b-principle  (for details, see Examples in \S\ref{s:11}). 
Similarly, the moduli space for special generic fold maps is related to diffeomorphism groups of spheres~\cite{Sa3}. Also, as we show in a review \cite{Sa4}, 
moduli spaces are defined not only for solutions of differential relations, but also for any class $\mathcal{C}$ of smooth maps of a fixed dimension $d$ provided that $\mathcal{C}$ is closed with respect to taking pullbacks, i.e., if $f\in \mathcal{C}$  and $g$ is any smooth map transverse to $f$ then $g^*f\in \mathcal{C}$. For example, the moduli space for embeddings of codimension $d$ is homotopy equivalent to the Thom space of the universal vector bundle over $\BO_d$.  
\end{example}


\section{The moduli space $h\M_\r$ of stable formal solutions of $\mathcal R$}\label{s:3}
Recall that a \emph{formal solution} of a differential relation $\mathcal{R}$ is a continuous map $f\co M\to N$ of smooth manifolds together with a continuous family $F=\{F_x\}$ of map germs 
\[
    F_x\co (T_xM, 0)\longrightarrow (f^*T_{f(x)}N, 0)
\]
parametrized by $x\in M$ such that each map germ $F_x$ is a solution of $\mathcal {R}$. For example,  
a formal immersion is essentially a fiber bundle map $F\co TM\to TN$ which restricts on each fiber $T_xM$ to a linear monomorphism. A \emph{stable formal solution} is defined similarly by means of stable map germs, or s-germs for short. We will see shortly that in contrast to formal solutions, stable formal solutions behave well under pullbacks and, in particular, there is a well-defined moduli space for stable formal solutions. 
 
 {\ccc
 Often a formal solution is defined to be a family of $k$-jets, not germs. However, for an open differential relation of order $k$ the two notions are equivalent as the space $R$ of formal (germ) solutions is homotopy equivalent to the space $J$ of formal ($k$-jet) solutions. In the present paper, we work with map germs since, in general, the formulas and statements involving germs are simpler than those for jets. 

\begin{lemma} The projection $R\to J$ is a homotopy equivalence. 
\begin{proof} The space $R$ can be identified with the space of sections of a fiber bundle over $M\times N$ with fiber $R_0$ given by the space of $\r$-map germs $(R^m, 0) \to (R^n, 0)$. Similarly, the space $J$ is a space of sections of a fiber bundle over $M\times N$ with fiber $J^k_0$ of $k$-jets of $\r$-map germs $(R^m, 0) \to (R^n, 0)$. The $k$-jet of a map germ $(\R^m, 0)\to (\R^n, 0)$ can be uniquely represented by the map germ of a polynomial map of degree $k$. Hence there is an inclusion $J^k_0\to R_0$, which is homotopy inverse to  
the projection $R_0\to J^k_0$. Hence, the projection $R\to J$ is a homotopy equivalence. 
\end{proof}
 \end{lemma}}

\begin{definition} Let $M$ and $N$ be two smooth manifolds, and $x$ and $y$ are two points in $M$ and $N$  respectively. 
 An \emph{s-germ} is an equivalence class of map germs 
\begin{equation}\label{eq:3.1.0}
g\co (M\times \R^k, x\times 0)\to (N\times \R^k, y\times 0),
\end{equation} 
where $k$ is some non-negative integer;   the relation is generated by the equivalences $g\sim g\times\id_\R$, where
$\id_\R$ is the identity map germ of $\R$ at $0$.   
\end{definition}

The equivalence class of a map germ $g$  is denoted by the same symbol $g$. In fact, we write $g\co (M, x)\to (N,y)$ even if the s-germ $g$ does not admit a representation by a map of an open neighborhood of $x$ in $M$. Furthermore, 
we suppress the reference points from the notation 
if $M$ and $N$ are vector spaces and the reference points are their respective origins. 

The set of all stable map germs from $M$ to $N$ forms the total space $E=E(M,N)$ of a fiber bundle over $M\times N$. More precisely, let $E_0$ denote the infinite dimensional vector space of s-germs $\R^m \to \R^n$ where $m$ and $n$ are the dimensions of $M$ and $N$ respectively. In fact, since 
every s-germ admits a representation by a map of an open neighborhood of $0$ in $\R^m\times \R^i$ for some finite $i$, the space $E_0$ is the colimit of vector spaces of map germs $\R^{m+i}\to \R^{n+i}$ for $i\ge 0$.  
There is a left action on $E_0$ by the group $G=\mbox{GL}_m\times \mbox{GL}_n$ of pairs of linear isomorphisms; a pair $(g, h)$ of linear transformations sends a stable map germ $u$ to 
\[
(h\times \id) \circ u  \circ  (g^{-1} \times \id).
\]
 Let $P$ denote the total space of the principle $G$-bundle associated with the tangent bundle of $M\times N$; a point in $P$ corresponds to a  $4$-tuple $(x, y, \tau_M, \tau_N)$ of a point $x$ in $M$, a point $y$ in $N$, and bases $\tau_M$ and $\tau_N$ of $T_xM$ and $T_yN$ respectively. The mentioned space $E(M,N)$ is the Borel construction 
\[
      E(M, N) = P\times_G E_0. 
\]
Thus, a point in $E$ is a triple $(x, y, f)$ of a point $x\in M$, a point $y\in N$ and an s-germ $f\co T_xM\to T_yN$. Also, the space $E$ is 
the total space of a fiber bundle 
$E\to M$; its projection is given by composing the projection of $E$ onto $M\times N$ with the projection of $M\times N$ onto the first factor.

\begin{definition}\label{d:3.2}  Let $u\co M\to N$ be a continuous map of smooth manifolds. A \emph{tangential stable formal $\mathcal R$-map} over $u$ is a continuous family $F$ of s-germs 
\begin{equation}\label{eq:3.1.1}
    F_x\co  T_xM \longrightarrow T_{u(x)}N
\end{equation}
parametrized by $x\in M$ such that $F_x$ is a solution germ of $\mathcal R$-maps for each $x$.  More precisely, a tangential stable formal $\r$-map is a continuous section of the fiber subbundle of $E(M, N)\to M$ whose total space consists of triples $(x, y, f)$ where $f$ is a stable formal $\r$-map.  
 In some constructions we replace $T_{u(x)}N$ with its canonically isomorphic copy  $u^*T_{u(x)}N$.
A stable formal $\mathcal R$-map over  
$u$ is said to be \emph{proper} or \emph{smooth} if the underlying map $u$ is such.  A smooth map to $N$ is said to be \emph{transverse} to $F$ if it is transverse to the underlying map $u$.
\end{definition}

\begin{example}  For closed manifolds,  a stable formal immersion over $u$ is essentially an equivalence class of fiberwise linear monomorphisms  
$TM\oplus \varepsilon^k \to TN\oplus\varepsilon^k$ covering $u$; here $\varepsilon$ stands for a trivial line bundle 
over an appropriate space, and $k$ an integer. Note that every such map is homotopic to the direct sum of a map $TM\oplus \varepsilon\to TN\oplus \varepsilon$ and the identity map of the trivial bundle $\varepsilon^{k-1}$ over $M$.  
\end{example}

In the notation of Definition~\ref{d:3.2},  every s-germ $F_x$ admits a representation by a map germ from a neighborhood of $M\times \R^{k(x)}$ for some finite $k(x)$. Such a family of representing map germs with the same $k(x)=k$ for all $x$ is said to be a \emph{finite representation of $F$}. Every proper tangential stable formal $\r$-map to a compact manifold admits a finite representation.  

Tangential stable formal $\r$-maps do not behave well under pullbacks. In particular, there is no direct construction of a moduli space for tangential stable formal $\r$-maps. For this reason we also define the notion of a \emph{normal stable formal $\r$-map}.  Recall 
that in the definition of a graphic map $(u, \beta)$ we require that the image of $\beta$ belongs to $\R^t$ for some finite $t$. Given a  graphic map 
\[
(u, \beta)\co M\longrightarrow N\times \R^{t},
\]
let $\nu$ denote its normal bundle; if we identify $M$ with the image of $(u, \beta)$, then the vector bundle $\nu$ is the quotient of the vector bundle $TN\times T\R^t|M$ by the subbundle $TM$. We emphasize that we do not use Riemannian metrics here. Note also that $t=d+|\nu|$, where $|\nu|$  is the dimension of the vector bundle $\nu$. 

\begin{definition} A \emph{normal stable formal $\r$-map} $F^\nu$ covering a graphic map $(u, \beta)$ is a continuous family of s-germs 
\[
        F^\nu_x\co \R^{d+|\nu|}\longrightarrow \nu_x
\]
parametrized by $x\in M$, where $\nu_x$ stands for the fiber of $\nu$ over $x$.   
More precisely, a normal stable formal $\r$-map is a section of a fibre bundle constructed as in the definition of a tangential stable formal $\r$-map. 
\end{definition}

The notion of a normal stable formal $\r$-map $F^{\nu}=F^\nu(t)$ covering $(u, \beta)$  is well-defined, i.e., it does not depend on the choice of $t$. Indeed, suppose that $t$ is replaced with $t+1$ in the definition of $F^\nu$. Then $\nu$ is replaced with $\nu\oplus \varepsilon$, where $\varepsilon$ is the trivial linear bundle over $M$. Thus, families $F^\nu(t)$ can be canonically identified with families $F^{\nu}(t+1)$.

If a normal stable formal $\r$-map $F^{\nu}$ is represented by a family of map germs 
\[
       F_x^{\nu}\co \R^{d+|\nu|}\times \R^k\longrightarrow \nu_x\times \R^k
\] 
with the same $k$ for all $x$, then the representation is \emph{finite}. Every normal proper stable formal $\r$-map to a compact manifold admits a finite representation. Such a representation defines a (unique up to homotopy) finite representation of a tangential stable formal $\r$-map $F$:
\begin{equation}\label{eq:5a}
  T_xM\times \R^{d+|\nu|}\times \R^k \xrightarrow{\id_{T_xM}\times F_x^\nu} T_xM\times \nu_x \times \R^k\simeq T_{u(x)}N\times \R^{d+|\nu|}\times \R^{k}.  
\end{equation}

\begin{remark} The isomorphism in the composition (\ref{eq:5a}), say $\varphi$, is not canonical, and therefore the tangential stable formal $\r$-map $F$ is defined only up to homotopy. The isomorphism $\varphi$ can be chosen in a canonical way if the manifolds $M$ and $N$ are Riemannian. However, the so-defined correspondence $F^{\nu}\to F$ does not behave well (on the nose, not just up to homotopy) with respect to taking pullbacks.    
\end{remark}

Conversely, under a mild assumption, every tangential finite representation defines a (unique up to homotopy) normal finite representation. Indeed, recall that in the definition of a graphic map 
\[
(u, \beta)\co M\longrightarrow N\times \R^{\infty}
\]
we require that $(u, \beta)$ is an embedding. We say that a graphic map is \emph{generic} if  the map $\beta$ is actually an embedding itself.  We note that in contrast to graphic maps, generic graphic maps do not lead to a sheaf. 

Let $F$ be a tangential stable formal $\r$-map covering a graphic map $(u, \beta)$ of a compact manifold $N$. Suppose that $(u, \beta)$ is generic; in particular, the image of $\beta$ is a submanifold of $\R^t\subset \R^{\infty}$ for some finite $t$. Let $TM^\perp$ denote the normal bundle of $\beta(M)$ in $\R^t$.  
There is an isomorphism 
\[
TM^\perp\oplus u^*TN \simeq \nu
\]
of vector bundles over $M$, where $\nu$ is defined as above to be the normal bundle of the image of $M$ in $N\times \R^t$. Note that $t=d+|\nu|$.    
Suppose that $F$ is represented by a family of $\r$-map germs
\[
     F_x\co T_xM\times \R^k \longrightarrow T_{u(x)}N\times \R^k,
\] 
where $k$ is a non-negative integer, the same for all $x$. Then $F$ defines a (unique up to homotopy) normal finite representation by map germs
\[
     F_x^{\nu}\co \R^t\times\R^k\simeq T_xM^\perp\oplus T_xM\times \R^k
     \xrightarrow{\id\oplus F_x} T_xM^{\perp}\oplus u^*T_{u(x)}N\times \R^k\simeq \nu_x\times \R^k. 
\]
We observe that the first isomorphism is defined canonically up to homotopy. 

\begin{remark}
In most of the arguments below tangential stable formal $\r$-maps  are most helpful. However, there is  no direct definition of pullbacks of tangential stable formal $\r$-maps.   Since pullbacks are essential in the definition of moduli spaces, we use both tangential and normal stable formal $\r$-maps.  
\end{remark}

Finally, we are in position to define pullbacks (and moduli spaces) for (normal) stable formal $\r$-maps. 

Let $(u,\beta)\co M\to N\times \R^{\infty}$ be a graphic map with image of $\beta$ in $\R^t$, and  $f$ a smooth map $N'\to N$ transverse to $u$. Then the pullback map $f^*u$ has a natural structure of a smooth graphic map:
\[
     (f\times \id)^*(u, \beta)\co  M'\longrightarrow N'\times \R^{t},
\]
where $\id$ is the identity map of $\R^{t}$. As above, let $\nu$ denote the normal bundle over $M$ of the embedding $(u, \beta)$. Similarly, let $\nu'$ be the normal bundle over $M'$ of the embedding $(f \times \id)^*(u, \beta)$. Denote by $\tilde f$ the map from the diagram 
\[
\xymatrix{
N'\times \R^t\ \ar[d]  &\ M' \  \ar[d]_{f^*u}  \ar@{_{(}->}[l]  \ar[r]^{\tilde f} &\ M\ \ar[d]_{u} \ar@{^{(}->}[r]  &\ N\times \R^t  \ar[d]   \\
N' \ \ar@{=}[r] &\ N'  \ \ar[r]^{f}   &\ N \ \ar@{=}[r]  &\ N,  
}
\]

where the central square is a pullback square. Then $\tilde f$ 
is covered by a canonical fiberwise isomorphism $\nu'\to \nu$.  Let $F$ be a normal stable formal $\r$-map covering $u$. 

\begin{definition} The \emph{pullback} $f^*F^\nu$ of a normal stable formal $\r$-map $F^{\nu}$ is the normal  stable formal $\r$-map with the normal finite representation by map germs
\[
    (f^*F)^{\nu'}_z\co   \R^t\times \R^k\stackrel{F_x^\nu}\longrightarrow \nu_x\times \R^k \simeq \nu_z'\times \R^k
\]
where $z$ ranges over points of $M'$, and where the canonical isomorphism identifies the normal space of $M'$ at $z$ with the normal space of $M$ at $x=\tilde f(z)$. 
\end{definition}

All the definitions and constructions that we used in the case of $\mathcal R$-maps, can be adopted to the case of normal stable formal $\mathcal R$-maps in an obvious way, e.g., the definition of the classifying $\Gamma$-space of stable formal $\mathcal R$-maps $h\mathcal{M}_{\mathcal R}$ is obtained from the definition of $\mathcal{M}_{\mathcal R}$ by replacing proper $\mathcal R$-maps by proper  normal stable formal $\mathcal R$-maps. More precisely, let $hF\co \mathbf{Sset}^{op}\to \mathbf{Set}$ be the functor that associates to each simplicial set $C_{\bullet}$
the set of graphic normal stable formal $\r$-maps to $C_\bullet$. As above, there is a moduli space $hX_\bullet$
for $hF$ and its geometric realization is the space $h\mathcal{M}_\r$. 
Thus, every $n$-simplex in $h\M_\r$ comes with 
\begin{itemize}
\item a submanifold $V\subset \Delta^n_e\times \R^{\infty}$
whose inclusion is a graphic map $(u, \beta)$, and 
\item a normal stable formal $\mathcal R$-map $ F$ over $(u, \beta)$.
\end{itemize}

Finally, the space $h\M_\r$ can be defined to be the geometric realization of the partial sheaf $h\mathcal{F}_\r$ of
stable formal $\r$-maps; the partial sheaf $h\mathcal{F}_\r$ associates with a manifold $N$ the set of proper normal stable
formal graphic $\r$-maps to $N$.

\section{The canonical 
$\Gamma$-space structures on $\M_\mathcal{R}$ and $h\M_\r$}\label{ss:2.2} 

In this section we define a coherent operation 
on the space $\mathcal{M}_{\mathcal R}$ and a similar operation on $h\M_\r$. The coherence of the operation is formulated in terms of the Segal category $\Gamma$, which we recall here, for more details see \cite{Se} and \cite{BF}. 

The \emph{category $\Gamma$} consists of one object for each finite set $\mathbf{n}=\{1, ..., n\}$, with $n=0, 1, ...$, and one morphism $\mathbf{n}\to \mathbf{m}$ for each map $\theta$ of the set $\mathbf{n}$ to the set of subsets of $\mathbf{m}$ such that $\theta(i)$ is disjoint from $\theta(j)$ for each pair of distinct elements $i,j\in \mathbf{n}$. The composition corresponding to two maps $\theta_1$ and $\theta_2$ is the morphism corresponding to the map $i\mapsto \theta_2(\theta_1(i))$. Note that the category $\Gamma^{op}$ is isomorphic to the category of finite pointed sets $n^+=\{0,...,n\}$ and pointed maps; the base point in $n^+$ is $0$. 
A \emph{$\Gamma$-space} is a functor $A$ from  $\Gamma^{op}$ to the category of topological spaces such that
for any $n\ge 0$ the map 
\[
(i_1^*, \cdots, i_n^*)\co A(\mathbf{n})\to A(\mathbf{1}) \times \cdots \times  A(\mathbf{1})
\] 
is a homotopy equivalence, where $i_k\co \mathbf{1}\to \mathbf{n}$ is the inclusion $1\mapsto k$. 
The $0$-th space $A(\mathbf{0})$ of a $\Gamma$-space is required to be contractible. 
We note that the term $A(\mathbf{1})$ of a $\Gamma$-space $A$ is an $H$-space with operation
\[
    A(\mathbf{1})\times A(\mathbf{1}) \stackrel{\simeq}\longrightarrow A(\mathbf{2}) \stackrel{\mu^*}\longrightarrow A(\mathbf{1}),
\]
where $\mu\co \mathbf{1}\to \mathbf{2}$ is the morphism in $\Gamma$ defined by $\mu(1)=\{1, 2\}$. The 
identity element in $A(\mathbf{1})$ is an arbitrary point in the image of the unique morphism $A(\mathbf{0})\to A(\mathbf{1})$.
Slightly abusing notation, we will also say that $A(\mathbf{1})$ has a structure of a $\Gamma$-space. Similarly, we will use the notation $A$ both for the contravariant functor and for the space $A(\mathbf{1})$. Under these conventions, a $\Gamma$-space
is an H-space with a coherent operation. 


\begin{theorem}\label{th:2.9} Given a stable differential relation $\mathcal R$, the moduli spaces $\M_{\mathcal R}$ 
and $h\M_\r$ 
admit canonical $\Gamma$-space structures. 
\end{theorem}
\begin{proof}
For a non-negative integer $m$, let $X_n(\mathbf{m})\subset (X_n)^m$ be the set of ordered 
$m$-tuples of elements $V_i\in X_n$ such that 
$\sqcup V_i$ itself is an element 
of $X_n$. The space $\mathcal{M}_{\mathcal R}(\mathbf{m})$ is 
defined to be the geometric realization of the simplicial set 
$X_{\bullet}(\mathbf{m})\co [k]\mapsto X_k(\mathbf{m})$; 
the boundary and degeneracy operators of $X_{\bullet}(\mathbf{m})$ are 
respectively defined by means of actions of $\delta_i$ and $\sigma_j$ given in Lemma~\ref{l:2.2}. 

Given a morphism $\theta\co \mathbf{n}\to \mathbf{m}$ in the category $\Gamma$, there is a map of simplicial sets $\theta_X\co X_{\bullet}(\mathbf{m})\to X_{\bullet}(\mathbf{n})$ defined by the correspondence
 \[
 (V_1,\dots, V_m) \mapsto (V_{\theta(1)},\dots, V_{\theta(n)}),
 \]
 where each symbol $V_{\theta(i)}$ stands for the disjoint union 
 of objects $V_j$ indexed by $j\in \theta(i)$. The geometric realization $|\theta_X|$ of $\theta_X$ is a continuous map of topological spaces $\M_\r(\mathbf{m})\to\M_\r(\mathbf{n})$. It also immediately follows that compositions of morphisms $\theta$ in $\Gamma$ correspond to compositions of continuous maps $|\theta_X|$. The two axioms of $\Gamma$-spaces
are easily verified for $\M_{\mathcal R}$, see Theorem~\ref{th:5.6} 

The proof of Theorem~\ref{th:2.9} for $h\M_\r$ is similar. 
\end{proof}

\section{The cohomology theories $k^*_{\mathcal{R}}$ and $h^*_{\r}$}\label{ss:2.3}

In view of the Segal construction~\cite{Se}, the classifying $\Gamma$-space of $\mathcal{R}$-maps gives rise to a spectrum 
\[
   \mathbf{M}_{\mathcal R}\colon\quad \mathcal{M}_\mathcal{R}, B\mathcal{M}_{\mathcal{R}}, \ B^2\mathcal{M}_{\mathcal{R}}, \ B^3\mathcal{M}_{\mathcal{R}}, \dots,
\]
see \S\ref{asec:6} for a review.
Thus, every open stable differential relation $\mathcal R$ imposed on maps of dimension $d$ gives rise to a spectrum $\{B^i\mathcal{M}_{\mathcal R}\}$ of spaces indexed by $i\ge 0$, where we write $B^0\mathcal{M}_{\mathcal R}$ and $B^1\mathcal{M}_{\mathcal R}$ for $\mathcal{M}_{\mathcal R}$ and $B\mathcal{M}_{\mathcal R}$ respectively. The infinite loop space of the spectrum $\mathbf{M}_{\mathcal R}$ is denoted by $\Omega^{\infty}\mathbf{M}_{\mathcal R}$.

\begin{definition} The \emph{cohomology theory $k^*_{\mathcal R}$ of solutions of $\mathcal{R}$} is defined by
\[    
      k^n_{\mathcal R}(X)=[X, \Omega^{\infty-n}\mathbf{M}_{\mathcal R}],
\]
where $X$ is a topological space and $n$ is an integer. 
\end{definition}

Similarly, the spectrum of the $\Gamma$-space
$h\M_\r$ is denoted by $h\mathbf{M}_\r$, and the infinite loop space of $h\mathbf{M}_\r$  by $\Omega^{\infty}h\mathbf{M}_\r$.

\begin{definition} The \emph{cohomology theory $h^*_{\mathcal R}$ of stable formal solutions} is defined by 
\[    
      h^n_{\mathcal R}(X)=[X, \Omega^{\infty-n}h\mathbf{M}_{\mathcal R}],
\]
where  $X$ is a topological space and $n$ is an integer. 
\end{definition}

\begin{remark}\label{r:2.6} The spectrum of a $\Gamma$-space is connective \cite[Proposition 1.4]{Se}. 
Hence the $k$-cohomology groups and $h$-cohomology groups of a point are trivial in positive degrees. 
\end{remark}

In the rest of the section we construct a natural transformation $k^*_\r\to h^*_\r$ of cohomology theories. Let $M$ and $N$ be Riemannian manifolds, and 
\[
(u, \beta)\co M\to N\times \R^{\infty}
\]
a graphic $\r$-map such that the image of $\beta$ is in a finite dimensional space $\R^t$. Then  there is a  tangential stable formal $\r$-map defined by the family of map germs
 \begin{equation}\label{eq:5}
     F_x\co (T_xM, 0)\stackrel{\approx}\longrightarrow (M, x) \stackrel{u}\longrightarrow (N, y) \stackrel{\approx}\longrightarrow (T_{y}N, 0), 
 \end{equation}
where $y=u(x)$, and the two outer maps  identify open neighborhoods in manifolds with open neighborhoods in tangent spaces by means of exponential maps. The tangential finite representation $\{F_x\}$ gives rise to a normal finite representation
\[
    F_x^{\nu}\co  \R^{t} \longrightarrow \nu_x,
\]
where $\nu$ is the normal bundle of the embedding $(u, \beta)$ of $M$ into $N\times \R^t$, and $\nu_x$ is the fiber of $\nu$ over $x$.   Thus we associated with every graphic $\r$-map a normal stable formal $\r$-map (unique up to homotopy). Similarly, a concordance of graphic $\r$-maps corresponds to a concordance of normal stable formal $\r$-maps, and therefore for every compact manifold $N$ there is a well-defined natural map 
\[
     [N, \M_\r] \longrightarrow [N, h\M_\r]
\]
that associates to a concordance class of $\r$-maps to $N$ a concordance class of normal stable formal $\r$-maps to $N$.  We will see that this natural correspondence defines a natural transformation  
\[
\mathfrak{A}=\mathfrak{A}_{\mathcal{R}}\co k^*_{\mathcal R}\longrightarrow h^*_{\mathcal R}
\] 
of cohomology theories. 
In fact we will construct a model for the classifying space $B\M_\r$ so that it is a subspace of a model for $Bh\M_\r$. The corresponding inclusion of the (infinite) loop space $\Omega B\M_\r$ into the (infinite) loop space $\Omega Bh\M_\r$ is the very map classifying the natural transformation $\mathfrak{A}$.   

\begin{definition}\label{d:5} We say that an open stable differential relation $\mathcal R$ satisfies the \emph{b-principle} if the natural transformation $\mathfrak{A}_{\mathcal R}$ is an equivalence.
\end{definition}


\begin{remark}\label{r:4.3} The natural transformation $\mathfrak{A}_\r$ for the submersion differential relation imposed on oriented maps of dimension $d\ge 1$ is not an equivalence. In this case the space $\M_\r$ is the classifying space $\sqcup \BDiff M$ of Example~\ref{ex:1}. It is a topological monoid whose group completion is an infinite loop space classifying the cohomology theory $k^*_\r$. As a topological space, the infinite loop space classifying the cohomology theory $h^*_\r$ is homotopy equivalent to the infinite loop space of the Madsen-Tillmann spectrum $\Omega^{\infty}{\mathop\mathrm{MTSO(d)}}$; see Examples~\ref{ex:1}, \ref{e:3}. The Pontrjagin-Thom construction yields a map 
\[
   t\co \sqcup \BDiff M\longrightarrow \Omega^{\infty}{\mathop\mathrm{MTSO(d)}},
\] 
which in its turn, by the universal property of group completions, gives rise to the map 
\[
   \Omega B(\sqcup \BDiff M)\longrightarrow \Omega^{\infty}{\mathop\mathrm{MTSO(d)}}
\]
classifying the natural transformation $\mathfrak{A}_\r$. The zero homotopy group of the space on the left hand side is the group completion of the monoid $\pi_0(\sqcup \BDiff M)$, while the group of path components of the space on the right hand side is computed by Ebert in \cite{Eb}; these two groups are not isomorphic. Also recently Ebert computed the kernel of the induced homomorphism $t^*$ in rational cohomology groups~\cite{Eb1}. It is non-trivial for $d$ odd (but trivial for $d$ even).  

In the case $d=2$ the natural transformation $\mathfrak{A}_\r$ is an equivalence for an appropriately modified cohomology theory $k^*_\r$. This fact is known as the generalized Mumford conjecture; it was proved in the positive by Madsen and Weiss~\cite{MW}.  
\end{remark}

\begin{remark} Another open stable differential relation $\mathcal R$ for which $\mathfrak{A}_\r$ possibly  fails to be an equivalence is given by the minimal differential relation whose solutions include functions with Morse singularities of index $0$. In this case $k^*_\r$ is related to the diffeomorphism group of the sphere of dimension $d$; see \cite{Sae1}, \cite{Sa3}. 
\end{remark}

\section{$C$-valued partial $\Gamma$-sheaves}\label{sec:6}

Every open stable differential relation $\mathcal R$ imposed on maps of dimension $d$ gives rise to a spectrum $\mathbf{M}_\r=\{B^i\mathcal{M}_{\mathcal R}\}$ of spaces indexed by $i\ge 0$ by means of the Segal construction (\S\ref{asec:6}). In this section we define $C$-valued partial $\Gamma$-sheaves that we will need in section~\ref{s:GMTW} where we define (see Definition~\ref{d:2.12}) a nice model $\mathcal{B}\M_\r(\bullet)$ for the classifying $\Gamma$-space $B\M_\r(\bullet)$ of  $\M_\r(\bullet)$, and prove the equivalence of the two $\Gamma$-spaces.  Models for iterative classifying 
spaces $B^n\M_\r(\bullet)$ for $n>1$ can be 
constructed similarly; we will not need them however and omit their description. 

\begin{remark}
Determination of the classifying space of a monoid is a very old problem, with many approaches and views on the matter. In the current paper I choose to follow as close as possible the argument of Galatius-Madsen-Tillmann-Weiss in \cite{GMTW} which is based on the cocycle sheaf construction from the paper \cite{MW} of Madsen and Weiss.  There are two reasons for my choice. First, the constructions in the original ``The-four-author-paper" are widely known and therefore for many the presentation in the current section will be relatively easy to follow. And second, the language proposed in \cite{GMTW} is very convenient in the setting under consideration. 
We have to make a few modifications in the construction  though. First, we consider partial sheaves instead of sheaves; the necessity of partial sheaves should be clear from Example~\ref{ex:a1}. Second, we consider category valued $\Gamma$-sheaves, i.e., families of category valued sheaves with coherent operations. In particular, we consider $3$-simplicial sets $\Delta^{op}\times \Delta^{op}\times \Delta^{op}\to \mathbf{Set}$. This is necessary to make sure that all infinite loop structures under consideration are compatible. 
\end{remark}

Let $\mathcal{F}$ and $\mathcal{G}$ be two contravariant functors on $\Gamma$ with values in a category $\mathcal{C}$. Given an equivalence relation on objects of $\mathcal{C}$, we say that a natural transformation $\psi\co \mathcal{F}\to \mathcal{G}$ is a \emph{weak equivalence} if $\psi(\mathbf{m})$ is an equivalence for all $\mathbf{m}$. In particular, a transformation $\psi$ of $\mathbf{Top}$-valued functors is a weak homotopy equivalence if each $\psi(\mathbf{m})$ is a  homotopy equivalence of topological spaces. Similarly, if the values of $\mathcal{F}$ and $\mathcal{G}$ are topological categories, then $\psi$ is a weak equivalence if the morphism $|\psi(\mathbf{m})|$ is a homotopy equivalence of topological spaces for each $\mathbf{m}$. Two contravariant functors $\mathcal{F}$ and $\mathcal{G}$ on $\Gamma$ are \emph{weakly equivalent} if there exists a sequence of weak equivalences of contravariant functors
\[
   \mathcal{F}\longrightarrow \mathcal{H}_0 \longleftarrow \mathcal{H}_1 \longrightarrow\cdots \longrightarrow \mathcal{H}_{n-1}\longleftarrow \mathcal{H}_n\longrightarrow \mathcal{G}.
\]   
The category $\mathbf{PSet}$ of partial sets is defined to be a small category whose objects are sets and whose morphisms $S\to T$ are partially defined maps, i.e., a morphism in $\mathbf{PSet}$ is a map from a subset of $S$ to $T$.  The category of smooth manifolds and smooth maps is denoted by $\mathcal{X}^{op}$. A \emph{partial sheaf} is a functor $\mathcal{F}\co \mathcal{X}^{op}\to \mathbf{PSet}$ that satisfies the Sheaf Property: for any locally finite covering $\{U_i\}$
of a manifold $M$ and any sections $s_i\in \mathcal{F}(U_i)$ that agree on common domains---that is, $s_i=s_j$ on $U_i\cap U_j$---there exists a section $s\in \mathcal{F}(M)$ with $s=s_i$ on $U_i$ for all $i$. 
Many sheaf constructions extend to constructions for partial sheaves, for details see~\ref{a:1.5}.

Given a category $\mathcal{C}$,  a $\mathcal{C}$-valued partial \emph{$\Gamma$-sheaf} is a functor 
from $\Gamma^{op}$ to the category 
of $\mathcal{C}$-valued partial sheaves, see \S\ref{a:1.5}. Thus, a $\mathcal{C}$-valued partial $\Gamma$-sheaf consists of a family 
$\{\mathcal{F}(\mathbf{m})\}$ of $\mathcal{C}$-valued partial sheaves together with structure morphisms
$\mathcal{F}(\varphi)\co \mathcal{F}(\mathbf{m})\to \mathcal{F}(\mathbf{n})$, one for each morphism 
$\varphi\co \mathbf{n}\to \mathbf{m}$
in $\Gamma$. The geometric realization of a $\Gamma$-sheaf is a functor 
$|\mathcal{F}|$
on the category $\Gamma^{op}$ with 
\[
   |\mathcal{F}|(\mathbf{m})=|\mathcal{F}(\mathbf{m})| \qquad \mathrm{and}\qquad |\mathcal{F}|(\varphi)= |\mathcal{F}(\varphi)|.
\]

\section{The Galatius-Madsen-Tillmann-Weiss argument}\label{s:GMTW}

The already mentioned model for $B\M_\r(\bullet)$ is defined in terms of 
\emph{graphic maps of order $1$}, i.e., smooth embeddings 
\[
   (f, \alpha, \beta)\colon V\hookrightarrow N\times \R\times \R^{\infty-1}
\]
such that $(f, \alpha)$ is proper and the closure of $\beta(V)$ is compact. {\ccc For $x\in N$, the fiber $f^{-1}(x)$ is said to be \emph{broken} if the restriction $\alpha|_{f^{-1}(x)}\co f^{-1}(x)\to \R$ is not surjective, i.e., if $\alpha(f^{-1}(x))\ne \R$. The 
graphic map 
$(f, \alpha, \beta)$ of order $1$ is \emph{broken} if
$f^{-1}(x)$ is broken for all $x\in N$. }

\begin{definition}\label{d:2.12}
For a non-negative integer $\mathbf{m}$ and 
a smooth manifold $N$, 
let $D_\r(\mathbf{m})(N)$ denote the 
set of $m$-tuples of disjoint submanifolds $V_1,..., V_m$ in $N\times \R\times \R^{\infty-1}$
such that the inclusion of $\sqcup V_i$ is a broken graphic map $(f, \alpha, \beta)$ of order $1$  and 
the projection of $\sqcup V_i$ to $N$ is an $\r$-map. Then $D_\r(\mathbf{m})$
is a set valued partial sheaf for each $\mathbf{m}$; given a map $g\co N'\to N$, the map $D_\r(\mathbf{m})(g)$
is defined by 
\[
(f, \alpha, \beta)\mapsto (g\times \id_\R\times \id_{\R^{\infty-1}})^*(f, \alpha, \beta)
\] 
on the subset of $D_\r(\mathbf{m})(N)$ of graphic
maps $(f, \alpha, \beta)$ with $f$ transverse to $g$. Furthermore, 
$D_\r(\bullet)$ is a set valued partial $\Gamma$-sheaf; the 
$\Gamma$-structure morphisms $D_\r(\mathbf{m}\to \mathbf{n})$ are defined 
as in the proof of Theorem~\ref{th:2.9}.
The geometric realization 
$\Gamma^{op}\to \mathbf{Top}$ of the partial $\Gamma$-sheaf $D_\r(\bullet)$ is denoted by $\mathcal{B}\M_\r(\bullet)$, explicitly $\mathcal{B}\M_\r(\mathbf{k})=|D_\r(\mathbf{k})|$.
\end{definition}

An element in $D_\r(\mathbf{1})(N)$ can be depicted as in Figure~\ref{fig:1a}; here and below the $\R^{\infty-1}$ component in the figures is absent, and in captions we omit $\r$ and $\mathbf{m}=\mathbf{1}$. 

\begin{figure}[ht]
\centering
\begin{minipage}{.5\textwidth}
\centering
\includegraphics[height=2in]{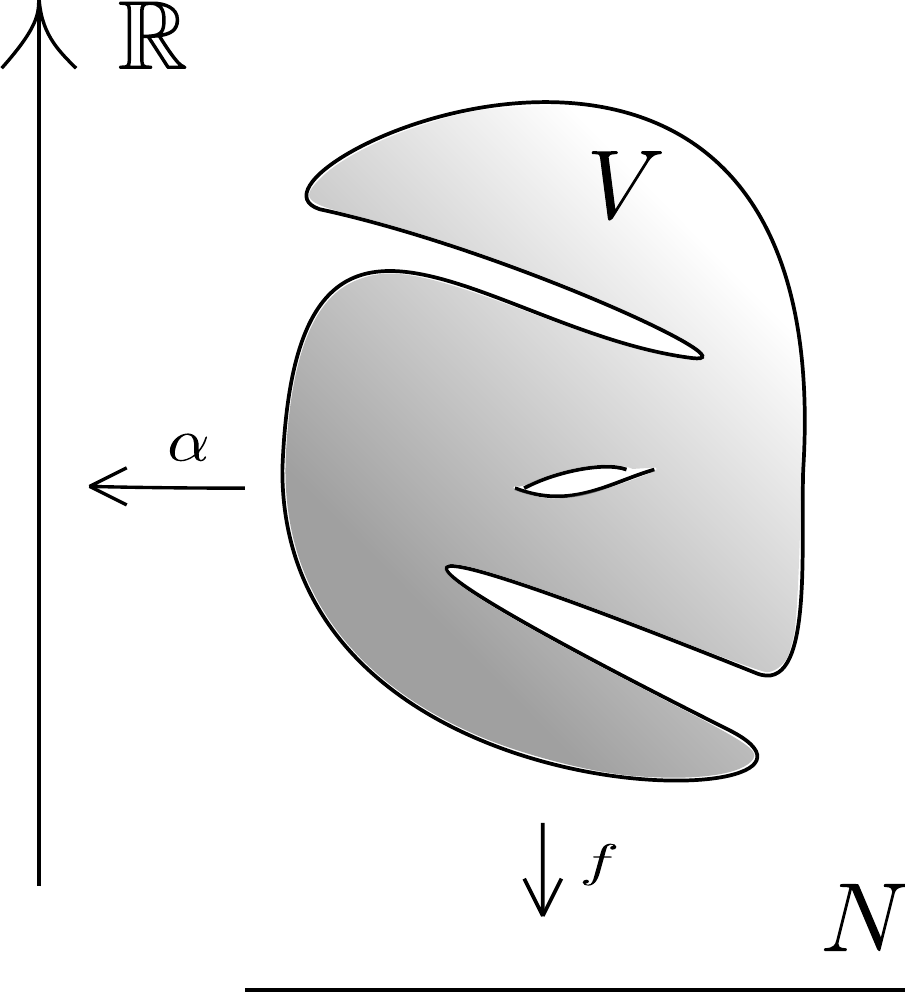}
\caption{An element in the set $D(N)$.}
\label{fig:1a}
\end{minipage}%
\begin{minipage}{.5\textwidth}
  \centering
\includegraphics[height=2in]{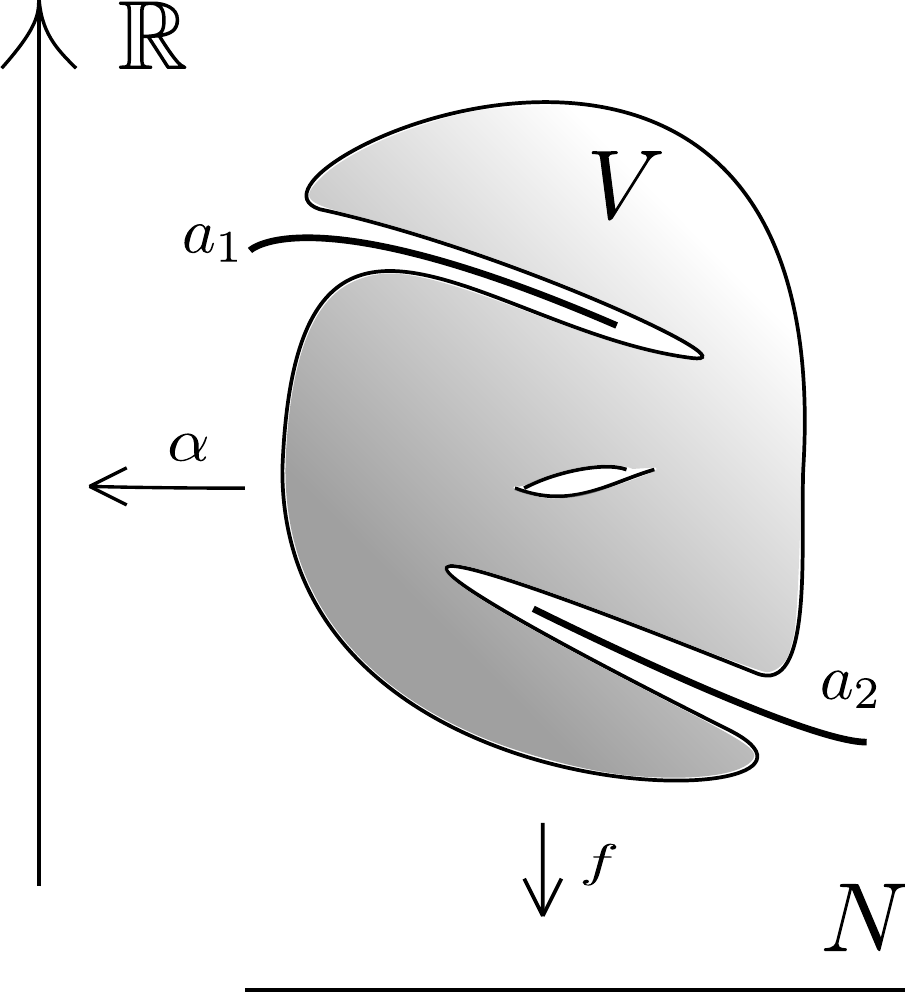}
\caption{An element in the set $\beta D^{\pitchfork}(N)$.}
\label{fig:2a}
\end{minipage}
\end{figure}

\begin{theorem}\label{th:6.2} The functor $\mathcal{B}\M_\r(\bullet)$  is a $\Gamma$-space. 
In fact, it is weakly homotopy equivalent to the classifying 
$\Gamma$-space $B\M_\r(\bullet)$. 
\end{theorem}

The proof of Theorem~\ref{th:6.2}  is a version for $\mathcal{C}$-valued partial $\Gamma$-sheaves of  
a technical argument of Galatius, Madsen, Tillmann and Weiss in \cite[Section 4]{GMTW}. The reader may skip it on the first reading as the argument will not be needed elsewhere in the paper.

\begin{proof}[Proof of Theorem~\ref{th:6.2}]
We will introduce a poset valued partial 
$\Gamma$-sheaf $D^\pitchfork_\r(\bullet)$ and a category valued partial $\Gamma$-sheaf
$C_\r(\bullet)$
with weak homotopy equivalences
\[
   \mathcal{B}\M_\r(\bullet)=|D_\r(\bullet)| \longleftarrow |\beta D^\pitchfork_\r(\bullet)|  \quad \textrm{and}\quad  B|D^\pitchfork_\r(\bullet)|
   \longleftarrow B|C_\r(\bullet)| \longrightarrow B\M_\r(\bullet),  
\]
of $\mathbf{Top}$-valued functors on $\Gamma^{op}$, where $\beta D^\pitchfork_\r(\mathbf{m})$ is the cocycle partial sheaf of $D^{\pitchfork}_\r(\mathbf{m})$, see Definition~\ref{a:2}.  In view of the 
Madsen-Weiss construction~\cite{MW} of a natural equivalence of $|\beta D^\pitchfork_\r(\bullet)|$
and $B|D^\pitchfork_\r(\bullet)|$, this will imply Theorem~\ref{th:6.2}.

\begin{remark} For all partial sheaves $\mathcal{F}$ below, it is easily verified that $|\mathcal{F}|$ is 
connected and the homotopy classes of maps $S^n\to |\mathcal{F}|$ of spheres are in bijective 
correspondence with $\pi_n|\mathcal{F}|$ for all $n\ge 0$, {\ccc see the proof of Corollary~\ref{c:10.4}}. In particular, in all cases under consideration, 
in order to prove that a map $f\co \mathcal{F}\to \mathcal{G}$ of partial sheaves is a homotopy equivalence, 
it suffices to show that $f_*\co [S^n, |\mathcal{F}|]\to [S^n, |\mathcal{G}|]$ is an isomorphism. 
Consequently, it suffices to show that the map $f$ induces an isomorphism of the sets of 
concordance classes of elements in $\mathcal{F}(S^n)$ and $\mathcal{G}(S^n)$.
\end{remark}

\begin{definition} For a smooth manifold $N$ and a positive integer $\mathbf{m}$, let $D^{\pitchfork}_\r(\mathbf{m})(N)$ denote 
the set of pairs of a broken graphic map $(f,\alpha,\beta)\in D_\r(\mathbf{m})(N)$ of order $1$ and a function 
$a$ on $N$ such that $\alpha(f^{-1}(x))\ne a(x)$ for all $x\in N$. 
Then $D^\pitchfork_\r(\mathbf{m})$ is a sheaf of posets in which $(f, \alpha, \beta, a)\le (f', \alpha', \beta', a')$
if $(f,\alpha, \beta)=(f', \alpha', \beta')$, $a\le a'$ and the set $\{ a(x)=a'(x) \}$ is open in $N$, see Figure~\ref{fig:3a}.
Also $D^{\pitchfork}_\r(\bullet)$ is a poset valued partial $\Gamma$-sheaf.
\end{definition}

{\ccc
Recall that the cocycle sheaf $\beta F$ of a category valued sheaf $F$ is a set valued sheaf defined as follows. Let $J$ be an uncountable indexing set $J$. 
For a family of sets $\{U_j\}$ indexed by a finite subset $R\subset J$, let $U_R$ denote the intersection $\cap U_j$ over all $j\in R$. 

\begin{definition} \label{a:2}
A \emph{$\mathcal{C}$ bundle} over a manifold $X$ is a triple $(U_R, c_R, \varphi_{RS})$ of 
\begin{itemize}[nosep]
\item a locally finite open cover $U=\{U_j\}$ of $X$ by open subsets indexed by $J$, 
\item a continuous bundle of objects $c_R\co U_R\to \Ob\mathcal{C}$ for each finite subset $R\subset J$, 
\item a continuous transition function $\varphi_{RS}\co U_S\to \Mor\mathcal{C}$, where $\varphi_{RS}(x)$ is a morphism from $c_S(x)$ to $c_R(x)$ for each $x\in U_S$ and finite non-empty subsets $R\subseteq S\subset J$
\end{itemize} 

such that  
 each $\varphi_{RR}(x)$ is the identity morphism of $c_R(x)$, and 
 $\varphi_{RT}=\varphi_{RS}\circ \varphi_{ST}$ over $U_T$ for all $R\subseteq S\subseteq T$ of finite non-empty subsets of $J$. The cocycle sheaf $\beta F$ of a category valued sheaf $F$ associates with a manifold $X$ the set $\beta F$ of $C$ bundles over $X$. Similarly, the cocycle sheaf $\beta F(\bullet)$ of a category valued partial $\Gamma$-sheaf $F(\bullet)$ is a set valued partial $\Gamma$-sheaf.  The argument in the proof of \cite[Theorem A.12]{MW} shows that $|\beta F(\bullet)|\simeq B|F(\bullet)|$ for any category valued partial $\Gamma$-sheaf $F(\bullet)$.  
\end{definition}
}

In particular, we have $|\beta D^{\pitchfork}_\r(\bullet)|\simeq B|D^\pitchfork_\r(\bullet)$|.

{\ccc For a closed subset $A$ of a manifold $X$, let $s$ be a germ section over $A$ of a sheaf $F$, i.e., an element in the colimit of the sets $F(U)$ of sections over open neighborhoods $U$ of $A$.  The subset of sections in $F(X)$ that agree with $s$ near $A$ is denoted by $F(X, A; s)$. Two such sections are concordant relative to $A$ if there is a concordance  whose germ near $A$ is constant. The set of concordance classes relative to $A$ is denoted by $F[X, A; s]$. 

\begin{lemma}[Madsen-Weiss, \cite{MW}]\label{le:8.6} A map $\alpha\co F_1\to F_2$ of sheafs is a week equivalence if the induced map $F_1[X, A; s]\to F_2[X, A; \alpha(s)]$ is surjective for all pairs $X, A$. 
\end{lemma}
}

\begin{proposition}\label{p:2.14} The forgetful map $|\beta D^\pitchfork_\r(\bullet)|\to |D_\r(\bullet)|$ is a weak homotopy equivalence.
\end{proposition}
\begin{proof} The proof of Proposition~\ref{p:2.14}
is similar to those of \cite[Proposition 4.2.4]{MW} and \cite[Proposition 4.2]{GMTW}.
Given an arbitrary section $(f, \alpha, \beta)$ in $D_\r(\mathbf{m})(N)$, 
{\ccc for each point $x\in N$ the set $\alpha(f^{-1}(x))$ does not hit a value, say $a_x\in \R$, since $(f, \alpha, \beta)$ is broken. In fact,  $a_x\notin \alpha(f^{-1}(y))$ for all $y$ close to $x$ since $x$ is away from the set $f(\alpha^{-1}(a_x))$ which is closed in $N$.} Hence, there are a covering $\{U_i\}$ of $N$ by open sets $U_i\subset N$ 
and a family of functions $a_i\co U_i\to \R$ such that 
$\alpha(f^{-1}(x))\ne a_i(x)$ for all $x\in U_i$. We may also assume that
the set $\{\, x\, |\, a_i(x)=a_j(x)\,\}$ is open in $N$ for all $i, j$. 
For any subset $R$ of indices of the covering $\{U_i\}$, let $U_R$ denote the intersection of the sets $U_i$ with $i\in R$. Let 
$a_R$ denote the function on $U_R$ given by $\min\{\ a_i(x)\ |\ i\in R\ \}$. Finally, for indexing sets $R\subseteq S$, 
denote the triple $(f^{-1}(U_S), a_S, a_R|U_S)\in N_1D^{\pitchfork}(U_S)$ by $\varphi_{RS}$. Then 
$(U_R, a_R, \varphi_{RS})$ is a lift of $(f, \alpha, \beta)$ to $\beta D^{\pitchfork}_\r(\mathbf{m})(N)$. This
proves the surjectivity of maps $[N, |\beta D^\pitchfork_\r(\mathbf{m})|]\to [N, | D_\r(\mathbf{m})|]$. The 
injectivity is proved similarly, see  Lemma~\ref{le:8.6}. %
%
\end{proof}

\begin{figure}[ht]
\centering
\begin{minipage}{.5\textwidth}
\centering
\includegraphics[height=2in]{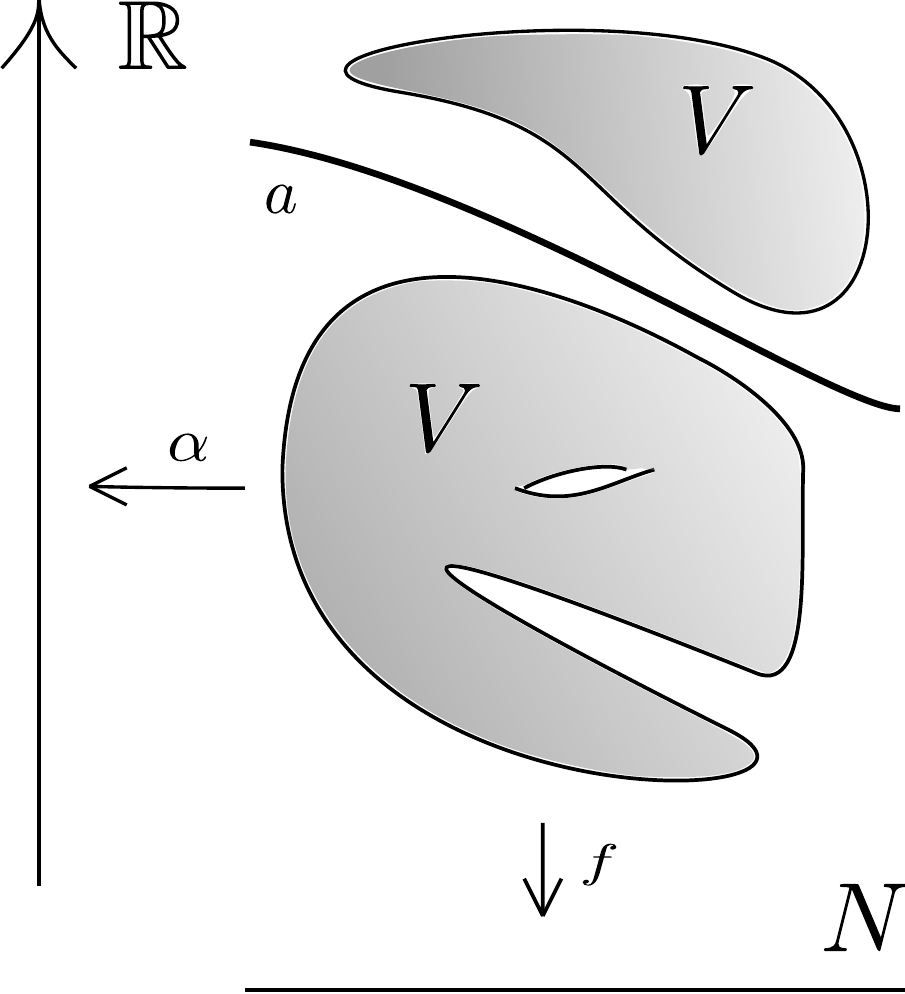}
\caption{An element in the poset $D^{\pitchfork}(N)$.}
\label{fig:3a}
\end{minipage}%
\begin{minipage}{.5\textwidth}
  \centering
\includegraphics[height=2in]{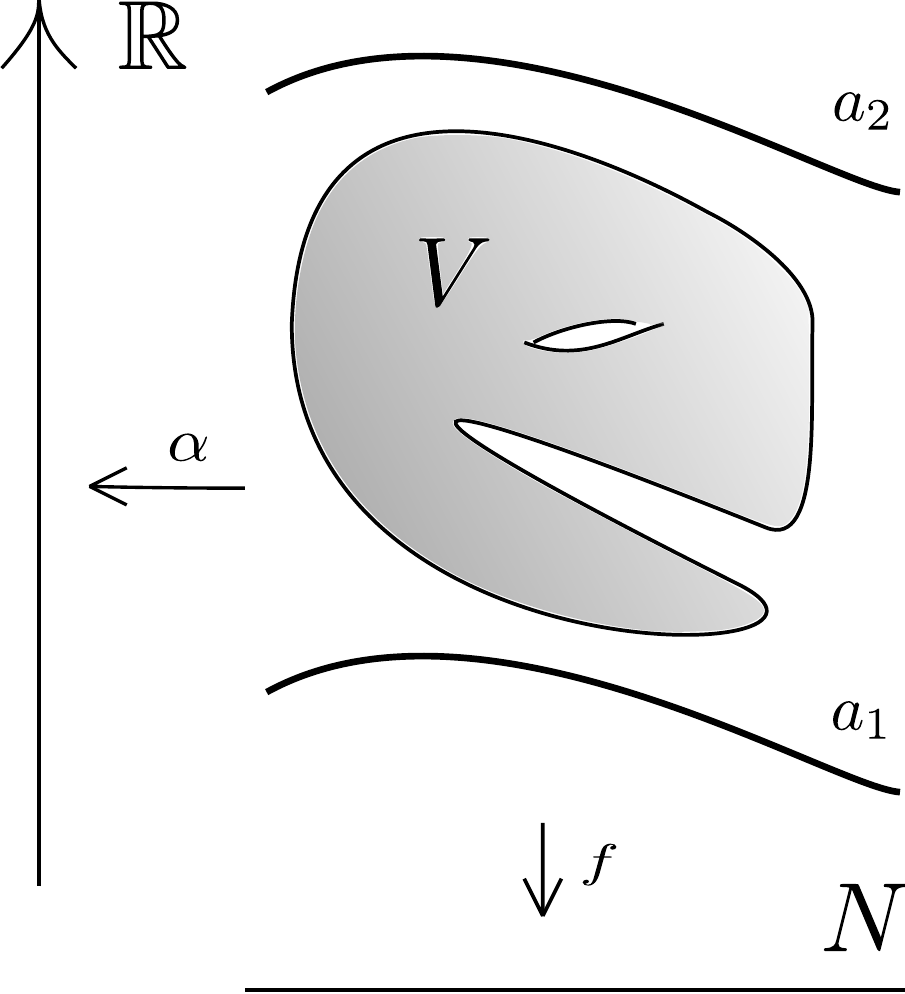}
\caption{A morphism in the category $C(N)$.}
\label{fig:4a}
\end{minipage}
\end{figure}

\begin{definition} For a manifold $N$ and a non-negative integer $\mathbf{m}$, 
let $C_\r(\mathbf{m})(N)$ denote the small category whose set 
of objects is the set of smooth functions on $N$, and whose set
of morphisms from $a_0$ to $a_1$ is the set of tuples  of $m$ smooth disjoint submanifolds of $N\times \R\times \R^{\infty-1}$ such that for the inclusion 
\[
     V_1\sqcup \cdots \sqcup V_m\xrightarrow{(f, \alpha, \beta)} N\times \R\times \R^{\infty-1}
\]
the map $f$ is a proper $\mathcal R$-map, the closure of the image of 
$\beta$ is compact, and $\alpha$ is a map with $\alpha(f^{-1}(x))$ in $(a_0(x)+\varepsilon(x), a_1(x)-\varepsilon(x))$ 
for each $x$ and some strictly positive function $\varepsilon(x)$. We require that the set $\{\, x\, |\, a_0(x)=a_1(x)\, \}$ is open. It follows that $V$ is empty if $a_0\equiv a_1$. Clearly $C_\r(\mathbf{m})$ is a category valued partial sheaf for each $\mathbf{m}$, and $C_\r(\bullet)$ is a category valued partial $\Gamma$-sheaf. 
\end{definition}

Recall that every poset is a category with morphisms given by inequalities $a\le b$. In particular, there is a map $\eta\co D^{\pitchfork}_\r(\bullet)\to C_\r(\bullet)$ of category valued partial $\Gamma$-sheaves. In terms of objects the map of categories 
\[
    D^{\pitchfork}_\r(\mathbf{m})(N)\to C_\r(\mathbf{m})(N)
\]
is the projection $(f, \alpha, \beta; a)\mapsto a$. For a map $f\co V\to N$, a morphism $(f, \alpha, \beta; a\le a')$ maps to its restriction to $V\cap (N\times [a_0, a_k]\times \R^{\infty-1})$, where $N\times [a_0, a_1]$ stands for the subspace of $N\times \R$ that consists of the union of all slices $\{x\}\times [a_0(x), a_1(x)]$ with $x\in N$.

\begin{proposition}\label{p:7.11} The map
$B\eta(\bullet)\co B|D^{\pitchfork}_\r(\bullet)|\to B|C_\r(\bullet)|$ is a weak homotopy equivalence of $\mathbf{Top}$-valued functors on $\Gamma^{op}$.
\end{proposition}
\begin{proof}  The proof of Proposition~\ref{p:7.11} is similar to that of \cite[Proposition 4.3]{GMTW}. For a category $X$, let $N_\bullet X$ denote the nerve of $X$.
Then an element in $N_kC_\r(\mathbf{m})(N)$ is represented by $k+1$ functions $a_0\le \cdots \le a_k$ and 
a manifold $V\subset N\times \R\times \R^{\infty-1}$ away of $\{x\}\times \{a_i(x)\}\times \R^{\infty-1}$ for all $x$ and $i$; the manifold $V$ is actually inside of the union of slices $\{x\}\times (a_0(x)\times a_k(x))\times \R^{\infty-1}$. The components $(a_0(x), a_k(x))$ can be stretched to $(-\infty, \infty)$ so that the same functions and the same $V$ represent an element in 
$N_kD^{\pitchfork}_\r(\mathbf{m})(N)$. This shows that $N_k\eta(\mathbf{m})$ is homotopy surjective.  The 
injectivity is proved similarly, see Lemma~\ref{le:8.6}. 
\end{proof}

Recall that by definition the space $\M_\r$ is the geometric realization of a partial set valued sheaf $\mathcal{F}_\r$. 
Similarly, for a manifold $N$, let $\mathcal{F}_\r(\mathbf{m})(N)$ denote the set of submanifolds $V\subset N\times \R^{\infty}$ such that the inclusion is a graphic $\r$-map, and $V$ is the disjoint union of $m$ manifolds $V_1,..., V_m$. Then the newly defined $\mathcal{F}_\r$ is a partial set valued $\Gamma$-sheaf. 
For each $\mathbf{m}$, the nerve $N_\bullet C_\r(\mathbf{m})$ and $\mathcal{F}_\r(\mathbf{m}\times \bullet)$
are partial sheaves with values in semi-simplicial sets. Furthermore, $N_\bullet C_\r(\bullet)$ and $\mathcal{F}_\r(\bullet\times \bullet)$
are partial $\Gamma$-sheaves with values in semi-simplicial sets, and there is a forgetting map 
$\gamma\co N_\bullet C_\r(\bullet)\to \mathcal{F}_\r(\bullet\times \bullet)$; under this forgetting map, an element $\{a_0\le \cdots \le a_n, V\}$ of $N_nC_\r(\mathbf{m})(N)$ maps to the element $\{V\cap [a_0, a_1], ..., V\cap [a_{n-1}, a_n]\}$ of $\mathcal{F}(\mathbf{m\times n})(N)$.

\begin{figure}[ht]
\centering
\begin{minipage}{.5\textwidth}
\centering
\includegraphics[height=2in]{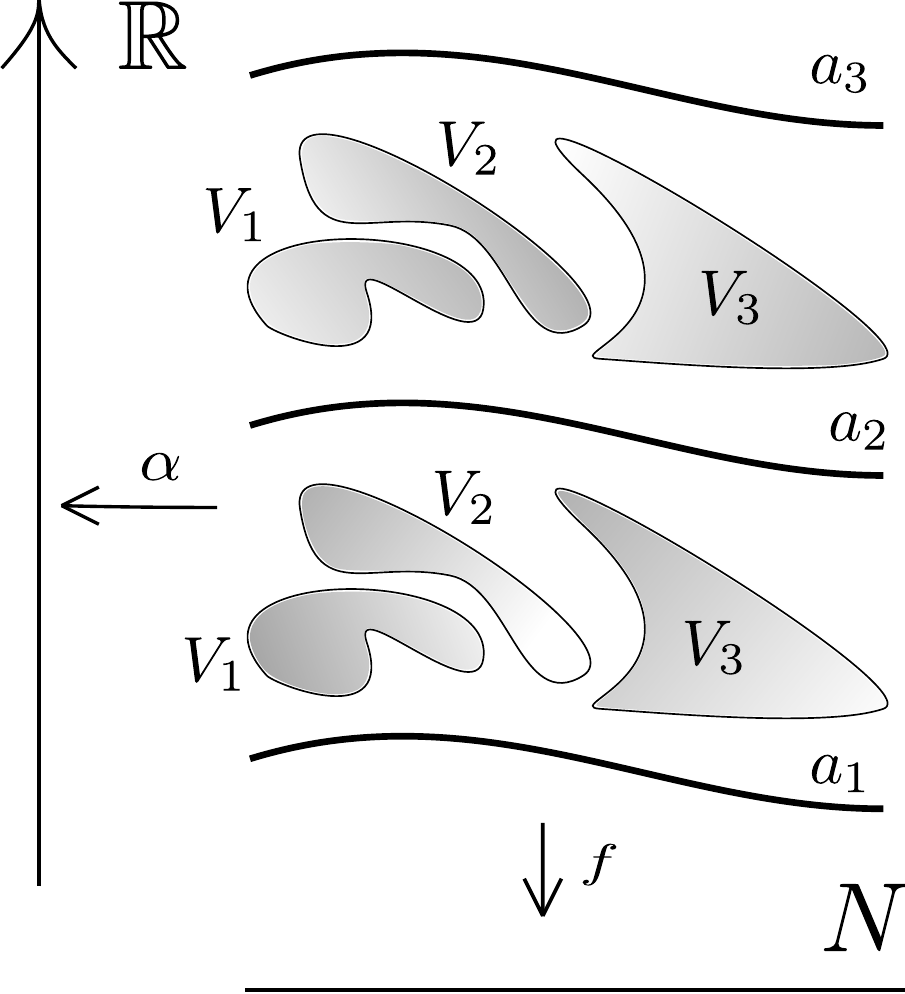}
\caption{An element in the set $N_n C(\mathbf{m})(N)$.}
\label{fig:5a}
\end{minipage}%
\begin{minipage}{.5\textwidth}
  \centering
\includegraphics[height=2in]{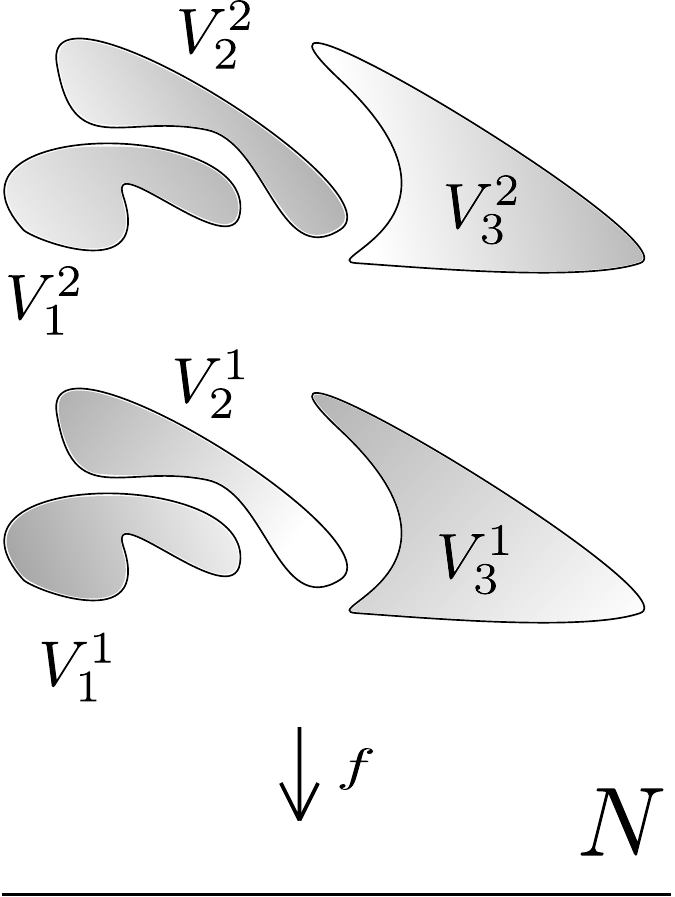}
\caption{An element in the set $\mathcal{F}(\mathbf{m}\times \mathbf{n})(N)$.}
\label{fig:6a}
\end{minipage}
\end{figure}

\begin{lemma} The map $B|\gamma|\co B|C_\r(\bullet)|\to B\M_\r(\bullet)$ is a weak homotopy equivalence.  
\end{lemma}
\begin{proof} Recall that $\M_\r(\mathbf{m})=|\mathcal{F}_\r(\mathbf{m}\times \bullet)|$. Thus it 
suffices to observe that for $n\ge 0$, the map $N_n\eta(\mathbf{m})$ is a homotopy equivalence of the spaces of 
$n$-simplices $N_n|C_\r(\mathbf{m})|=(N_1|C_\r(\mathbf{m})|)^n$ in $B|C_\bullet(\mathbf{m})|$ and $N_n\M_\r(\mathbf{m})=\M_\r(\mathbf{m}\times \mathbf{n})$
in $B\M_\r(\mathbf{m})$. 
\end{proof}
This completes the proof of Theorem~\ref{th:6.2}.
\end{proof}

\section{The classifying $\Gamma$-space $Bh\M_\r$} \label{ss:5.2}

The constructions in sections~\ref{sec:6}, \ref{s:GMTW} are also applicable in the case of moduli spaces of normal stable formal solutions. In particular, 
for a non-negative integer $\mathbf{m}$ and a smooth manifold $N$, let $hD_\r(\mathbf{m})(N)$ denote the 
set of pairs $(V, F)$ of 
\begin{itemize}
\item a submanifold $V=V_1\sqcup  ... \sqcup V_m$ in $N\times\R\times \R^{\infty-1}$ such that the inclusion 
of $V$ is a broken graphic map of order $1$, and
\item a normal stable formal $\r$-map $F$ coving the projection of $V$ to $N$. 
\end{itemize}

Then $hD_\r(\bullet)$ is a set valued partial 
$\Gamma$-sheaf.  The proof of Theorem~\ref{th:9.1a} is similar to that of  Theorem~\ref{th:6.2}.

{\ccc
\begin{theorem}\label{th:9.1a} The geometric realization of $hD_\r(\bullet)$ is a $\Gamma$-space weakly homotopy equivalent  to the classifying 
$\Gamma$-space $Bh\M_\r(\bullet)$. 
\end{theorem}
}


\begin{figure}[ht]
\centering
\begin{minipage}{.5\textwidth}
\centering
\includegraphics[height=2in]{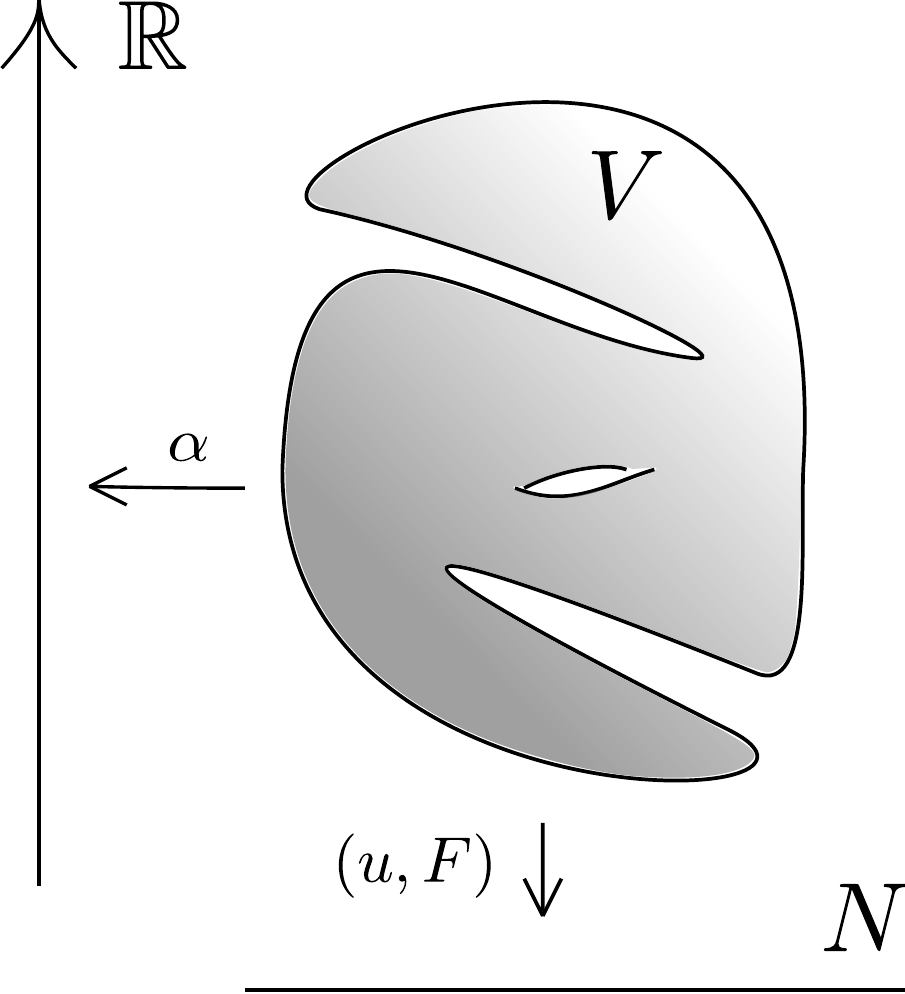}
\caption{An element in the set $hD(N)$.}
\label{fig:7a}
\end{minipage}%
\begin{minipage}{.5\textwidth}
  \centering
\includegraphics[height=2in]{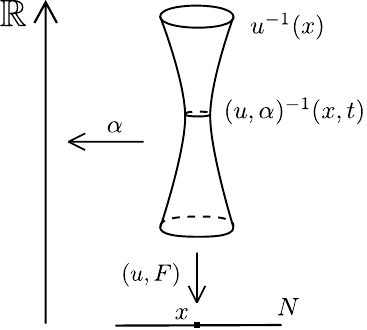}
\caption{A fiber of a map in the set $hE(N)$.}
\label{fig:8a}
\end{minipage}
\end{figure}

In this section we give another model $|hF_\r|$ for the $\Gamma$-space $Bh\M_\r(\bullet)$, see Definition~\ref{d:7.3}. 
In the case $d>0$ we need a preliminary discussion, and until Definition~\ref{d:7.3}, we will assume that $d>0$. In the case $d>0$ we will construct an intermediate space $|hE_\r|$ and show that there are inclusions of geometric realizations:
\begin{equation}\label{eq:11.1b}
           |hD_\r| \stackrel{\subset}\longrightarrow |hE_\r| \stackrel{\supset}\longleftarrow |hF_\r|. 
\end{equation}
Informally, all the spaces under consideration classify certain graphic normal stable formal solutions, i.e., certain pairs of a graphic map $(u, \alpha, \beta)\co V\to N\times \R\times \R^{n-1}$ and a normal stable formal solution $F$ covering $u$. The
space $|hD_\r|$ classifies those pairs for which $(u, \alpha, \beta)$ is 
broken. For the space $|hE_\r|$ this property is replaced with the property that each regular fiber of the proper map $(u, \alpha)$ is a closed manifold cobordant to the empty manifold. Finally, for the space $|hF_\r|$ we add an additional condition that $F$ is represented by $du$. In order to show that the inclusions in (\ref{eq:11.1b}) are homotopy equivalences, we will observe that each of the classifying spaces are path connected H-spaces. Therefore, in view of Corollary~\ref{c:10.4}, we will only need to show that 
for each closed manifold $X$ the inclusions induce isomorphisms of concordance classes     
\[     hD_\r[X] \simeq hE_\r[X] \simeq hF_\r[X].
\]
By definition, an element of each of these sets is a concordance class of certain graphic \emph{normal} stable formal solutions. Since $X$ is compact, these classes are in canonical bijective correspondence with graphic \emph{tangential} stable formal solutions. For this reason, we will often identify corresponding normal and tangential stable formal solutions. 
  
To carry out our program, we will need a preliminary lemma, Lemma~\ref{l:8.1}. It shows that near a generic fiber of $(u, \alpha)$ the formal structure $F$ can be integrated. We will show later that a generic fiber with an integral structure $F$ near it can be broken, see the proof of Theorem~\ref{th:7.1}.  

Let $\mathbf{V}$ denote a graphic $\r$-map $(V, u, \alpha, \beta)$ of order $1$ to a smooth manifold $N$ together 
with a tangential stable formal $\r$-map $F$ covering $u$. We say that a value $(x,t)$ in $N\times \R$ is \emph{regular} if it is a regular value of $(u, \alpha)$ and over the inverse image $M$ of $(x, t)$ each s-germ $F$ is regular, i.e., the differential of each map germ in the family is surjective. In this case $M$ is called a \emph{regular} fiber of $\mathbf{V}$. We note that for a generic $\mathbf{V}$, the set of regular values of $\mathbf{V}$ is dense and open in $N\times \R$. 

\begin{lemma}\label{l:8.1} Let $M$ be a regular fiber of $\mathbf{V}$. Then $F$ is homotopic through tangential stable formal $\r$-maps to a tangential stable formal $\r$-map that {\ccc coincides with $F$ outside of a given neighborhood of $M$ and} coincides with $du$ over a {\ccc smaller} neighborhood of $M$.
\end{lemma}
\begin{proof} Let $(x, t)$ be a regular value of $\mathbf{V}$ and $M$ the fiber over $(x,t)$. Then there exists a product neighborhood $U$ of $(x,t)$ in $N\times \R$ that consists of regular values. If $U$ is sufficiently small, then the inverse image of $U$ with respect to $(u, \alpha)$ can be identified with $M\times U$ in such a way that the restriction of the map $(u, \alpha)$ to $M\times U$ is given by the projection onto the second factor. Then over $M\times U$ the family  of tangential stable map germs $d_xu$ is represented by the family of constant map germs $T_xM\times \R\to \R^0$. 
On the other hand, 
since $M$ is compact, the family $F|M$ of stable map germs can be represented by a continuous family of regular map germs 
\[
F_x\co T_xM\times\R^{n+1}\to \R^{n}, \qquad \textrm{parametrized by } x\in M,
\] 
 for an appropriate finite $n$ which can be chosen to be the same for all $x$. We will construct a homotopy of the family of map germs $\{F_x\}_{x\in M}$ to the family of constant map germs $d_xu\times \id_{\r^n}$. To begin with, we observe that there exists a canonical linear homotopy of map germs $F_x$ to linear map germs $dF_x$. Thus we may assume that for each $x\in M$ the map germ $F_x$ is linear. 

Choose a Riemannian metric on $M$. Then for each $x\in M$, there is a unique $n$-tuple $\tau_F$ of vectors $v_1,..., v_n$ in $T_xM\times \R^{n+1}$ normal to the kernel of the linear surjective map $F_x$ such that $F_x$ takes $v_i$ to the $n$-th basis vector in $\R^n$ for each $i$.    
In fact, the family $F|M$ of linear map germs can be identified with an $n$-tuple of linearly independent sections $\tau_F$ of the vector bundle $TM\times \R^{n+1}$. Choose a homotopy of $\tau_F$ through linearly independent $n$-tuples of sections to the tuple of constant $n$ sections given by the last $n$ basis vectors of $\R^{n+1}$. The corresponding homotopy of $F|M$ takes it to the family of map germs 
\[
       d_xu\times \id_{\R^n}\co (T_xM\times \R)\times \R^{n}\longrightarrow \R^0\times \R^n,
\]
which by definition is stably equivalent to the family of constant map germs $d_xu$. Furthermore, the family $F$ of s-germs over $V$ is homotopic to a family of s-germs that near $M$ is represented by the constant map germs, {\ccc and the homotopy can be chosen to be trivial outside of a given neighborhood of $M$.}  This homotopy is the desired one.     
\end{proof}

Next we will describe a construction of concordances breaking fibers of graphic maps of order $1$. Recall that we still assume that $d>0$. Furthermore, let us assume that the open differential relation $\r$ is chosen so that every Morse function on a manifold of dimension $d+1$ is a solution of $\r$. 

By definition, given a map $f\co M\to N$, a point $x\in M$ is said to be a \emph{fold point}, if there are a neighborhood $U$ about $x$ diffeomorphic to $\R^{n-1}\times \R^{d+1}$ and a neighborhood of $f(x)$ in $N$ diffeomorphic to $\R^{n-1}\times \R$ such that $f|U$ is the product map $\id_{\R^{n-1}}\times g$, where $g$ is a Morse function. A point $x\in M$ is said to be \emph{regular} if the restriction of $f$ to some neighborhood of $x$ is a submersion. If each point of $M$ is either fold or regular, then $f$ is said to be a \emph{fold map}. Since $\r$ is a stable differential relation, the assumption that Morse functions satisfy $\r$ implies that each fold map of dimension $d$ is also a solution of $\r$. 

We will need the following well-known fact, for a proof see Proposition 4.2 in the review \cite{Sa4}. 

\begin{lemma}\label{l:9.2} Let $f\co M\to \R^{n+1}$ be a fold map and $\pi\co \R^{n+1}\to \R^{n}$ the projection along the first coordinate. Let $\Sigma\subset M$ denote the set of fold points of $f$. Suppose that the composition $\pi\circ f|\Sigma$ is a fold map of $\Sigma$. Then $\pi\circ f$ is a fold map of $M$, and the fold points of $\pi\circ f$ coincide with the fold points of $\pi\circ f|\Sigma$. 
\end{lemma}

\begin{figure}[ht]
\centering
\begin{minipage}{.5\textwidth}
\centering
\includegraphics[height=1.6in]{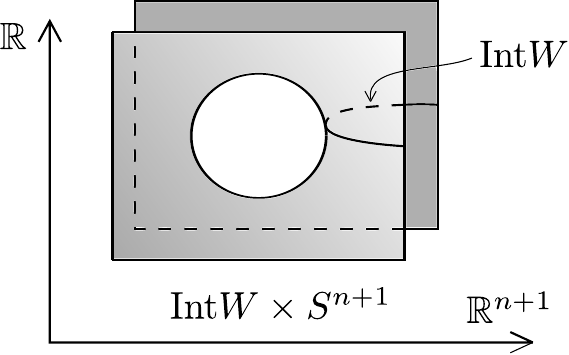}
\caption{A fold map $\mbox{Int} W\times S^{n+1}\to \R^{n+2}$.}
\label{fig:8}
\end{minipage}%
\begin{minipage}{.5\textwidth}
  \centering
\includegraphics[height=1.6in]{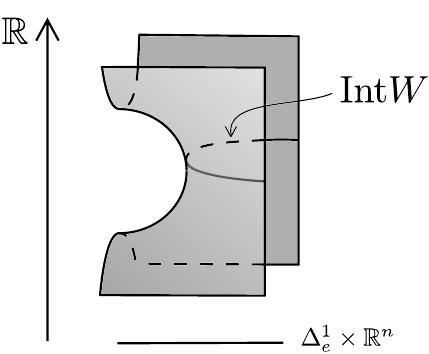}
\caption{A breaking concordance.}
\label{fig:9}
\end{minipage}
\end{figure}

Let $M$ be a closed manifold of dimension $d-1$ bounding a compact manifold $W$. There is a positive valued proper Morse function $m$ on the interior $\mbox{Int}\,{W}$ of $W$ such that $m^{-1}(1/2, \infty)$ is diffeomorphic to $M\times (1/2, \infty)$, and the restriction of $m$ to the latter is the projection onto the second factor. Let $(u, \alpha)$ denote the composition
\[
    \mbox{Int}\,{W}\times S^{n+1} \xrightarrow{(m, \id)} \R_+\times S^{n+1}\longrightarrow \R^{n+2}\simeq \R^{n+1}\times \R,
\]
where the second map is the one that takes a positive scalar $r$ and a vector $v$ of length $1$ to the vector $rv$. By Lemma~\ref{l:9.2}, if all Morse functions are solutions of $\r$, then $u$ is an $\r$-map. Choose a map $\beta$ of $\mbox{Int}\,{W}\times S^{n+1}$ to $\R^{\infty-1}$ such that $\mathbf{u}=(u, \alpha, \beta)$ is a graphic $\r$-map. Identify $\Delta^1_e$ with the opening of $[0,1]$ in $\R$. 
Restrict $\mathbf{u}$ to $u^{-1}(\Delta^{1}_e\times \R^{n})$; then $\mathbf{u}$ is a graphic $\r$-map to $\Delta^1_e\times \R^{n}$. For $i=0,1$, denote the restriction of $\mathbf{u}$ to $u^{-1}(\{i\}\times \R^{n})$ by $\mathbf{u}_i$. Then $\mathbf{u}$ is a concordance of $\mathbf{u}_1$ to $\mathbf{u}_0$. We say that the concordance $\mathbf{u}$
\emph{breaks} the fiber of $\mathbf{u}_1$ over $0\in \R^{n}$. Note that $u_1$ is a submersion $M\times \R\times \R^n\to \R^n$, while $u_0$ is a map to $\R^n$ such that for each $x\in \R^n$ close to $0\in \R^n$, we have $\alpha_0(u_0^{-1}(x))\ne \R$.

\begin{definition}\label{d:7.1} Given an integer $\mathbf{m}\ge 0$ and a smooth manifold $N$, let $hE_\r(\mathbf{m})(N)$
denote the set of pairs $(V, F)$ of 
\begin{itemize}
\item a submanifold $V=V_1\sqcup \dots \sqcup V_m$ in $N\times \R\times \R^{\infty-1}$ such that 
the inclusion $(u, \alpha, \beta)$ of $V$ is a  graphic map of order $1$ and every regular fiber $M$ of $(u, \alpha)$ is zero cobordant, and
\item a normal stable formal $\r$-map $F\co V\to N$ covering $u$. 
\end{itemize}
Then $hE_\r(\mathbf{m})$ is a partial sheaf for each $\mathbf{m}$, and $hE_\r(\bullet)$ is a partial $\Gamma$-sheaf.
\end{definition}

It follows that every pair $(V, F)$ of $hD_\r(\mathbf{m})(N)$ is in $hE_\r(\mathbf{m})(N)$. Indeed, the condition 
$\alpha(u^{-1}(x))\ne \R$ is strictly stronger than the condition that $M$ is zero cobordant; 
if $M$ is a regular fiber of $(u, \alpha)$ over a point $(x, t)$ and $\alpha(u^{-1}(x))\ne a$, then we may choose
the cobordism $W$ to be the inverse image of $[t, a]$ if $t\le a$ or $[a, t]$ if $a<t$ with respect to the map $\alpha|u^{-1}(x)$. 

To simplify notation, for any index $z$, we will write $\mathbf{V}_z$ for a tuple $(V_z, f_z, u_z, \alpha_z, \beta_z)$
of a graphic map $(V_z, u_z, \alpha_z, \beta_z)$ of order $1$ and a normal stable formal $\r$-map $f_z$ covering $u_z$.

\begin{theorem}\label{th:7.1} Let $\r$ be an open stable differential relation imposed on maps of dimension $d>0$ such that every Morse function on a manifold of dimension $d+1$ is a solution of $\r$.
Then the inclusion $|hD_\r(\bullet)|\to |hE_\r(\bullet)|$ is a weak homotopy equivalence.
\end{theorem}

\begin{proof} Clearly, the inclusion in the statement of Theorem~\ref{th:7.1} is a natural transformation. Hence, it remains to show that the map $\varphi\co hD_\r(\mathbf{m})(N)\to hE_\r(\mathbf{m})(N)$
induces an isomorphism of concordance classes for each non-negative integer $\mathbf{m}$ and each closed manifold $N$.  
We will only prove the surjectivity of the induced map;  the 
injectivity is proved similarly, see Lemma~\ref{le:8.6}.

Since we consider sections of sheaves up to concordance, we may identify normal stable formal $\r$-maps with the corresponding tangential stable formal $\r$-maps. 

 Let $\mathbf{V}_0$ be an element in the target of $\varphi$. If $\mathbf{V}_0$
is not in the image of $\varphi$, then, choose a regular value $(x, t)$ of $(u_0, \alpha_0)$. We may assume that near the regular fiber $M$ over $(x,t)$, the map $(u_0, \alpha_0)$ is given by the projection $M\times U\to U$ in the notation of Lemma~\ref{l:8.1}. By Lemma~\ref{l:8.1}, we may modify $f_0$ by homotopy in a neighborhood of the closure of $M\times U$ so that $f_0=du_0$ over $M\times U$. Next, by the construction before Definition~\ref{d:7.1}, there is a concordance of $\mathbf{V}_0$ to $\mathbf{V}_1$ with support in $M\times U$ such that $\alpha_1(u^{-1}_1(x))\ne t$. 
{\ccc Such a  concordance breaks fibers not only over $x$, but also over all points in a neighborhood $O_x$ of $x$. 

We claim that after applying the construction on different $t$-levels over finitely many points $x$ we obtain an element in the image of $\varphi$ concordant to $\mathbf{V}_0$. 
Indeed, a slight perturbation of a solution of an open stable differential relation $\r$ is again a solution of $\r$. Hence, by slightly perturbing the map $(u_0, \alpha_0)$, we may assume \cite[Section 2.1]{AVG} that the set of its non-regular values is a stratified set, each stratum $\Sigma\subset N\times \R$ of which is an embedded manifold of dimension $\le \dim N$. Furthermore, we may assume \cite{AVG} that each non-empty fiber of the projection $\Sigma\subset N\times \R\to N$ is a discrete set of points. Then for any $x\in N$, there exists $t_x\in \R$ such that the pair $(x, t_x)$ is a regular value of $(u_0, \alpha_0)$. Since $N$ is compact, we may choose finitely many points $x$ such that the neighborhoods $O_x$ cover $N$. Furthermore, since the set of regular values of $\alpha_0|u^{-1}(x)$  is open, we may choose $t_x$ to be distinct for different chosen points $x$. Then, applying the breaking concordance over the chosen points $x$, we obtain an element in the image of $\phi$ concordant to $\mathbf{V}_0$.}
\end{proof}

\begin{definition}\label{d:7.3} Given an integer $\mathbf{m}\ge 0$ and a smooth manifold $N$, let $hF_\r(\mathbf{m})(N)$
denote the set of submanifolds
\[
    V=V_1\sqcup \cdots \sqcup V_m \longrightarrow N\times \R\times \R^{\infty-1}
\]
such that the inclusion $(u, \alpha, \beta)$ is a graphic $\r$-map of order $1$, and 
every regular fiber of $(u, \alpha)$ is a manifold cobordant to zero. 
Then $hF_\r(\mathbf{m})$ is a set valued partial sheaf for every $\mathbf{m}$, and $hF_\r(\bullet)$ is a set valued
partial $\Gamma$-sheaf. 
\end{definition}

There is an inclusion $hF_\r(\bullet)\to hE_\r(\bullet)$ of partial $\Gamma$-sheaves that takes an element $(V, u, \alpha, \beta)$ in 
$hF_\r(\mathbf{m})(N)$ to the element $(V, f=du, u, \alpha, \beta)$ in 
$hE_\r(\mathbf{m})(N)$. 

\begin{theorem}\label{th:10.6} The inclusion $|hF_\r(\bullet)|\to |hE_\r(\bullet)|$  is a weak homotopy equivalence. 
\end{theorem}

To prove Theorem~\ref{th:10.6} we need a so-called Destabilization Lemma such as one in \cite{Sa}
and in \cite{EGM}. Observe that every formal $\r$-map $F\co TM\to TN$ represents a tangential stable formal $\r$-map. It is also easy to see that not every tangential stable formal $\r$-map is represented by a formal $\r$-map.
For example, there is a tangential stable formal immersion $TS^2\times \R\to T\R^2\times \R$, but there is no formal immersion $TS^2\to T\R^2$ of the $2$-sphere $S^2$ into $\R^2$. However, for many open stable differential relations $\R$, every tangential stable formal $\r$-map is concordant to a formal $\r$-map, \cite[Section 11]{Sa}, \cite[Section 2.5]{EGM}; e.g., the mentioned tangential stable formal immersion of the sphere is concordant to the immersion of an empty manifold.   

We will prove shortly a version of the Destabilization Lemma for all open stable differential relations, see Lemma~\ref{l:11.1}. Namely, we will show that for a compact manifold $N$, every element in  $hE_\r(\mathbf{1})(N)$ is concordant to an element in which the tangential stable formal solution is a formal solution.

\begin{proof}[Proof of Theorem~\ref{th:10.6}] We will show that for each closed manifold $N$, the map $hF_\r(\mathbf{m})(N)\to hE_\r(\mathbf{m})(N)$
induces an isomorphism of concordance classes. In fact we will show only the surjectivity;  the 
injectivity is proved similarly, see \ref{le:8.6}.

Every element $\mathbf{V}_0$  in $hE_\r(\mathbf{m})(N)$ is 
concordant to an element $\mathbf{V}_1$ such that each path component of $V_1$
is an open manifold. Indeed, there is a regular value $t$ of $\alpha_0$. Let $s\co [0,1]\times \R\to \R$ be 
the map $s(a,b)=t+(1-a)b$. It is transverse to $\alpha_0$, and therefore the map 
\[
    N\times [0,1] \times \R\times \R^{\infty} \longrightarrow N\times \R\times \R^{\infty-1},
\]
\[
  (x, y, t, z) \mapsto (x, s(y,t), z)
\]
pulls $\mathbf{V}_0$ back to a concordance from $\mathbf{V}_0$ to the desired element $\mathbf{V}_1$. In 
fact, $V_1=\alpha_0^{-1}(t)\times \R$. 

By the Destabilization Lemma~\ref{l:11.1}, we may assume that the stable formal $\r$-map $F_1$ is represented by a formal $\r$-map which we continue to denote by $F_1$. By the Gromov h-principle~\cite{Gro}, the formal $\r$-map $F_1$ covering $u_1$ is homotopic to 
the formal $\r$-map $F_2$ covering $u_2$ such that $F_2=du_2$. Without loss of generality we may assume that $\beta_1$ is an embedding. Then
$(V_1, u_2, \alpha_1, \beta_1)$ is an element in $hF_\r(\mathbf{m})(N)$, which, as an element in $hE_\r(\mathbf{m})(N)$ is concordant to $\mathbf{V}_0$. 
\end{proof}

\begin{corollary}\label{c:10.7} The classifying $\Gamma$-space $Bh\M_\r(\bullet)$ is weakly homotopy  
equivalent to the geometric realization $|hF_\r(\bullet)|$. 
\end{corollary}

Corollary~\ref{c:10.7} provides us with a convenient model (see Definition~\ref{d:7.3}) for the classifying $\Gamma$-space $Bh\M_\r(\bullet)$.

\section{The Destabilization Lemma}\label{s:11a}

In this section we prove the Destabilization Lemma. The proof heavily relies on deep facts from the singularity theory. For readers' convenience we review the most important notions in Appendix \ref{a:1}. 

Let $f\co (X, x)\to (Y, f(x))$ be a map germ. Recall that an \emph{unfolding} $(F, i, j)$ of $f$ is a cartesian diagram of map germs
\[
     \begin{CD}
     (P, p) @>F>> (Q, F(p)) \\
     @AiAA @AjAA \\
     (X, x) @>f>> (Y, f(x)),
     \end{CD}
\]
where $i,j$ are immersion germs such that $j$ is transverse to $F$. The \emph{dimension} of an unfolding is the difference $\dim P-\dim X$. 
By definition, a \emph{morphism}
$(\varphi, \psi)\co (F, i, j)\to (F', i', j')$
of two unfoldings of a map germ $f$ is a commutative diagram of map germs (in order to simplify notation we will omit the reference points of map germs):

\begin{figure}[h]
\[
\xymatrix{
P \ar[rrr]^F \ar[dd]^{\varphi} & & &  Q \ar[dd]^{\psi} \\
  & X \ar[ul]^i \ar[r]^f \ar[dl]^{i'} & Y \ar[ur]^j \ar[dr]^{j'}  & \\
P' \ar[rrr]^{F'} & & & Q'.   
}
\]
\caption{A morphism diagram}\label{d:1}
\end{figure}
An unfolding $(F', i', j')$ of a map germ $f$ is said to be \emph{versal} if for every unfolding $(F, i, j)$ of $f$ there is a morphism $(F, i, j)\to (F', i', j')$ of unfoldings.  It is known~\cite[page 91]{GWPL} that an unfolding is versal if and only if it is \emph{stable} (for a definition, see \ref{a:1}). On the other hand, any two stable unfoldings of the same dimension of a given map germ $f$ are isomorphic~\cite[page 86]{GWPL}. 

Let us spell out the definition of a morphism in the case where $\psi$ is an embedding. We may replace all manifolds in the diagram in  Figure~\ref{d:1} with coordinate neighborhoods of reference points, and thus, obtain a new diagram, see Figure~\ref{d:2}. 
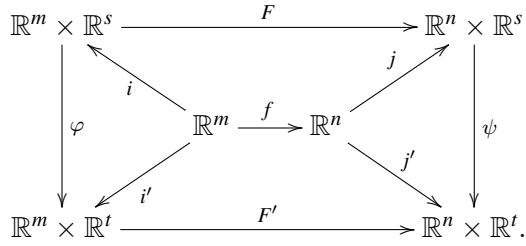
\begin{figure}[h]
\[
\xymatrix{
\R^m\times \R^s \ar[rrr]^F \ar[dd]^{\varphi} & & &  \R^n\times \R^s \ar[dd]^{\psi} \\
  & \R^m \ar[ul]^i \ar[r]^f \ar[dl]^{i'} & \R^n \ar[ur]^j \ar[dr]^{j'}  & \\
\R^m\times \R^t \ar[rrr]^{F'} & & & \R^n\times \R^t.   
}
\]
\caption{A morphism diagram}\label{d:2}
\end{figure}
Since $j, j'$ and $\psi$ are immersion germs, we may assume that these map germs are the standard inclusion germs. Since $j$ is transverse to $F$ and $j'$ is transverse to $F'$, we may assume that 
\[
   F(x,v)=(g(x, v), v), \qquad x\in \R^m, v\in \R^s,
\]
\[
   F'(x, u)=(g'(x, u), u), \qquad x\in \R^m, u\in \R^t, 
\]
for some map germs $g$ and $g'$. Chasing the diagram in Figure~\ref{d:2}, we deduce now that $\varphi$ is an immersion germ with image in $\R^m\times \R^s\subset \R^m\times \R^t$. Therefore, by a change of coordinates, we may assume that $\varphi$ is also a standard embedding. In other words, we may assume that $F$ is a restriction of $F'$.   

Recall that by definition a map germ $f$ is a solution germ of an open stable differential relation $\r$ if its stable unfolding is a solution of $\r$, see \ref{a:1}. Observe that a stable unfolding of $f$ is not unique. However, Theorem~\ref{th:11.1a} below shows that the notion of a solution is well-defined.   

\begin{theorem}\label{th:11.1a} Let $f$ be a solution germ of an open stable differential relation $\r$. Then, for every unfolding $(F, i, j)$ of $f$, the map germ $F$ is a solution of $\r$.   
\end{theorem}
\begin{proof} Let $(F', i', j')$ be a stable unfolding of $f$. Since it is versal, there is a morphism $(\psi, \varphi)$ of the unfolding $(F, i, j)$ to $(F', i', j')$, see Figure~\ref{d:1}.  Without loss of generality we may assume that the map germ $\psi$ is an embedding germ; if necessary we can always replace $F'$ with $F'\times \id_{\R^k}$. 
Then, the map germ $F$ is stably unfolded by $(F', \varphi, \psi)$, and therefore $F$ is a solution germ of $\r$. 
\end{proof}

Next, let us show how a solution germ of an open stable differential relation can be straightened.
Let 
\[
      f\co \R^m\times \R^k\to \R^n\times \R^k
\]
be a solution germ of $\r$ at $0$ transverse to $\R^n\times \{0\}$; the map $f$ is not required to preserve the product structures. Then after a change of coordinates, we may assume that $f$ is given by $(x, s)\mapsto (g(x, s), s)$, where $x\in \R^m$ and $s\in \R^k$. We may also assume that the map $f$ representing the $\r$-map germ under consideration is actually an $\r$-map defined over all $\R^m\times \R^k$. Then for every fixed $s_0$, the function $x\mapsto g(x, s_0)$ {\ccc is unfolded by $f$, and therefore it} is 
an $\r$-map. Let $\{f_t\}$ be the family of map germs parametrized by $t\in \R$ and defined by  
\[
    f_t\co (x, s) \mapsto (g(x, (1-t)s), s). 
\]
In particular, $f_0=f$ and $f_1$ is a product of the map $g(-,0)\co \R^m\to \R^n$ and the identity map of $\R^k$. We note that 
the family $\{f_t\}$ forms a map 
\[
     F\co \R^m\times \R^k\times \R\to \R^n\times \R^k\times \R
\]
that together with the embeddings $i$ of $\R^m\times \{s_0\}\times \{t_0\}$ and $j$ of $\R^n\times \{s_0\}\times \{t_0\}$ is an unfolding of $g(-, (1-t_0)s_0)$ for every pair of fixed parameters $s_0$ and $t_0$. By Theorem~\ref{th:11.1a}, the homotopy $\{f_t\}$ of map germs is through solution germs of $\r$. 
Thus we constructed a homotopy parametrized by $t\in [0,1]$ from a given map germ $f$ to the straightened map germ $g(-, 0)\times \id_{\R^k}$ through solution germs of $\r$.  

We will apply the straightening homotopy in the following situation (Lemma~\ref{l:12.1a}). Let $u\co \R^{m+k}\to \R^n\times \R^k$ be an $\r$-map transversal to the copy $\R^{n}\times \{0\}$ of $\R^n$, and let $V\subset \R^{m+k}$ be the inverse image of {\ccc $\R^n\times \{0\}$} with respect to $u$. Then $u|V$ is a map $V\to \R^n$, and $du|V$ is a representative
\[
    TV\times \R^k\longrightarrow T\R^n\times \R^k
\] 
of a stable formal $\r$-map. 

\begin{lemma} \label{l:12.1a}  There is a straightening homotopy of $du|V$ through stable formal $\r$-maps to a stable formal $\r$-map represented by 
\[
   g_0\times \id_{\R^k}\co TV\times \R^k\longrightarrow T\R^n\times \R^k,
\]
where $g_0$ is a formal $\r$-map. 
\end{lemma}

Now we are in position to formulate and {\ccc prove} the Destabilization Lemma. 

\begin{lemma}[Destabilization Lemma]\label{l:11.1} For a compact manifold $N$, every element 
\[
(u_0\co V_0\to N, \alpha_0, \beta_0, F_0) \qquad \textrm{in} \quad hE_\r(\mathbf{1})(N)
\]
is concordant to an element $(u_1\co V_1\to N, \alpha_1, \beta_1, F_1)$ in which the tangential stable formal solution is represented by a formal solution $TV_1\to TN$. 
\end{lemma}

The tangential stable formal solution of every element in $hE_\r(\mathbf{1})(N)$ is represented by a formal solution $TV\times \R^n\to TN\times \R^n$ for some $n$. The Destabilization Lemma asserts that after changing a given element in $hE_\r(\mathbf{1})(N)$ by concordance, the value $n$ can be chosen to be $0$. 

\begin{proof} Recall that an element in $hE_\r(\mathbf{1})(N)$ consists of a submanifold $V_0$ in the space $N\times \R\times \R^{\infty-1}$ with the inclusion map $(u_0, \alpha_0, \beta_0)$ and a tangential stable formal $\r$-map $F_0$ covering $u_0$. The inclusion $(u_0, \alpha_0, \beta_0)$ of the manifold $V_0$ is required to be a graphic map of order $1$, and every regular fiber of $(u_0, \alpha_0)$ is required to be zero cobordant. To begin with, observe that since $N$ is compact, the tangential stable formal solution $F_0$ admits a representation by a formal $\r$-map 
\[
    \mathbf{F}_0\co T(V_0\times \R^k)\longrightarrow T(N\times \R^k)
\]
covering $u$ in the sense that each map germ 
\[
\mathbf{F}_0|_{x\times \{0\}}\co T_xV_0\times \R^k \longrightarrow  T_{u_0(x)}N\times \R^k
\] 
belongs to the equivalence class $F_0|_{x}$.   
By the Gromov h-principle~\cite{Gro}, there is a homotopy $\mathbf{F}=\{\mathbf{F}_t\}$ of $\mathbf{F}_0$ through 
\begin{figure}[h]
\centering
\includegraphics[height=1.9in]{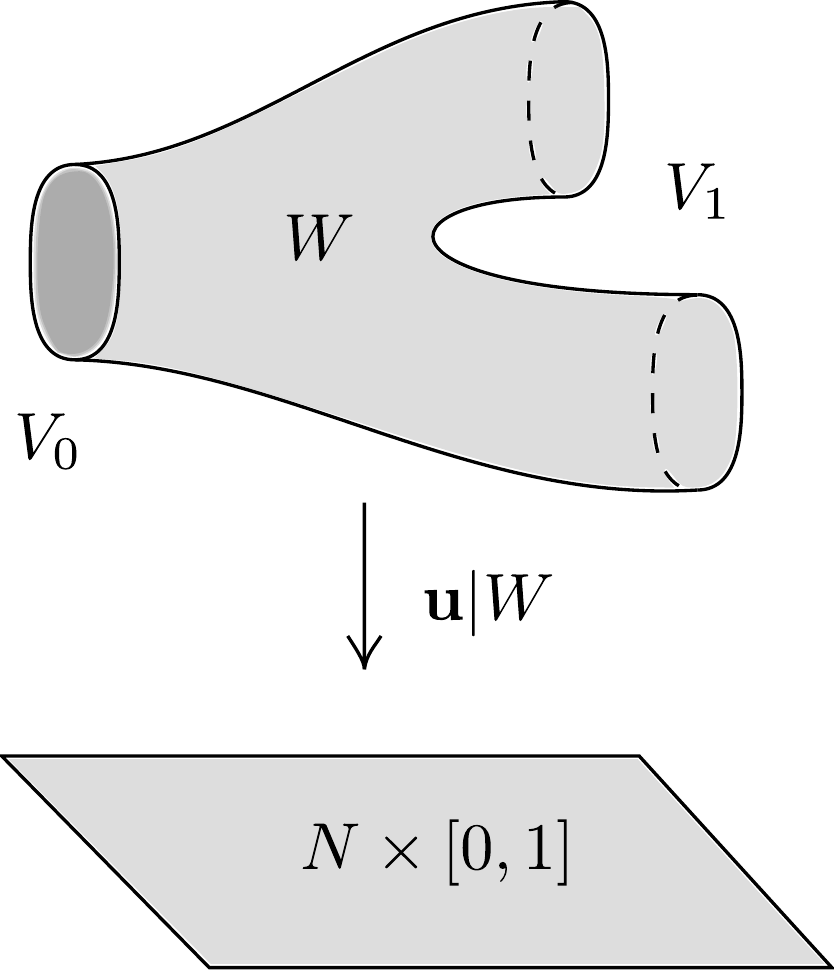}
\caption{Projection of $W$ to $N\times [0,1]$.}
\label{fig:14}
\end{figure}
formal solutions parametrized by $t\in [0,1]$ such that $\mathbf{F}_1$ is the formal solution represented by a genuine solution $\mathbf{u}_1\co V_0\times \R^k\to N\times \R^k$ in the sense that $\mathbf{F}_1=d\mathbf{u}_1$. The homotopy $\mathbf{F}$ itself is a map
\[
\mathbf{F}\co T(V_0\times \R^k)\times [0,1] \to T(N\times \R^k)\times [0,1],
\]
where $[0,1]$ stands for the vector bundle of dimension $0$ over $[0,1]$. For example, the source space of $\mathbf{F}$ is  the total space of the vector bundle over $V_0\times \R^k\times [0,1]$ with fiber over $(x, t, s)$ given by $T_xV_0\times T_t\R^k$. 
We may assume that the map 
\[
   \mathbf{u}\co V_0\times \R^k\times [0,1]\to N\times \R^k\times [0,1]
\]
underlying the homotopy $\mathbf{F}$
is smooth and transverse to the submanifold $N\times\{0\}\times [0,1]$, which we identify with $N\times [0,1]$.  
In particular, for $W=\mathbf{u}^{-1}(N\times [0,1])$, the map 
\[
   \mathbf{u}|W\co W\to N\times [0,1]
\]
is a smooth concordance, see Figure~\ref{fig:14}, covered by a family  $\mathbf{F}|W$ of stable formal solutions. In particular,  we have a commutative diagram:
\[
\begin{CD}
   T(V_0\times \R^k)\times [0,1]|_{W} @>\mathbf{F}|W>> (TN\times \R^k)\times [0,1] \\
   @VVV @VVV \\
   W @>\mathbf{u}|W>> N\times [0,1],
\end{CD}
\]
where we identify $T(N\times \R^k)|_{N\times \{0\}}$ with $TN\times \R^k$. It is convenient to represent the family $\mathbf{F}|W$ of stable formal solutions by the equivalent family of stable formal solutions $\mathbf{F}\times \id_{T[0,1]}|W$, which we will also denote by  $\mathbf{F}|W$. Then we obtain a commutative diagram
\[
\begin{CD}
   T(V_0\times \R^k\times [0,1])|_{W} @>\mathbf{F}|W>> TN\times \R^k\times T[0,1] \\
   @VVV @VVV \\
   W @>\mathbf{u}|W>> N\times [0,1],
\end{CD}
\]

Since the map $\mathbf{u}$ is transverse to $N\times [0,1]$ and the normal bundle of $N\times [0,1]$ in the target space of $\mathbf{u}$ is trivial, we deduce that the normal bundle of $W$ in $V_0\times \R^k\times [0,1]$ is trivial of dimension $k$. In other words, we may identify $T(V_0\times \R^k\times [0,1])|_W$ with $TW\times \R^k$. Consider the so-obtained concordance 
\[
\begin{CD}
   T{W}\times \R^k @>\mathbf{F}|W>> TN\times \R^k\times [0,1] \\
   @VVV @VVV \\
   W @>\mathbf{u}|W>> N\times [0,1].
\end{CD}
\]
The boundary of $W$ is the disjoint union of $V_0$ and some manifold $V_1$ such 
that $\mathbf{u}|W$ maps $V_0\subset \partial W$ to $N\times \{0\}$ and $V_1\subset \partial W$ to $N\times \{1\}$. The
\begin{figure}[h]
\centering
\includegraphics[height=1.7in]{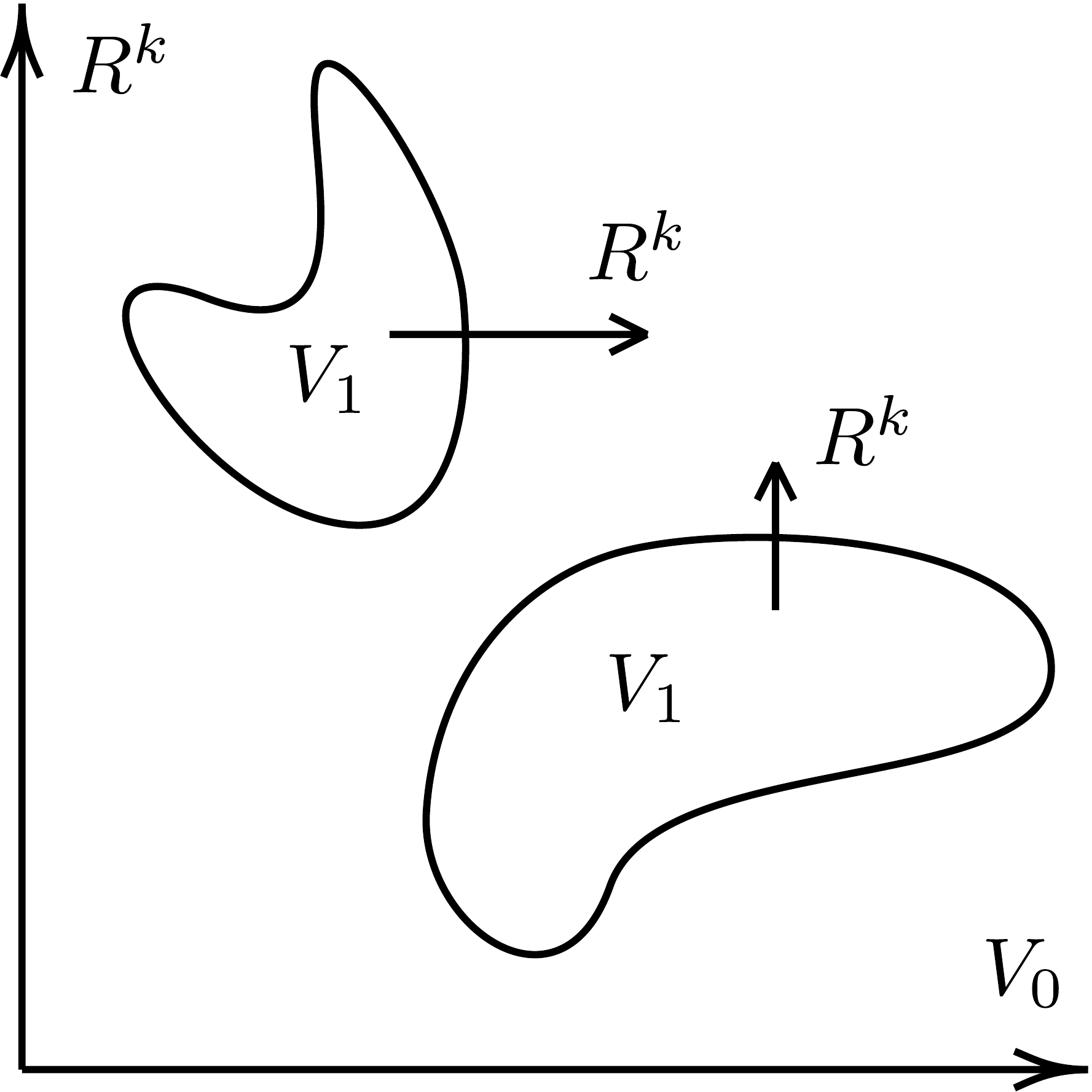}
\caption{The manifold $V_1$}
\label{fig:15}
\end{figure}
map $u_0=\mathbf{u}|V_0$ is covered by the formal map $\mathbf{F}_0|V_0$ representing the original stable formal map $F_0$, while the map $u_1=\mathbf{u}_1|V_1$ is covered by a formal map $F_1=\mathbf{F}_1|V_1$ so that we have a commutative diagram:
\[
\begin{CD}
   TV_1\times \R^k @>\mathbf{F}_1|V_1>> TN\times \R^k  \\
   @VVV @VVV \\
   V_1 @>\mathbf{u}_1|V_1>> N.
\end{CD}
\]
Observe that the tangential stable formal solution $\mathbf{F}_1|V_1$ may not be represented 
by $\mathbf{u}_1|V_1$; we only know that $\mathbf{F}_1|V_1$ is a restriction of a formal solution 
\[
\mathbf{F}_1\co T(V_0\times \R^k)\longrightarrow T(N\times \R^k)
\]
 represented by the genuine solution $\mathbf{u}_1$. However, in view of the straightening homotopy described in Lemma~\ref{l:12.1a}, we may assume that $\mathbf{F}_1|V_1$ is represented by $\mathbf{u}_1|V_1$. 

{\ccc Note that for any extension of $\alpha_0$ to a function $\alpha$ over $W$, each regular fiber of $(\alpha, \mathbf{u}|W)$ is cobordant to a regular fiber of $(\alpha_0, u_0)$ and therefore is zero cobordant. }Therefore,  
 we may extend $\alpha_0$ and $\beta_0$ over $W$ so that together with $\mathbf{u}|W$ and $\mathbf{F}|W$ they define a desired concordance. 
\end{proof}

\section{The proof of the bordism principle}\label{sec:8}

We have seen that the functors $B\M_\r(\bullet)$ and $Bh\M_\r(\bullet)$ are weakly homotopy equivalent to 
the geometric realizations of partial $\Gamma$-spaces $D_\r(\bullet)$ and $hF_\r(\bullet)$ respectively. Recall that 
$hF_\r(\mathbf{m})(N)$ is defined to be the set of graphic $\r$-maps $(f, \alpha, \beta)$ to $N$ of order $1$ such that every regular fiber of $(u, \alpha)$ is a manifold cobordant to zero, while $D_\r(\mathbf{m})(N)$ is the subset of $hF_\r(\bullet)(\mathbf{m})(N)$ of broken ones. In particular, there is an inclusion of classifying spaces
\begin{equation}\label{eq:7.5}
 \alpha_\r\co   B\M_\r\longrightarrow Bh\M_\r. 
\end{equation}
In its turn, the map (\ref{eq:7.5}) gives rise to the map of infinite loop spaces 
\begin{equation}\label{eq:3.5}
   \Omega B\M_\r\longrightarrow \Omega Bh\M_\r,
\end{equation}
which can be identified with the map of infinite loop spaces of spectra $\Omega^{\infty}\mathbf{M}_\r\to \Omega^{\infty}h\mathbf{M}_\r$.  By definition, the \emph{b-principle map} for an open stable differential relation $\r$ is the composition
\[
     \M_\r \longrightarrow \Omega B\M_\r \xrightarrow{\Omega\alpha_\r} 
     \Omega Bh\M_\r \stackrel{\simeq}\longrightarrow h\M_r.
\]
Thus in order to prove the b-principle it suffices to show that $\alpha_\r$ is a homotopy equivalence. 

We can now derive Theorem~\ref{th:1} in the case where $d<1$; in particular, we derive the Barratt-Priddy-Quillen Theorem (see Theorem~\ref{th:bpq}).

\begin{proof}[Proof of Theorem~\ref{th:1} for $d<1$]
Note that the condition $\alpha(u^{-1}(x))\ne \R$ in the definition of $\mathcal{B}\M_\r(\bullet)\simeq B\M_\r(\bullet)$
and the cobordism condition in the definition of $|hF_\r(\bullet)|\simeq Bh\M_\r(\bullet)$ are void in 
the case $d<1$. Thus, in this case, the mentioned models for  $B\M_\r(\bullet)$ and $hB\M_\r(\bullet)$ coincide. 
\end{proof}

\begin{proof}[Proof of Theorem~\ref{th:1} for $d\ge 1$] We need to show that the sets of concordance classes of sections in $D_\r(\mathbf{m})(N)$ and $hF_\r(\mathbf{m})(N)$ are isomorphic. This immediately follows from the construction of concordance that breaks fibers, see the proof of Theorem~\ref{th:7.1}. \end{proof}

\section{The homotopy type of the group completion of $\M_\r$}\label{s:11}

In this section we recall a fairly simple spectrum $\mathbf{B}_\r$ from \cite{Sa}, and in \S\ref{s:11} we show that the topological space $\Omega^{\infty}\mathbf{B}_\r$ is homotopy equivalent  to the infinite loop space  classifying the cohomology theory $h_\r$

 One of the first constructed models of $\mathbf{B}_\r$ is due to Wells~\cite{We} who used a  model of $\mathbf{B}_\r$ in the case of the immersion differential relations to compute bordism groups of immersions of dimension $<0$. Later Eliashberg introduced "equivariant  spectra" to study $k$-mersions, Legendrian and Lagrangian immersions~\cite{El}. Audin~\cite{Au3} modified the Eliashberg equivariant construction and obtain spectra that are prototypes of the spectra that we will describe in the current section.  Using the Audin-Eliashberg spectra explicit computations can be carried out (see books \cite{Au3} by Audin, \cite{Va} by Vassiliev, and Theorem~\ref{th:EA} below).  Over years, many related classifying spaces were constructed in singularity theory; notably, by Ando~\cite{An},  Szucz~\cite{Sz}, Rim\'anyi-Szucz~\cite{RS}, Kazarian~\cite{Kaz}, and the author \cite{Sa}.  Already mentioned the Tillmann-Madsen spectra and the sphere spectrum are also examples of $\mathbf{B}_\r$ corresponding to the cases of differential relations of submersions and coverings respectively.




\noindent{\bf Construction of $\Omega^{\infty}\mathbf{B_\r}$.\ } Let $\mathop\mathrm{EO}_n\to \BO_n$ denote the universal principle $\mathop\mathrm{O}_n$-bundle. For a given open stable differential relation $\r$ imposed on maps of dimension $d$, let $J_\r(n)$
denote the space of {\ccc jets of} $\r$-map germs
\[
     (\R^{n+d}, 0) \longrightarrow (\R^n, 0) 
\]
{\ccc
endowed with the Whitney $C^{\infty}$-topology. 
There is a left action of $\mathop\mathrm{O}_n$ on the space $J_\r(n)$; {\ccc in terms of representatives} an orthonormal transformation $h$ takes an $\r$-germ $f$ to an $\r$-germ $h\circ f$. The Borel construction results in a space 
\[
      B'_{n+d}\co={\mathop\mathrm{EO}}_n \times_{\mathop\mathrm{O}_n} J_r(n). 
\]
together with a fiber bundle projection $f_{n+d}\co B'_{n+d}\to \BO_n$. The colimit of maps $f_{n+d}$   is a normal $(B_\r, f_\r)$ structure $f_\r\co B_\r\to \BO$ of dimension $-d$ with associated Thom spectrum $\mathbf{B}_\r$, and the infinite loop space $\Omega^{\infty}\mathbf{B}_\r$. 

}


\begin{example}\label{e:3} For the submersion differential relation $\mathcal{R}$ imposed on maps of dimension $d$, the space $B_\r$ can be identified with the space of equivalence classes of linear maps $\alpha\co \R^{n+d}\to \R^{\infty}$ of rank $n$ for some $n\ge 0$; the equivalence relation is generated by equivalences $\alpha\sim \alpha'$, where 
$ \alpha\co \R^{n+d}\longrightarrow \R^{\infty}$  and 
\[
 \alpha'\co \R^{n+d+1}\stackrel{\equiv}\longrightarrow  \R\times\R^{n+d}\xrightarrow{\id_{\R}\times \alpha}  \R\times \R^{\infty}\stackrel{\equiv}\longrightarrow \R^{\infty}.
\]
There is a vector bundle $\ker\alpha$ of dimension $d$ over $B_\r$ whose fiber over the class of $\alpha$ is given by the kernel of $\alpha$. Thus, there is a well-defined map $\ker\alpha\co B_\r\to \BO_d$. The fiber of this map is contractible and therefore $B_\r\simeq \BO_d$. The map $f_\r\co  B_\r\to \BO\times\{d\}$ of the tangential structure is the standard inclusion $\BO_d\to \BO$ since it classifies the stable vector bundle represented by the kernel of the map $\alpha$.   
\end{example}

\begin{example} For the submersion differential relation $\r$ imposed on \emph{oriented} maps of dimension $2$, the tangential structure is given by the standard inclusion $\mbox{BSO}_2\to \mbox{BSO}$, and therefore the spectrum $\mathbf{B}_\r$ in this case coincides with the Madsen-Tillmann spectrum $\mathbf{MTSO}(2)$ which appears in the statement of the Madsen-Weiss theorem~\cite{MW} (the generalized Mumford Conjecture).  
\end{example}

\begin{example} \label{ex:6.4} By an argument similar to that in Example~\ref{e:3} we deduce that for the immersion differential relation on maps of dimension $d\le 0$, the map $f\co B\to \BO\times \{d\}$ of the normal structure is the standard inclusion $\BO_d\to \BO$. In particular, for the covering differential relation ($d=0$), the spectrum $\mathbf{B}_\r$ is the sphere spectrum.  
\end{example}

\begin{example} In the case of the differential relation whose solutions are all maps of dimension $d$, the space $B_\r$ is homotopy equivalent to $\BO$, and therefore 
$\mathbf{B}_\r= \Sigma^{-d}\mathbf{MO}$, where $\mathbf{MO}$ is the Thom spectrum. 
\end{example}

\begin{warn}\label{w:7.9} We will show that $\Omega^{\infty}h\mathbf{M}_\r$ is homotopy equivalent to $\Omega^{\infty}\mathbf{B}_\r$. However, the two deloopings provided by the spectra $h\mathbf{M}_\r$ and $\mathbf{B}_\r$ are different. Indeed, we have
\[
   \pi_ih\mathbf{M}_\r\simeq \pi_i\mathbf{B}_\r \qquad \mathrm{for}\quad i\ge 0.
\]
However, $\pi_ih\mathbf{M}_\r=0$ for $i<0$, while $\pi_i\mathbf{B}_\r$ may not be trivial for $i<0$ in general. 
\end{warn}


To prove that the topological space $h\M_\r$ is homotopy equivalent to $\Omega^{\infty}\B_\r$, we need the following lemma. 

\begin{lemma}\label{l:6.3} The concordance classes of tangential stable formal $\mathcal R$-maps to a compact manifold $N$ are in bijective correspondence with cobordism classes of $(B_\r, f_\r)$ maps to $N$. 
\end{lemma}
\begin{proof} The desired correspondence is essentially the same as the Pontrjagin-Thom construction in \cite{Sa}. For reader's convenience, 
we sketch the construction. 

Let $F$ be a stable formal $\r$-map covering a smooth proper map $u\co V\to N$. We will show that $F$ determines a normal $(B_\r, f_\r)$ structure on $u$, i.e.,  $F$ defines a lift
\[
\xymatrix{
 &  & & B_\r \ar[d]^{f_\r} \\
 V \ar@{-->}[urrr] \ar[rrr]^{u^*TN\ominus TV} & & &\BO\times \{-d\}. 
}
\]
of the classifying map of the stable vector bundle $u^*TN\ominus TV$ over $V$. 
The space $B_\r$ consists of equivalence classes of pairs $(i, P)$ of a solution germ $i\co \R^{n+d}\to P$ and a subspace $P\subset \R^{\infty}$ of dimension $n>0$ which represents a point in $\BO_n\subset \BO$. In terms of representatives, the projection is defined by $(i, P)\mapsto P$.

Since $V$ is compact, we may assume that $F$ is given by a family of map germs
\[
    F_x\co T_xV\otimes \varepsilon^{k}\longrightarrow u^*T_{u(x)}N\otimes \varepsilon^{k},
\]
parametrized by $x\in V$, 
where $k$ is a fixed non-negative integer. Choose an embedding $V\subset S^{n+d}$ into a sphere of sufficiently big dimension so that $V$ lies in the equator $S^{n+d-k}$ where $k$ is the above fixed integer.  Then the product with $u$ defines a new embedding $V\subset S^{n+d}\times N$. Choose a map $\nu\co V\to \BO_n$ classifying the normal bundle of $V$ in $S^{n+d}\times N$. It  represents $TN\ominus TV$, and therefore in order to construct a normal $(B_\r, f_\r)$ structure on $u$ it remains to construct a family of $\r$-map germs $\varepsilon^{n+d}\to \nu_x$ parametrized by $x\in V$. The desired  family is given by the compositions
\[
   \varepsilon^{n+d}\simeq (T_xV\oplus \varepsilon^k)\oplus T_x^{\perp}V \xrightarrow{F_x\oplus \id} 
(u^*T_{u(x)}N\oplus \varepsilon^{k}) \oplus T_x^{\perp}V \simeq \nu_x,
\]
where $T_x^{\perp}V$ is the normal bundle of $V$ in $S^{n+d-k}$.

The converse construction shows that a $(B_\r, f_\r)$-structure on a map $u$ determines a tangential stable formal solution.
These two constructions are sufficient to establish Lemma~\ref{l:6.3}. 
\end{proof}

In view of the Brown representability theorem, see Corollary~\ref{c:10.4}, Lemma~\ref{l:6.3} implies Corollary~\ref{c:14.3}. For an H-space $X$, let $X_0$ denote the path component of $X$ containing the base point. 

\begin{corollary}\label{c:14.3} The spaces $[h\M_\r]_0$ and $[\Omega^\infty\mathbf{B}_\r]_0$ are homotopy equivalent. 
\end{corollary}

In the case where $N$ is a point, Lemma~\ref{l:6.3} reduces to Lemma~\ref{c:14.4}.

\begin{corollary}\label{c:14.4} There is an isomorphism of monoids $\pi_0h\M_\r\simeq \pi_0 \Omega^{\infty}\mathbf{B}_\r$. In particular, the monoid $\pi_0h\M_\r$ is a group and therefore $h\M_\r$ is an infinite loop space. 
\end{corollary}

\begin{theorem}\label{th:14.5} The topological spaces $h\M_\r$ and $\Omega^{\infty}\mathbf{B}_\r$ are homotopy equivalent. 
\end{theorem}
\begin{proof} Theorem~\ref{th:14.5} immediately follows from Corollary~\ref{c:14.3} and Corollary~\ref{c:14.4}. 
\end{proof}


\section{Applications of the b-principle}\label{s:8}

\subsection{Cobordism monoids of solutions}\label{ss:5.1}

In this section we describe the set of homotopy classes $[N, \M_\r]$ of maps of smooth manifolds $N$. These are interpreted as cobordism classes of $\mathcal R$-maps into $N$. Definition~\ref{d:cob}
of the \emph{cobordism equivalence relation} differs from the standard one only in that it avoids manifolds with boundaries.

To simplify notation we often
suppress $\beta$ in the notation of a graphic map 
and denote a graphic $\mathcal R$-map $(f, \beta)$ by $f\colon V\to N$. 

\begin{definition}\label{d:cob}
Two graphic $\r$-maps $f_i\co V_i\to N$, with $i=0,1$, are said to be \emph{cobordant}
if  there is a graphic $\mathcal{R}$-map $f\co V\to N\times \R$ 
such that $\pi^*f_1=f$ over $N\times (-\infty, 0)$ and $\pi^*f_2=f$ over $N\times (1, \infty)$, where 
$\pi$ is the projection of $N\times \R$ onto the first factor. 
Two proper $\mathcal{R}$-maps are said to be \emph{cobordant} if they admit lifts to 
cobordant graphic $\mathcal R$-maps.
\end{definition}

The set of cobordism classes of proper $\mathcal{R}$-maps into a manifold $N$ forms a commutative monoid, called the \emph{monoid of proper $\mathcal{R}$-maps into $N$}. The monoidal operation is given by taking the disjoint union of maps, i.e., the sum of cobordism classes of maps $f_i\co  V_i\to N$, with $i=1,2$, is defined to be the cobordism class of the composition 
\[
V_1\sqcup V_2\xrightarrow{f_1\sqcup f_2} N\sqcup N\longrightarrow N
\]
where the latter map is given by the identity map on each of the two components of $N\sqcup N$.  The identity element is represented by the map of an empty manifold.

Note that the space of homotopy classes $[N, \mathcal{M}_\r]$ of continuous maps of a smooth manifold $N$ is a monoid since 
$\M_\r$ is an $H$-space. 


\begin{theorem}\label{th:2.3} The monoid of proper $\mathcal{R}$-maps into an arbitrary manifold $N$ is isomorphic to  $[N, \mathcal{M}_{\mathcal R}]$.
\end{theorem}
\begin{proof} The argument is similar to that given by Madsen and Weiss~\cite[\S A.1]{MW}; for reader's convenience we sketch the construction of the correspondence.   

Given a proper $\mathcal{R}$-map $f$ to $N$, let $\beta$ be an embedding of the source manifold of $f$ to $\R^{\infty}$ with image in a compact set. Choose a triangulation of $N$ with ordered vertex set 
so that all embeddings of simplices are coherent and transverse to $f$. Choose an extension of 
the embedding $\Delta^n\to N$ of each simplex in the triangulation to an embedding $j\co \Delta^n_e\to N$. Then the maps $(j^*f, \beta|f^{-1}j(\Delta^n_e))$ are graphic $\mathcal{R}$-maps into $\Delta^n_e$, each of which has a unique 	counterpart in the geometric realization $\mathcal{M}_{\mathcal R}$. In view of coherency of embeddings $j$,  we obtain a map $N\to \mathcal{M}_{\mathcal R}$. Its homotopy class depends only on the cobordism class $[f]$. 

Conversely, given a map $g\co N\to \mathcal{M}_{\mathcal R}$, choose a triangulation of $N$ and modify $g$ by homotopy so that it takes each simplex of dimension $k\ge 0$ in the triangulation of $N$ linearly onto a simplex of dimension $\le k$ in $\mathcal{M}_{\mathcal R}$. There is a smooth homotopy $h_t$ of $h_0=\id_N$ with $t\in [0,1]$ such that 
\begin{itemize}
\item $h_t$ maps each simplex of the triangulation of $N$ onto itself for each $t$, 
\item for each simplex $\Delta$ in the triangulation of $N$, the restriction of $h_t$ to $h_t^{-1}(\mathring{\Delta})$ is a submersion onto the interior $\mathring{\Delta}$ of $\Delta$, and
\item each simplex has a neighborhood that maps by $h_1$ onto the simplex.
\end{itemize}

Indeed to construct the homotopy $h_t$, we may choose by induction in $n\ge 0$ a homotopy $h_t^n$ of the (non-extended) simplex $\Delta^n$ so that $h^n_t\delta_i=\delta_ih^{n-1}_t\colon \Delta^{n-1}\to \Delta^n$ for each $n, i$ and $t$. Then the homotopies $\{h_t^n\}$ of the identity maps of simplices in the triangulation of $N$ agree and define a desired homotopy $h_t$. 
The map $gh_1$ pulls back a desired $\mathcal R$-map to $N$.  

It also follows that the monoidal operation on cobordism classes of proper $\mathcal R$-maps into $N$ agrees with that on $[N, \M_\r]$. 
\end{proof}

\begin{corollary} The cohomology group $k^0_\r(pt)$ of a point is the group completion of the monoid of cobordism classes of $\r$-maps to the point.   
\end{corollary}

\begin{remark}\label{r:14.4} Given a manifold $N$, the group $k^0_{\mathcal{R}}(N)$ may not be the group completion of the monoid of cobordism classes of $\mathcal{R}$-maps to $N$. Indeed, the cobordism monoid of coverings of $S^n$ for $n\ge 2$ is isomorphic to the monoid $\mathbb{N}$ of positive integers. On the other hand $k^0_{\mathcal R}(S^n)=[S^n, \Omega^{\infty}S^{\infty}]=\Z\times \pi^{st}_n$. 
\end{remark}

\begin{remark} In singularity theory, the term \emph{bordism monoids of $\mathcal R$-maps} is often replaced with the term \emph{cobordism monoids of $\mathcal R$-maps}. These two notions are different; bordism classes are represented by maps of closed manifolds while cobordism classes are represented by proper maps. For example, the cobordism group of $\mathcal R$-maps to $N$ is not isomorphic to the bordism group of $\mathcal R$-maps to $N$ for $N=\R$ in the case where $\mathcal R$ is the \emph{Morse differential relation}, i.e., the minimal open stable differential relation satisfied by Morse functions (for computations see the papers \cite{Ik}, \cite{IS}, \cite{Kal0}).  
\end{remark}

In the rest of the subsection we complete the proof of the assertion that the space $\M_\r$ has a structure of a $\Gamma$-space.

\begin{theorem}\label{th:5.6} The structure map $\gamma\colon \M_\r(\mathbf{m})\to \M_\r(\mathbf{1})^m$ is a homotopy equivalence.
\end{theorem}
\begin{proof} To simplify the notation we will assume that $m=2$. In general the argument is similar. Let $p$ and $q$ be two distinguished vertices in the source and target of $\gamma$ such that $\gamma(p)=q$. 
We may assume that both $p$ and $q$ are represented by a pair of disjoint submanifolds in $\Delta^n_e\times \R^{\infty}$. To show that $\gamma_*$ is surjective in homotopy groups in degree $n>0$, pick an element in the target of $\gamma_*$. By Theorem~\ref{th:2.3} it corresponds to a pair of cobordism classes of graphic $\r$-maps 
\[
     (f_i, \beta_i)\co V_i\longrightarrow S^n\times \R^{\infty}, \qquad i=1,2,
\]
where each $V_i$ is a smooth manifold. Furthermore, the pullback of $(f_i, \beta_i)$ with respect to the inclusion of the base point $\Delta^0_e\to S^n$ is a graphic $\mathcal R$-map representing $q$. A slight generic perturbation of $\beta_1$ outside 
$f^{-1}_1(\Delta^0_e)$ makes $\beta_1(V_1)$ disjoint from $\beta(V_2)$ in $\R^{\infty}$, and therefore defines a lift of the given class in  $\pi_n(\M_\r(\mathbf{1}))^2$
to a class in $\pi_n(\M_\r(\mathbf{2}))$. A similar argument shows that $\gamma$ induces an isomorphism of the sets of 
path components and that $\gamma_*$ is injective. 
\end{proof}

\subsection{Cobordism groups of solutions of $\mathcal R$}
One of the most obvious applications of the b-principle is computation of cobordism groups of solutions to differential relations. One of the first applications of the b-principle is due to Wells~\cite{We} who computed rational cobordism groups of immersions. 

\begin{theorem}[Wells] Let $T\xi_d$ denote the Thom space of the universal vector bundle of dimension $d$. Then modulo finite group the cobordism group of immersions of codimension $d$ into $S^{n+d}$ is isomorphic to $H_{n+d}(T\xi_d)$.
\end{theorem}
\begin{proof} An argument similar to that in Example~\ref{e:3} shows that in the normal structure $(B, f)$ of the immersion differential relation imposed on maps of codimension $d$ the space $B$ is the classifying space $\BO_d$ of vector bundles of dimension $d$, while $f\co B\to \BO\times \{d\}$ is the standard inclusion. Thus the spectrum of the bordism group of immersions is given by the suspension spectrum of the space $T\xi_d$. Hence the rational cobordism group of immersions of codimension $d$ into $S^{n+d}$ is isomorphic to $\pi_{n+d}(h\M_\r)\otimes\Q \approx\pi_{n+d}^{st}(T\xi_d)\otimes \Q\approx H_{n+d}(T\xi_d; \Q)$. 	
\end{proof}

The cobordism classes of Legendrian immersions of manifolds of dimension $n$ form an abelian group $L_n$ under the operation given by the disjoint union. The Eliashberg-Audin b-principle for Legendrian immersions~\cite{El}, ~\cite{Au3}
lead to computation of the rational cobordism group of Legendrian immersions~\cite{Va}. In fact the Eliashberg-Audin spectrum for Legendrian immersions is a prototype of our space $h\M_\r$, especially compare the construction in \cite{Au3} with \S\ref{s:11}. 

\begin{theorem}[Eliashberg-Audin]\label{th:EA} Let $L=\oplus_n^{\infty} L_n$ be the graded group of Legendrian immersions of manifolds of all dimensions. Then
\[
     L\otimes \mathbb{Q} \approx \Lambda_{\mathbb{Q}}[\alpha_0, \alpha_1, ...],
\]     
where $\Lambda_{\mathbb{Q}}[\alpha_0, \alpha_1, ...]$ is the exterior algebra in terms of classes $\alpha_i$ of degree $4i+1$.   
\end{theorem}

Computation of cobordism groups of maps with prescribed singularities were carried out for an infinite series of singularity types (called Morin singularities) by Sz\H{u}cs~\cite{Sz}  and the author~\cite{Sa2} for maps of negative and positive dimensions respectively by essentially computing the rational homology groups of the space $h\M_\r$.  We emphasize that the b-principle type theorems are essential for these computations.

\subsection{Invariants of solutions of $\mathcal R$}\label{s:inv}

 Let $\mathcal R$ be an open stable differential relation imposed on maps of dimension $d$. 

\begin{definition}\label{c:1}A \emph{characteristic class} of $\mathcal R$-maps is a function that assigns a class $w(f)\in H^*(|C_\bullet|)$ to every $\mathcal R$-map $f$ to a simplicial set $C_\bullet$ so that  
\begin{itemize}
\item $w(f)=w(g)$ for concordant $\r$-maps $f$ and $g$, and 
\item if $f: C_\bullet\to D_\bullet$ is an $\mathcal R$-map and $j: D'_\bullet\to D_\bullet$ a morphism, then 
$w(j^*f)=j^*w(f)$ 
for the pullback $\mathcal R$-map $j^*f$.    
\end{itemize}
\end{definition}

It immediately follows that characteristic classes of solutions are in bijective correspondence with cohomology classes of $\M_\r$. Similarly one defines characteristic classes of stable formal solutions. These classes are in bijective correspondence with classes in $H^*(h\M_\r)$. For example rational characteristic classes of stable formal oriented submersions are the generalized Miller-Morita-Mumford classes.  In general, from the Group Completion Theorem (e.g., see \cite{MS})  we deduce Theorem~\ref{th:12.10}.

\begin{theorem}\label{th:12.10} Let $\r$ be a differential relation satisfying the b-principle. Then the algebra $H^*(\M_\r; \mathbb{Q})$ of rational characteristic classes 
satisfies 
\[
H^*(\M_\r; \mathbb{Q})[\pi_0^{-1}] = \mathbb{Q}[\tau_1, \tau_2, \cdots]\otimes \Lambda_{\mathbb{Q}}[\eta_1, \eta_2, \cdots],
\]
where the classes $\tau_1, \tau_2, ...$ are of even degrees and $\eta_1, \eta_2, ....$ are of odd degrees. 
\end{theorem}

\appendix
\section{Appendix}

\subsection{Partial sheaves}\label{a:1.5}

Let $\mathcal{X}$ be the category of smooth manifolds and smooth maps. Recall that a set valued \emph{sheaf} $\mathcal{F}$ is 
a functor $\mathcal{X}^{op}\to \mathbf{Set}$ satisfying the Sheaf Axiom. Namely, for any locally finite covering $\{U_i\}$
of a manifold $M$ and any sections $s_i\in \mathcal{F}(U_i)$ that agree on common domains---that is, $s_i=s_j$ on $U_i\cap U_j$---there exists a section $s\in \mathcal{F}(M)$ with $s=s_i$ on $U_i$ for all $i$.

Similarly, a \emph{partial (set valued) sheaf} $\mathcal{F}$ is a partially defined functor $\mathcal{X}^{op}\to \mathbf{Set}$. To each manifold $M$, the partial sheaf $\mathcal{F}$ associates a set, as a usual sheaf does. However, for a smooth map $g\co M\to N$, the associated map $\mathcal{F}(g)$ is a map from a subset of $\mathcal{F}(M)$ to $\mathcal{F}(N)$. 
Finally, a partial sheaf is required to satisfy the Sheaf Axiom. 

A \emph{morphism} $f$ of a partial sheaf $\mathcal{F}$ into a partial sheaf $\mathcal{G}$ is a collection of 
maps $f_U\co \mathcal{F}(U)\to \mathcal{G}(U)$, one for each manifold $U\in \mathcal{X}$, such that the diagram 
\[
\xymatrix{
\mathcal{F}(U) \ar@{-->}[d]^{\mathcal{F}(\varphi)} \ar[r]^{f_U} & \mathcal{G}(U)\ar@{-->}[d]^{\mathcal{G}(\varphi)}\\
\mathcal{F}(V) \ar[r]^{f_V} & \mathcal{G}(V)}
\]
commutes for each $\varphi\co V\to U$. Note that the morphisms $f_U$ and $f_V$ are defined on the sets $\mathcal{F}(U)$ and $\mathcal{F}(V)$ respectively, 
while the morphisms $\mathcal{F}(\varphi)$ and $\mathcal{F}(\varphi)$ are defined only on subsets of $\mathcal{F}(U)$ and $\mathcal{G}(U)$ respectively. We require that if $\mathcal{F}(\varphi)$ is defined on a section $s$, then $\mathcal{G}(\varphi)$ is defined on the section $f_U(s)$. 

Many sheaf theoretic constructions are suitable for partial sheaves. For example, for each non-negative integer $m$, let $\mathcal{F}_m$ denote the subset of $\mathcal{F}(\Delta^m_e)$ of those sections that are 
in the domain of every morphism of the form $\mathcal{F}(\delta)$ where $\delta$ ranges 
over all simplicial morphisms $\Delta^n_e\to \Delta^m_e$. Then 
$\mathcal{F}_{\bullet}$ is a simplicial set and its geometric realization $|\mathcal{F}|$ is said to be the \emph{geometric realization}
of the partial sheaf $\mathcal{F}$. A morphism $f\co \mathcal{F}\to \mathcal{G}$ of partial sheaves gives rise to 
a continuous map $|f|$ of the realizations of the partial sheaves. 
A map $f$ of partial sheaves is said to be a \emph{weak equivalence} if $|f|$ is a homotopy equivalence. 

Two sections $s_0$ and $s_1$ in $\mathcal{F}(N)$ are said to be \emph{concordant} if there is a section $s$ in $\mathcal{F}(N\times \Delta^1_e)$ such that $\mathcal{F}(\delta_i)(s)=s_i$ for $i=0,1$. For all partial sheaves $\mathcal F$ considered in the paper, the general position argument shows that the set of concordance classes $\mathcal{F}[N]$ of sections 
in $\mathcal{F}(N)$ is isomorphic to the set of homotopy classes of maps $[N, |\mathcal{F}|]$. 

\begin{example}\label{ex:a1} Given $N\in \mathcal{X}$, let $\mathcal{F}(N)$ denote the set of all smooth maps $f$ of dimension $d$ to $N$. 
Given a map $u\co N'\to N$ transverse to $f$, let $\mathcal{F}(u)(f)$ denote the pullback of $f$. Then $\mathcal{F}$ is a partial sheaf. It is not a sheaf since $\mathcal{F}(u)(f)$ is not defined unless $u$ is transverse to $f$. 
\end{example}  

Similarly one defines partial $\mathcal{C}$-valued sheaves for categories $\mathcal{C}$ distinct from $\mathbf{Set}$.

\subsection{$\Gamma$-spaces}\label{asec:6}

Recall that the \emph{semi-simplicial category} $\Lambda$ is 
the category of finite sets $[n]=\{0,...,n\}$ and strictly monotone maps. In particular, every 
morphism in $\Lambda$ is a composition of face maps 
\[
   \delta_i\co [n]\mapsto [n+1] \quad\textrm{given by}\quad \{0,...,n\}\xrightarrow{\delta_i} \{0,...,\hat{i}, ..., n+1\}. 
\]
A \emph{semi-simplicial space} is a functor $X\co \Lambda^{op}\to \mathbf{Top}$. 
Its geometric realization is the quotient topological space 
\[
    |X|\co= \bigsqcup_{n\ge 0}\ X([n])\times \Delta^n\ /\ (d_ix, y)\sim (x, \delta_iy),
\]
where $d_i=X(\delta_i)$, and $\delta_i$ stands for 
the $i$-th face of a simplex. 
Every $\Gamma$-space has a canonical structure of a semi-simplicial space; it is defined
by means of the functor $\Lambda^{op} \to \Gamma^{op}$ that takes $[m]$ to $\mathbf{m}$
and $\delta_i$ to 
\[
    \{1,...,n\}\mapsto \Big\{\{2\}, ..., \{n+1\}\Big\}  \quad \textrm{if $i=0$},
\] 
\[
    \{1,...,n\}\mapsto \Big\{\{1\}, ..., \{i-1\}, \{i, i+1\}, \{i+2\}, ..., \{n+1\}\Big\}  \quad \textrm{if $i\ne 0, n+1$},    
\]
\[
    \{1,...,n\}\mapsto \Big\{\{1\}, ..., \{n\}\Big\}  \quad \textrm{if $i=n+1$}.
\]
In particular, there is a well-defined geometric realization of a $\Gamma$-space.

\begin{example} Every abelian topological group $A$ defines a $\Gamma$-space with $A(\mathbf{n})=A^n$
and with $A(\mathbf{m}\xrightarrow{\theta} \mathbf{n})$ given by $(a_1,...,a_n)\mapsto (b_1,..., b_m)$
where $b_i=\sum_{j\in \theta(i)} a_j$. The geometric realization of the $\Gamma$-space $A(\bullet)$ 
is a model for the classifying space of the abelian topological group $A$ provided that $(A, 0)$ is 
a well-pointed space, see \cite[Proposition A.1]{Se}.
\end{example}

\begin{remark} Segal explains in \cite{Se} that every $\Gamma$-space has a canonical 
structure of a simplicial space. He then works with the fat geometric realizations of simplicial sets of 
$\Gamma$-spaces. This essentially corresponds to our taking the semi-simplicial 
 geometric realizations $|A|$ of $\Gamma$-spaces. More precisely, Segal defines the \emph{fat realization} to be the telescope of inclusions $|A|_{n-1}\to |A|_n$  of the $(n-1)$-st skeleton of $A$ to the $n$-th one for $n>0$. Since $|A|$ is a deformation retract of the fat realization of $A$, the two spaces are homotopy equivalent. 
 We also note that $\mathbf{X}_{\bullet}$ in Theorem~\ref{th:-1}  is a 
simplicial \emph{set}, which is always \emph{good} in the terminology of Segal. 
\end{remark}

Every $\Gamma$-space $A(\bullet)$ determines a family of $\Gamma$-spaces
$R_k(\bullet)$ and $\Gamma$-spaces $L_n(\bullet)$, one 
$\Gamma$-space for each $k, n\ge 0$. These $\Gamma$-spaces consist of spaces  
$R_k(\mathbf{n})=L_n(\mathbf{k})=A(\mathbf{k}\times \mathbf{n})$
and structure morphisms
\[
    R_k(\mathbf{n}\to \mathbf{m}) = A(\mathbf{k}\times (\mathbf{n}\to \mathbf{m})), \quad
    L_n(\mathbf{k}\to \mathbf{l}) = A((\mathbf{k}\to \mathbf{l})\times \mathbf{n}),
\]
where the identity morphism of an object is identified with the object.  Since the diagrams
\[
\begin{CD}
 A(\mathbf{l}\times \mathbf{m}) @>L_m(\mathbf{k}\to \mathbf{l})>> A(\mathbf{k}\times \mathbf{m})\\
 @VR_l(\mathbf{n}\to\mathbf{m})VV @VVR_k(\mathbf{n}\to\mathbf{m})V\\
 A(\mathbf{l}\times \mathbf{n}) @>L_n(\mathbf{k}\to \mathbf{l})>> A(\mathbf{k}\times \mathbf{n})
\end{CD}
\]
commute, 
the geometric realizations $BA(\mathbf{k}) = |R_{k}(\bullet)|$ of $\Gamma$-spaces 
themselves form a $\Gamma$-space; 
the structure morphism $BA(\mathbf{k}\to \mathbf{l})$ is defined by means of morphisms 
\[
    R_l(\mathbf{n}) = A(\mathbf{l}\times \mathbf{n}) \xrightarrow{L_n(\mathbf{k}\to\mathbf{l})}
    A(\mathbf{k}\times \mathbf{n})= R_k(\mathbf{n}).
\]
It follows~\cite{Se} that the so-obtained family of classifying spaces forms a spectrum 
\[
 \quad A(\mathbf{1}), \ BA(\mathbf{1}), \ B^2A(\mathbf{1}), \ B^3A(\mathbf{1}), \dots.
\]

\subsection{$(B,f)$ structures}\label{a:3} Let us recall that the direct limit $\BO$ of spaces $\BO_n$ classifies stable vector bundles of a fixed dimension $d$ in the following sense. A \emph{stable vector bundle} of dimension $d$ over a topological space $X$ is represented by a pair $(\xi, \eta)$ of vector bundles over $X$ with $d=\dim \xi-\dim \eta$. The equivalence relation on representatives is generated by identifications of $(\xi,\eta)$ and $(\xi',\eta')$ for pairs with $\xi\oplus\eta'\cong\xi'\oplus \eta$. The stable vector bundles of pairs $(\xi, 0)$, $(0, \eta)$ and $(\xi, \eta)$ are also denoted by $\xi$, $-\eta$ and $\xi-\eta$ or $\xi\ominus\eta$ respectively. Stable vector bundles of dimension $d$ over a topological space $X$ are in bijective correspondence with homotopy classes of maps $f\co X\to \BO$. 

A \emph{normal $(B,f)$ structure} is a stable vector bundle $f\co B\to \BO$ of dimension $-d$ for some integer $d$. Each normal $(B,f)$ structure determines a cohomology theory. Indeed, let $f_{n+d}$ be the restriction of $f$ to $B_{n+d}=f^{-1}(\BO_{n})$ and $\xi_n$ denote the universal vector bundle over $\BO_n$. Let $T_n$ denote the one point compactification of the vector bundle $f_{n+d}^*\xi_n$ of dimension $n$ over $B_{n+d}$. Then the direct limit $\Omega^{\infty+d}\mathbf{B}$ of spaces $\Omega^{n+d}T_n$ is an infinite loop space, which determines a cohomology theory $\mathfrak{h}^*$. For a topological space $N$ and an integer $n$, the group 
\[
    \mathfrak{h}^n(N)=\lim_{i\to \infty}[X, \Omega^{i+d}T_{n+i}]
\] 
is called the $n$-th \emph{cohomology group} of $N$ with respect to the theory $\mathfrak{h}^*$ (e.g., see the book \cite{St}). Note that the homotopy groups of the spectrum
$\{\Omega^{\infty+d-i}\mathbf{B}\}$ that deloops $\Omega^{\infty+d}\mathbf{B}$ may not be trivial in negative degrees, i.e., the spectrum under consideration is not connective. There is another infinite delooping of $X=\Omega^{\infty+d}\mathbf{B}$ by a connective spectrum $\{B^iX\}$ where $B^i$ is the $i$-th classifying space of $X$. We say that $\{B^iX\}$ is the connective spectrum of the normal structure $(B, f)$, e.g., see \cite{Se}.

There is an involution $I$ on $\BO$ that takes the equivalence class of a pair $(\xi, \eta)$ onto that of $(\eta, \xi)$. The cohomology theory $\mathfrak{h}^*$ can be described not only by a normal $(B, f)$ structure, but also by a \emph{tangential $(B, f')$ structure} $f'=I\circ f$, which is a stable vector bundle $B\to \BO$ of dimension $d$; e.g., see \cite{St}. 

Given a tangential structure $f\co B\to \BO$, a \emph{tangential structure on a map $u\co M\to N$} of smooth manifolds  is a lift $M\to B$ with respect to $f$ of the map $M\to \BO$ classifying the stable tangent bundle $TM\ominus TN$ of $u$. The group $\mathfrak{h}^0(N)$ is the cobordism group of maps $M\to N$ of smooth manifolds $M$ equipped with a tangential $(B, f)$ structure.

\subsection{A sheaf theoretic version of the Brown representability theorem}


It follows from the Brown's construction~\cite{Br} that every natural transformation $F\to G$ of homotopy functors with connected pointed representing spaces is induced by a map $Y_F\to Y_G$ of representing spaces; and if the natural transformation is an isomorphism, then the map of representing spaces is a homotopy equivalence.  We will apply the Brown representability theorem in the case of set-valued sheaves $\mathcal{F}$ and $\mathcal{G}$. 

\begin{remark}
It is known that in general the statement of the Brown representability theorem is not true if the requirements that maps under consideration are pointed and $Y_H$ is connected are omitted~\cite{He}, \cite{He2}. 
\end{remark}

\begin{corollary} \label{c:10.4} Let $\mathcal{F}$ and $\mathcal{G}$ be two set valued sheaves on the category $\mathcal{X}$ of manifolds.  
Suppose that the classifying spaces of $\mathcal{F}$ and $\mathcal{G}$ are connected H-spaces. Let $\varphi_X\co \mathcal{F}[X]\to \mathcal{G}[X]$ be natural maps of sets for $X\in \mathcal{X}$. Then there is a map 
$|\mathcal{F}|\to |\mathcal{G}|$ inducing the natural transformation $\varphi$. Furthermore, if every map $\varphi_X$ is an isomorphism of sets, then the map $|\mathcal{F}|\to |\mathcal{G}|$ is a homotopy equivalence.  
\end{corollary}
\begin{proof} Recall that for any pointed CW-complexes $X, Y$, the fundamental group of $X$ acts on homotopy classes of pointed maps $Y\to X$. If $X$ is an H-space, then the action of the fundamental group $\pi_1X$ is trivial, see \cite[Example $4$A.$3$]{Hat}. Therefore in this case the set of homotopy classes of pointed maps $Y\to X$ is isomorphic to the set of homotopy classes of free maps $Y\to X$.  

Thus, under the hypotheses of Corollary~\ref{c:10.4}, there is a natural transformation defined by isomorphisms
\[
        F(X) \simeq [X, |\mathcal{F}|]_\bullet\simeq [\tilde{X}, |\mathcal{F}|] \simeq  \mathcal{F}[\tilde{X}] \longrightarrow \mathcal{G}[\tilde{X}] \simeq [\tilde{X}, |\mathcal{G}|] \simeq [X, |\mathcal{G}|]_\bullet \simeq G(X), 
\]
where $\tilde{X}$ is a manifold homotopy equivalent to the finite CW-complex $X$. Now Corollary~\ref{c:10.4} follows  from the original Brown theorem. 
\end{proof}

\end{document}